\documentclass[11pt]{article}
\usepackage[hscale=0.8,vscale=0.8]{geometry}

\usepackage{amsmath}
\usepackage{amsfonts}
\usepackage{latexsym}
\usepackage{graphicx}
\usepackage{multicol}
\usepackage{multirow}
\usepackage{amssymb}
\usepackage{mathtools}
\usepackage{hhline}
\usepackage{amsthm}
\usepackage{float}
\usepackage{cite}
\usepackage{color}
\usepackage{chngcntr}
\counterwithin{equation}{section}
\usepackage{latexsym}
\usepackage{graphicx}
\usepackage{caption}
\usepackage{subcaption}
\usepackage{comment}
\usepackage{dcolumn}
\usepackage{enumitem}

\def\h{{\bf h}}
\def\u{{\bf u}}
\def\s{{\bf s}}
\def\t{{\bf t}}
\def\w{{\bf w}}
\def\v{{\bf v}}
\def\r{{\bf r}}
\def\n{{\bf n}}
\def\f{{\bf f}}

\def\q{{\bf q}}
\def\x{{\bf x}}

\newcommand{\btau} { \mbox{\boldmath $\tau$} }
\def\bbeta{\mbox{{\boldmath$\eta$}}}
\def\bs{\mbox{{\boldmath$\sigma$}}}
\def\bS{\mbox{{\boldmath$\Sigma$}}}
\def\bt{\mbox{{\boldmath$\tau$}}}
\def\bxi{\mbox{{\boldmath$\xi$}}}

\def\U{{\bf U}}
\def\z{{\bf z}}
\def\V{{\bf V}}
\def\Q{{\bf Q}}

\def\D{{\bf D}}
\def\G{{\bf G}}
\def\X{{\bf X}}

\def\bI{{\bf I}}

\def\d{\partial}

\def\Ec{{{\mathcal{E}}}}
\def\Ac{{{\mathcal{A}}}}
\def\Bc{{{\mathcal{B}}}}

\def\Oc{{{\mathcal{O}}}}
\def\lam{\lambda}

\def\grad{\nabla}
\def\div{\grad\cdot}

\def\O{\Omega}
\def\Gam{\Gamma}

\def\<{\langle}
\def\>{\rangle}
\def\CO2{{CO$_{2}$}}

\allowdisplaybreaks

\newcommand\norm[1]{\left\lVert#1\right\rVert}

\newtheorem{lemma}{Lemma}[section]
\newtheorem{theorem}{Theorem}[section]
\newtheorem{corollary}{Corollary}[section]
\newtheorem{remark}{Remark}[section]

\begin{document}

\title{A nonlinear Stokes-Biot model for the interaction of a non-Newtonian fluid with poroelastic media}

\author{Ilona Ambartsumyan\thanks{Department of Mathematics, University of
Pittsburgh, Pittsburgh, PA 15260, USA;~{\tt \{ila6@pitt.edu, tqn4@pitt.edu, yotov@math.pitt.edu\}}; partially supported by NSF grants DMS 1418947 and DMS 1818775 and DOE grant DE-FG02-04ER25618.} \thanks{Institute for Computational Engineering and Sciences, The University of Texas at Austin, Austin, TX 78712, USA;
{\tt \{ailona@austin.utexas.edu\}}.}
\and
	Vincent J. Ervin\footnotemark[3]\thanks{Department of Mathematical Sciences, Clemson University, Clemson, SC 29634-0975, USA;~{\tt \{vjervin@clemson.edu\}}.}~\and
	Truong Nguyen\footnotemark[1]~\and
	Ivan Yotov\footnotemark[1]  }

\date{\today}

\maketitle

\begin{abstract}
We develop and analyze a model for the interaction of a
quasi-Newtonian free fluid with a poroelastic medium. The flow in the
fluid region is described by the nonlinear Stokes equations and in the
poroelastic medium by the nonlinear quasi-static Biot
model. Equilibrium and kinematic conditions are imposed on the
interface. We establish existence and uniqueness of a solution to the
weak formulation and its semidiscrete continuous-in-time finite
element approximation. We present error analysis, complemented by
numerical experiments.
\end{abstract}

\section{Introduction}
The interaction of a free fluid with a deformable porous medium is a
challenging multiphysics problem that has a wide range of
applications, including processes arising in gas and oil extraction
from naturally or hydraulically fractured reservoirs, designing
industrial filters, and blood-vessel interactions. The free fluid
region can be modeled by the Stokes or the Navier-Stokes equations,
while the flow through the deformable porous medium is modeled by the
quasi-static Biot system of poroelasticity \cite{Biot}. The two
regions are coupled via dynamic and kinematic interface conditions,
including balance of forces, continuity of normal velocity, and a no
slip or slip with friction tangential velocity condition. These
multiphysics models exhibit features of coupled Stokes-Darcy flows and
fluid-structure interaction (FSI). There is extensive literature on
modeling these separate couplings, see e.g. \cite{DMQ,LSY,RivYot} for
Stokes-Darcy flows and
\cite{FSI-Radon,formaggia2010cardiovascular,galdi2010fundamental}
for FSI. More recently there has been growing interest in modeling
Stokes-Biot couplings, which can be referred to as fluid-poroelastic
structure interaction (FPSI).  The well-posedness of the mathematical
model is studied in \cite{Showalter-Biot-Stokes}.  A variational
multiscale stabilized finite element method for the Navier-Stokes-Biot
problem is developed in \cite{Badia-Quaini-Quart}. In \cite{BYZ} a
non-iterative operator-splitting method is developed for the
Navier-Stokes-Biot model with pressure Darcy formulation. The well
posedness of a related model is studied in \cite{Cesmelioglu}. The
Stokes-Biot problem with a mixed Darcy formulation is studied in
\cite{Bukac-Nitsche} and \cite{AKYZ-LM} using Nitsche's method and a
Lagrange multiplier, respectively, to impose the continuity of normal
velocity on the interface. An optimization-based iterative algorithm
with Neumann control is proposed in \cite{Cesmelioglu-etal}. A
reduced-dimension fracture model coupling Biot and an averaged
Brinkman equation is developed in
\cite{Buk-Yot-Zun-fracture}. Alternative fracture models using the
Reynolds lubrication equation coupled with Biot have also been
studied, see e.g.  \cite{Girault-2015}.

All of the above mentioned works are based on Newtonian fluids. In
this paper we develop FPSI with non-Newtonian fluids, which, to the
best of our knowledge, has not been studied in the literature. In many
applications the fluid exhibits properties that cannot be captured by
a Newtonian fluid assumption. For instance, during water flooding in
oil extraction, polymeric solutions are often added to the aqueous
phase to increase its viscosity, resulting in a more stable
displacement of oil by the injected water
\cite{lopez2003predictive}. In hydraulic fracturing, proppant
particles are mixed with polymers to maintain high permeability of the
fractured media \cite{Lee-proppant}. In blood flow simulations of
small vessels or for patients with a cardiovascular disease, where the
arterial geometry has been altered to include regions of
re-circulation, one needs to consider models that can capture the
sheer-thinning property of the blood \cite{janela20103d}.

In this work we use nonlinear Stokes equations to model the free fluid
in the flow region and a nonlinear Biot model for the fluid in the
poroelastic region. Our model is built on the nonlinear Stokes-Darcy
model presented in \cite{ervin2009coupled} and the linear Stokes-Biot
model considered in \cite{AKYZ-LM}. Our Biot model is based on a
linear stress-strain constitutive relationship and a nonlinear Darcy
flow. The coupling conditions between the two subdomains include mass
conservation, conservation of momentum and the Beavers-Joseph-Saffman
slip with friction condition. We focus on fluids that possess the
sheer thinning property, i.e., the viscosity decreases under shear
strain, which is typical for polymer solutions and blood. Viscosity
models for such non-Newtonian fluids include the Power law, the Cross
model and the Carreau model
\cite{bird1977dynamics,chow2003numerical,owens2002computational,lopez2003predictive,pearson2002models
}. The Power law model is popular because it only contains two
parameters, and it is possible to derive analytical solutions in
various flow conditions \cite{bird1977dynamics}. On the other hand, it
implies that in the flow region the viscosity goes to infinity if the
deformation goes to zero, which may not be representative in certain
applications. The Cross and Carreau models have been deduced
empirically as alternatives of the Power law model. They have three
parameters, and in some parameter regimes, the viscosity is strictly
greater than zero and bounded. We assume that the viscosity in each
subdomain satisfies one such model, with dependence on the magnitude
of the deformation tensor and the magnitude of Darcy velocity in the
fluid and poroelastic regions, respectively. We further assume that
along the interface the fluid viscosity is a function of the fluid and
structure interface velocities. We consider both unbounded and bounded
parameter regimes. In the former case, the analysis is done in an
appropriate Sobolev space setting, using spaces such as $W^{1,r}$,
where $1 < r < 2$ is the viscosity shear thinning parameter. In the
latter case, the analysis reduces to the Hilbert space
setting. Nonlinear Stokes-Darcy models with bounded viscosity have
been studied in
\cite{ervin2011coupling,Gatica-etal-SD-nonlin,Gatica-etal-NSE-Darcy-nonlin},
while the unbounded case is considered in \cite{ervin2009coupled}.

Following the approach in \cite{AKYZ-LM}, we enforce the continuity of
normal velocity on the interface through the use of a Lagrange
multiplier. The resulting weak formulation is a nonlinear
time-dependent system, which is difficult to analyze, due to to the
presence of the time derivative of the displacement in some
non-coercive terms. We consider an alternative mixed elasticity
formulation with the structure velocity and elastic stress as primary
variables, see also \cite{Showalter-Biot-Stokes}. In this case we
obtain a system with a degenerate evolution in time operator and a
nonlinear saddle-point type spatial operator. The structure of the
problem is similar to the one analyzed in \cite{Showalter-SIAMMA}, see
also \cite{Boffi-Gastaldi} in the linear case. However, the analysis
in \cite{Showalter-SIAMMA} is restricted to the Hilbert space setting
and needs to be extended to the Sobolev space setting. Furthermore,
the analysis in \cite{Showalter-SIAMMA} is for monotone operators, see
\cite{showalter2013monotone}, and as a result requires certain right
hand side terms to be zero, while in typical applications these terms
may not be zero. Here we explore the coercivity of the operators to
reformulate the problem as a parabolic-type system for the pressure
and stress in the poroelastic region. We show well posedness for this
system for general source terms and that the solution satisfies the
original formulation. We also prove that the solution to the original
formulation is unique and provide a stability bound. We then consider
a semidiscrete finite element approximation of the system and carry
out stability and error analysis. For this purpose we establish a
discrete inf-sup condition, which involves a non-conforming Lagrange
multiplier discretization that allows for non-matching grids across
the Stokes-Biot interface.

The rest of the paper is organized as follows. In Section~2 we
introduce the governing equations. Section~3 is devoted to the weak
formulation, upon which we base the numerical method, and an
alternative formulation, which is needed for the purpose of the
analysis. In Section~4 we prove the well-posedness of the alternative
and original formulations. The semidiscrete approximation and its
well-posedness analysis are developed in Section~5. The error analysis
is carried out in Section~6. Numerical experiments are presented in 
Section~7.

\section{Problem set-up}

Let $\O \subset \mathbb{R}^d,\,d=2,3$ be a Lipschitz domain, which is
subdivided into two non-overlapping and possibly non-connected
regions: fluid region $\O_f$ and poroelastic region $\O_p$.  Let
$\partial\O_f \cap \partial \O_p = \Gam_{fp}$ denote the (nonempty)
interface between these regions and let $\Gam_f = \partial \O_f
\setminus \Gam_{fp}$ and $\Gam_p = \partial \O_p \setminus \Gam_{fp}$
denote the external parts of the boundary $\partial \O$. We denote by
$\n_f$ and $\n_p$ he unit normal vectors which point outward from
$\partial \O_f$ and $\partial \O_p$, respectively, noting that $\n_f=
-\n_p$ on $\Gam_{fp}$. Let $(\u_\star,p_\star)$ be the
velocity-pressure pairs in $\O_\star$, $\star = f$, $p$, and let
$\bbeta_p$ be the displacement in $\O_p$.

We assume that the flow in $\O_f$ is governed by the nonlinear generalized Stokes equations with homogeneous boundary conditions on $\Gam_f$: 
\begin{equation}\label{stokes}
	-\div \bs_f(\u_f,p_f)  = \f_f,
	\quad \div \u_f =  q_f   \quad \mbox{in } \O_f \times (0,T], 
\quad \u_f = {\bf 0}  \quad \mbox{on } \Gam_f \times (0,T],
\end{equation}
where $\D(\u_f)$ and $\bs_f(\u_f,p_f)$ denote the deformation and the stress tensors, respectively:
$$
\D(\u_f) = \frac 12 (\grad \u_f + \grad\u_f^T), \quad
\bs_f(\u_f,p_f) = -p_f \bI + 2\nu(\D(\u_f))\D(\u_f),
$$
where $\bI$ stands for the identity operator. We consider a
generalized Newtonian fluid with the viscosity $\nu$ dependent on the
magnitude of the deformation tensor, in particular shear-thinning
fluids with $\nu$ a decreasing function of $|\D(\u_f)| $. We consider the
following models \cite{chow2003numerical, owens2002computational},
where $1 < r < 2$, $0 \le \nu_\infty < \nu_0$, and $K_f > 0$ are constants:

{\em Carreau model:}
\begin{equation} \label{stokes-vis1}
	\nu(\D(\u_f)) = \nu_ \infty + (\nu_0 - \nu_\infty)/ 
( 1 + K_f|\D(\u_f)|^2)^{(2-r)/2},
\end{equation}

{\em Cross model:}
\begin{equation} \label{stokes-vis2}
	\nu(\D(\u_f)) = \nu_ \infty + (\nu_0 - \nu_\infty)/ ( 1 + K_f|\D(\u_f)|^{2-r}),
\end{equation}

{\em Power law model:}
\begin{equation} \label{stokes-vis3}
	\nu(\D(\u_f)) = K_f|\D(\u_f)|^{r-2}.
\end{equation}

In turn, in $\O_p$ we consider the quasi-static Biot system \cite{Biot}
\begin{align}
& - \div \bs_p(\bbeta_p,p_p) = \f_p \quad \mbox{in } \O_p \times (0,T] , 
\label{biot1}\\ 
& \nu_{eff}(\u_p) \kappa^{-1}\u_p + \grad p_p = 0,  \quad 
\frac{\d}{\d t} (s_0 p_p + \alpha_p \div \bbeta_p) + \div \u_p = q_p 
\quad \mbox{in } \O_p \times (0,T] , \label{biot2} \\
& \u_p \cdot \n_p = 0 \, \mbox{ on } \, \Gam_p^N\times (0,T] ,  
\quad p_p =0 \, \mbox{ on }  \Gam_p^D\times (0,T], \quad \bbeta_p = {\bf 0} 
\mbox{ on } \Gam_p
\times (0,T], \label{biot3}
\end{align}
where $\bs_e(\bbeta_p)$ and $\bs_p(\bbeta_p,p_p)$ are the elasticity and 
poroelasticity stress tensors, respectively,
\begin{equation}\label{stress-defn}
	\bs_e(\bbeta_p) = \lambda_p (\div \bbeta_p) \bI + 2\mu_p \D(\bbeta_p), 
\quad
	\bs_p(\bbeta_p,p_p) = \bs_e(\bbeta_p) - \alpha_p  p_p \bI,
\end{equation}
$\alpha_p$ is the Biot-Willis constant, $\lambda_p$, $\mu_p$ are
the Lam\`{e} coefficients, $s_0 > 0$ is a storage coefficient, $\kappa$ is a scalar
uniformly positive and bounded permeability function, and $\Gam_p = \Gam_p^N\cup\Gam_p^D$. 
To avoid the issue with restricting the mean value of the pressure, we assume 
that $|\Gamma_p^D| > 0$. We further assume that 
$\mbox{dist}(\Gamma_p^D,\Gamma_{fp}) \ge s > 0$.
We note that even though the analysis of our formulation is valid for
a symmetric and positive definite permeability tensor, we restrict it 
to $\kappa \bI$, due to assumptions made in the derivations of some of the
viscosity functions suitable for modeling non-Newtonian flow in porous
media. In particular, we consider the following two models for the
effective viscosity $\nu_{eff}$ in $\Omega_p$ 
\cite{lopez2003predictive, pearson2002models},
where $1 < r < 2$, $0 \le \nu_\infty < \nu_0$, and $K_p > 0$ are constants:

{\em Cross model:}
\begin{equation} \label{darcy-vis1}
	\nu_{eff}(\u_p) = \nu_ \infty + (\nu_0 - \nu_\infty)/ ( 1 + K_p|\u_p|^{2-r}),
\end{equation}

{\em Power law model:}
\begin{equation} \label{darcy-vis2}
	\nu_{eff}(\u_p) = K_p(|\u_p|/(\sqrt{\kappa}m_c))^{r-2},
\end{equation}
where $m_c$ is a constant that depends on the internal structure of the 
porous media. 

Following \cite{Showalter-Biot-Stokes,Badia-Quaini-Quart}, the {\it
  interface conditions} on the fluid-poroelasticity interface
$\Gamma_{fp}$, are {\it mass conservation}, {\it balance of normal
  stress}, the Beavers-Joseph-Saffman (BJS) {\it slip with friction}
condition \cite{BJ,Saffman}, and {\it conservation of momentum}: 
\begin{align}
& \u_f\cdot\n_f + \left(\frac{\d \bbeta_p}{\d t} + \u_p\right)\cdot\n_p = 0 
\quad \mbox{on } \Gam_{fp}, 
\label{eq:mass-conservation} \\
& - (\bs_f \n_f)\cdot\n_f =  p_p \quad \mbox{on } \Gam_{fp}, \label{Gamma-fp-01} \\ 
& - (\bs_f\n_f)\cdot\t_{f,j} = \nu_I \, \alpha_{BJS} \sqrt{\kappa^{-1}}
\left(\u_f - \frac{\d \bbeta_p}{\d t}\right)\cdot\t_{f,j} \quad \mbox{on } 
\Gam_{fp}, \label{Gamma-fp-1} \\
& \bs_f\n_f= - \bs_p\n_p \quad \mbox{on } \Gam_{fp}, \label{Gamma-fp-2}
\end{align}
where $\t_{f,j}$, $1 \le j \le d-1$, is an orthogonal system of unit
tangent vectors on $\Gamma_{fp}$ and $\alpha_{BJS} \ge 0$ is an
experimentally determined friction coefficient.  We note that the
continuity of flux takes into account the normal velocity of the solid
skeleton, while the BJS condition accounts for its tangential
velocity. We assume that along the interface the fluid viscosity
$\nu_I$ is a function of the magnitude of the tangential component 
of the slip velocity 
$\left| \sum_{j=1}^{d-1}((\u_f - \d_t \bbeta_p)\cdot \t_{f,j})\t_{f,j}\right|$
given by the Cross model \eqref{darcy-vis1} or the Power law model 
\eqref{darcy-vis2}, where $\d_t \phi: = \partial \phi/\partial t$.
For the rest of the paper we will write $\nu$, $\nu_{eff}$ or $\nu_I$ 
keeping in mind that these are nonlinear functions as defined above.

The above system of equations is complemented by a set of initial conditions:
$$
p_p(0,\mathbf{x}) = p_{p,0}(\mathbf{x}), \,\, 
\bbeta_p(0,\mathbf{x}) = \bbeta_{p,0}(\mathbf{x}) \mbox{ in } \O_p.
$$

In the following, we make use of the usual notation for Lebesgue
spaces $L^p(\Omega)$, Sobolev spaces $W^{k, p}(\Omega)$ and Hilbert
spaces $H^k(\Omega)$. For a set $\Oc \subset \mathbb{R}^d$, the $L^2(\Oc)$
inner product is denoted by $(\cdot,\cdot)_{\Oc}$ for scalar, vector and
tensor valued functions. For a section of a subdomain boundary $S$ we
write $\langle \cdot, \cdot \rangle_{S}$ for the $L^2(S)$ inner
product (or duality pairing).  We also denote by $C$ a generic
positive constant independent of the discretization parameters.

Adopting the approach from \cite{ervin2009coupled, ervin2011coupling},
we assume that the viscosity functions satisfy one of the two sets of
assumptions \eqref{A1}--\eqref{A2} or \eqref{B1}--\eqref{B2}
below. Let $g(\x):\mathbb{R}^d \rightarrow\mathbb{R}^+ \cup \{0\}$ and
let $\G(\x):\mathbb{R}^d\rightarrow\mathbb{R}^d$ be given by $\G(\x)=
g(\x)\x$. For $\x,\h \in \mathbb{R}^d$, let $\G(\x)$ satisfy,
for constants $C_1, \ldots, C_4 >0$ and $c\geq 0$,
\begin{align}
 \left( \G(\x + \h)-\G(\x) \right) \cdot  \h &\geq C_1 |\h|^2, 
\tag{A1} \label{A1}\\
|\G(\x + \h) - \G(\x)| &\leq C_2 |\h|,  \tag{A2} \label{A2} 
\end{align}
or
\begin{align}
(\G(\x+\h)-\G(\x))\cdot \h & \geq C_3 \frac{|\h|^2}{c+|\x|^{2-r}+|\x+\h|^{2-r}}, 
\tag{B1} \label{B1}\\
|\G(\x+\h) -\G(\x)| & \leq C_4 \frac{|\h|}{c+|\x|^{2-r}+|\x+\h|^{2-r}}, \tag{B2} \label{B2}
\end{align}
with the convention that $\G(\x)={\bf 0}$ if $\x ={\bf 0}$, and
$|\h|/(c+|\x|+|\h|) =0$ if $c=0$ and $\x=\h={\bf 0}$. From
\eqref{B1}--\eqref{B2} it follows that there exist constants $C_5, C_6
> 0$ such that for $\s,\t,\w \in (L^r(\Oc))^d$ \cite{Sandri}
\begin{gather}
(\G(\s)-\G(\t),\s -\t)_{\Oc} \geq C_5
\left((|\G(\s)-\G(\t)|,|\s-\t|)_{\Oc} +
\frac{\|\s-\t\|^2_{L^r(\Oc)}}{c+\|\s\|^{2-r}_{L^r(\Oc)}+\|\t\|^{2-r}_{L^r(\Oc)}}
\right), \label{monotonicity}\\
(\G(\s)-\G(\t),\w)_{\Oc} \leq C_6 
\left\|\frac{|\s-\t|}{c+|\s|+|\t|} \right\|^{\frac{2-r}{r}}_{L^{\infty}(\Oc)} 
(|\G(\s)-\G(\t)|, |\s-\t|)_{\Oc}^{1/r'}\|\w\|_{L^r(\Oc)}. 
\label{continuity}
\end{gather}
\begin{remark}
It is shown in \cite{Gatica-etal-SD-nonlin} that conditions
\eqref{A1}--\eqref{A2} are satisfied for $g(\D(\u_f))=\nu(\D(\u_f))$
given in the Carreau model \eqref{stokes-vis1} with $\nu_{\infty}>0$, in
which case $\nu_\infty \leq g(\x) \leq \nu_0$.  A similar argument can
be applied to show that \eqref{A1}--\eqref{A2} hold for the Cross
model, with $g(\D(\u_f))=\nu(\D(\u_f))$ given in \eqref{stokes-vis2} for
Stokes and $g(\u_p)=\nu_{eff}(\u_p)$ given in \eqref{darcy-vis1} for Darcy,
in the case of $\nu_{\infty}>0$.  Furthermore, it is shown in
\cite{Sandri} that conditions \eqref{B1}--\eqref{B2} with $c > 0$ hold
in the case of the Carreau model \eqref{stokes-vis1} with
$\nu_{\infty} = 0$, and that conditions \eqref{B1}--\eqref{B2} with $c
= 0$ hold for the Power law model \eqref{stokes-vis3} and
\eqref{darcy-vis2}.
\end{remark}

\section{Variational formulation}\label{sec:var}

We will consider two cases when defining the functional spaces,
depending on which set of assumptions holds. In the case
\eqref{B1}--\eqref{B2}, we consider Sobolev spaces. For a given $r >
1$ let $r'$ be its conjugate, satisfying $r^{-1}+(r')^{-1}=1$. Let
\begin{align}
	\V_{f} &= \{ \v_f \in (W^{1,r}(\O_{f}))^d : \v_f = {\bf 0} \text{ on } \Gam_{f} \}, && W_{f} = L^{r'}(\O_{f}), \label{spaces1}
\end{align}
with the corresponding norms
\begin{align*}
\|\v_f\|_{\V_f} = \|\v_f\|_{(W^{1,r}(\O_f))^d},
\quad \|w_f\|_{W_f} = \|w_f\|_{L^{r'}(\O_f)}. && 
\end{align*}
With $L^r(\text{div};\O_p) = \{\v_p \in (L^r(\O_p))^d : \nabla \cdot \v_p \in L^r(\O_p)\}$,
let 
\begin{align}
\V_{p} &= \{ \v_p \in L^r(\text{div}; \O_{p}) : \v_p \cdot \n_p = 0 \text{ on }\Gam^N_{p}  \},&& W_{p} = L^{r'}(\O_{p}), \nonumber\\
\X_{p} &= \{ \bxi_p \in (H^{1}(\O_{p}))^d : \bxi_p= {\bf 0} \text{ on } \Gam_{p} \}. \label{spaces2}
\end{align}
with norms
\begin{align*}
	\|\v_p\|^r_{\V_p} & = \|\v_p\|^r_{(L^{r}(\O_p))^d} 
+ \|\nabla \cdot \v_p\|^r_{L^{r}(\O_p)}, && \|w_p\|_{W_p} = \|w_p\|_{L^{r'}(\O_p)}, &&  \\
	\|\bbeta_p\|_{\X_p} & = \|\bbeta_p\|_{(H^{1}(\O_p))^d}. 
\end{align*}
In the case of \eqref{A1}--\eqref{A2}, we consider Hilbert spaces, with 
the above definitions replaced by
\begin{align}
\V_{f} &= \{ \v_f \in (H^{1}(\O_{f}))^d : \v_f = {\bf 0} \text{ on } \Gam_{f} \}, && W_{f} = L^{2}(\O_{f}), \\
\V_{p} &= \{ \v_p \in H(\text{div}; \O_{p}) : \v_p \cdot \n_p = 0 \text{ on }\Gam^N_{p}  \},&& W_{p} = L^{2}(\O_{p}).
\end{align}
The global spaces are products of the subdomain spaces. For simplicity
we assume that each region consists of a single subdomain.

\begin{remark}
For simplicity of the presentation, for the rest of the paper we focus on the case 
\eqref{B1}--\eqref{B2}, which is the technically more challenging case. The arguments
apply directly to the case \eqref{A1}--\eqref{A2}. 
\end{remark}

\subsection{Lagrange multiplier formulation}
To derive the weak formulation, we multiply
\eqref{stokes}, \eqref{biot1}--\eqref{biot2} by appropriate test functions and
integrate each equation over the corresponding region, utilizing
the boundary and interface conditions
\eqref{eq:mass-conservation}--\eqref{Gamma-fp-2}. 
Integration by parts in the first equation in \eqref{stokes},
\eqref{biot1}, and the first equation in \eqref{biot2} leads to 
the Stokes, Darcy and the elasticity functionals
\begin{align*}
	a_f(\cdot , \cdot) \, : \, \V_{f} \times \V_{f} \, \longrightarrow  \, \mathbb{R} , \quad \quad a_f(\u_f,\v_f) &:= (2\nu \D(\u_f),\D(\v_f))_{\O_f}, \\
	a^d_p(\cdot , \cdot) \, : \, \V_{p} \times \V_{p} \, \longrightarrow  \, \mathbb{R} , \quad \quad a^d_p(\u_p,\v_p) &:= (\nu_{eff}\kappa^{-1}\u_p,\v_p)_{\O_p}, \\
	a^e_p(\cdot , \cdot) \, : \, \X_{p} \times \X_{p} \, \longrightarrow  \, \mathbb{R} , \quad \quad a^e_p(\bbeta_p,\bxi_p)&:= (2\mu_p\D(\bbeta_p),\D(\bxi_p))_{\O_p} + (\lam_p \nabla \cdot \bbeta_p,\nabla \cdot \bxi_p)_{\O_p},
\end{align*}
the bilinear forms
\begin{align*}
b_{\star}(\cdot , \cdot) \, : \, \V_{\star} \times W_{\star} \, \longrightarrow  \, \mathbb{R} , 
\quad \quad b_{\star}(\v,w) := -(\nabla \cdot \v, w)_{\O_{\star}}, \quad \star = f,p,
\end{align*}
and the interface term
$$
I_{\Gamma_{fp}} = - \<\bs_f\n_f,\v_f\>_{\Gamma_{fp}} - \<\bs_p\n_p,\bxi_p\>_{\Gamma_{fp}}
+ \<p_p,\v_p\cdot\n_p\>_{\Gamma_{fp}}.
$$
This term is incorporated into the weak formulation by introducing a Lagrange multiplier which has a meaning of normal stress/Darcy pressure on the interface:
\begin{align*}
	\lam = -(\bs_f\n_f)\cdot \n_f = p_p, \qquad \mbox{on } \Gam_{fp}.
\end{align*}
With $\lam$ introduced, we have, using \eqref{Gamma-fp-01},
\eqref{Gamma-fp-1} and \eqref{Gamma-fp-2},
\begin{align*}
	I_{\Gamma_{fp}} = a_{BJS}(\u_{f},\d_t\bbeta_{p};\v_{f},\bxi_{p}) + 
	b_{\Gamma}(\v_{f},\v_{p},\bxi_{p};\lam),
\end{align*}
where
\begin{align*}
	a_{BJS}(\u_f,\bbeta_p;\v_f,\bxi_p) &= 
	\sum_{j=1}^{d-1} 
	\big\<\nu_{I} \, \alpha_{BJS}\sqrt{\kappa^{-1}}(\u_f - \bbeta_p)
	\cdot\t_{f,j},(\v_f - \bxi_p)\cdot\t_{f,j} \big\>_{\Gam_{fp}}, \\
	b_{\Gamma}(\v_f,\v_p,\bxi_p;\mu) &= \<\v_f\cdot\n_f 
	+ (\bxi_p + \v_p)\cdot\n_p,\mu\>_{\Gam_{fp}}.
\end{align*}
For the term $b_{\Gamma}(\v_f,\v_p,\bxi_p; \lam)$ to be well-defined,
we choose the Lagrange multiplier space as $\Lambda =
W^{1/r,r'}(\Gam_{fp})$. It is shown in \cite{ervin2009coupled} that in
the case $\mbox{dist}(\Gamma_p^D,\Gamma_{fp}) \ge s > 0$, if 
$\v_p \in L^r(\text{div}; \O_p)$, then $\v_p\cdot\n_p|_{\Gamma_{fp}}$ can be 
identified with a functional in $W^{-1/r,r}(\Gamma_{fp})$. Furthermore,
for  $\v_f \in W^{1,r}(\O_f)$, $\v_f\cdot\n_f \in W^{1/r',r}(\d\O_f)$,
and for $\bxi_p \in H^1(\O_p)\subset W^{1,r}(\O_p)$, 
$\bxi_p\cdot\n_p \in W^{1/r',r}(\d\O_p)$. Therefore, with 
$\mu \in W^{1/r,r'}(\Gamma_{fp})$, the integrals in 
$b_{\Gamma}(\v_f,\v_p,\bxi_p; \lam)$ are well-defined.
 
The variational formulation reads: \textit{given $\f_{f} \in
  W^{1,1}(0, T; \V_{f}')$, $\f_{p} \in W^{1,1}(0, T; \X_{p}')$, $q_f
  \in W^{1,1}(0, T; W_{f}')$, $q_p \in W^{1,1}(0, T; W_{p}')$, and
  $p_p(0) = p_{p,0} \in W_p$, $\bbeta_p(0) = \bbeta_{p,0} \in \X_p$,
  find, for $t \in (0,T]$, $(\u_f(t), p_f(t),\u_p(t),p_p(t),$
$\bbeta_p(t),\lam(t))\in L^{\infty}(0,T;\V_f)\times
L^{\infty}(0,T;W_f)\times L^{\infty}(0,T;\V_p) \times
W^{1,\infty}(0,T;W_p) \times W^{1,\infty}(0,T;\X_p) $ $\times
L^{\infty}(0,T;\Lambda)$, such that for all $\v_f \in \V_f$, $w_f \in
W_f$, $\v_p \in \V_p$, $w_p \in W_p$, $\bxi_p \in \X_p$, and $\mu \in
\Lambda$,}
\begin{align}
	&  
	a_{f}(\u_{f},\v_{f}) + a^d_{p}(\u_{p},\v_{p}) + a^e_{p}(\bbeta_{p},\bxi_{p})
	+ a_{BJS}(\u_{f},\d_t \bbeta_{p};\v_{f},\bxi_{p}) + b_f(\v_{f},p_{f})+ b_p(\v_{p},p_{p})\nonumber
	\\
	& \qquad\quad   + 
	\alpha_p b_p(\bxi_{p},p_{p}) + b_{\Gamma}(\v_{f},\v_{p},\bxi_{p};\lam)  = (\f_{f},\v_{f})_{\O_f} + (\f_{p},\bxi_{p})_{\O_p}, \label{h-cts-1}\\
	& \left( s_0 \d_t p_{p},w_{p}\right)_{\O_p} 
	- \alpha_p b_p\left(\d_t\bbeta_{p},w_{p}\right) - b_p(\u_{p},w_{p}) - b_f(\u_{f},w_{f}) 
	= (q_{f},w_{f})_{\O_f} + (q_{p},w_{p})_{\O_p}, \label{h-cts-2} \\
	& b_{\Gamma}\left(\u_{f},\u_{p},\d_t\bbeta_{p};\mu\right) = 0. \label{h-cts-gamma}
\end{align}
Note that $(s_0 \d_t p_p,w_p)_{\O_p}$ is well-defined, since for
$r<2$, we have that $r'>2$ and $L^{r'}(\O_p) \subset L^2(\O_p)$.

Although related models have been analyzed previously, e.g. the
non-Newtonian Stokes-Darcy model was investigated in
\cite{ervin2009coupled} and the Newtonian dynamic Stokes-Biot model
was studied in \cite{Showalter-Biot-Stokes}, the well posedness of
\eqref{h-cts-1}--\eqref{h-cts-gamma} has not been established in the
literature. Analyzing this formulation directly is difficult, due to
the presence of $\d_t\bbeta_p$ in several non-coercive terms.
Instead, we analyze an alternative formulation and show that the two
formulations are equivalent.

\subsection{Alternative formulation}
Our goal is to obtain a system of evolutionary saddle point type,
which fits the general framework studied in
\cite{Showalter-SIAMMA}. Following the approach from
\cite{Showalter-Biot-Stokes}, we do this by considering a mixed
elasticity formulation with the structure velocity and elastic stress
as primary variables. Recall that the elasticity stress tensor $\bs_e$
is connected to the displacement $\bbeta_p$ through the
relation \cite{brezzi2008mixed}:
\begin{align}\label{elasticity-compliance}
	A\bs_e = \D(\bbeta_p),
\end{align}
where $A$ is a symmetric and positive definite compliance tensor. In the 
isotropic case $A$ has the form
\begin{align}
	A\bs_e = \frac{1}{2\mu_p}\left(\bs_e - \frac{\lambda_p}{2\mu_p+d\lambda_p}\text{tr}(\bs_e)\bI\right) , 
	\ \ \mbox{ with } \ 
	A^{-1} \bs_e  = 2\mu_p  \, \bs_e \, + \,  \lambda_p \text{tr}(\bs_e)\bI .
	\label{defAten}
\end{align}
The space for the elastic stress is $\bS_{e} = (L^{2}_{sym}(\O_p))^{d\times d}$ with the norm
$
	\|\bs_e\|^{2}_{\bS_e} :=  \sum_{i,j=1}^d\|(\bs_e)_{i,j}\|^{2}_{L^{2}(\O_p)}
$.

The derivation of the alternative variational formulation differs from
the original one in the way the equilibrium equation \eqref{biot1} is handled. As before,
we multiply it by a test function $\v_s \in \X_p$ and integrate by parts. However, instead
of using the constitutive relation of the first equation in \eqref{stress-defn}, we use only
the second equation in \eqref{stress-defn}, resulting in 
\begin{align*}
	\int_{\O_p}\left(\bs_e:\D(\v_s) - \alpha_p p_p \div \v_s\right) \, d\x
	- \int_{\Gamma_{fp}} \bs_p \n_p\cdot \v_s \, ds = \int_{\O_p}\f_p \cdot \v_s\, d\x. 
\end{align*}
We eliminate the displacement $\bbeta_p$ from the system by differentiating 
\eqref{elasticity-compliance} in time and introducing a new variable 
$\u_s := \d_t \bbeta_p \in \X_p$, which has a meaning of structure velocity. Multiplication
by a test function $\btau_e \in \bS_e$ gives
\begin{align*}
\int_{\O_p}\left( A\d_t\bs_e:\btau_e -\D(\u_s):\btau_e \right)\, d\x = 0. 
\end{align*}
The rest of the equations are handled in the same way as in the original
weak formulation, resulting in the same Stokes and Darcy functionals,
$a_f(\u_f,\v_f)$ and $a^d_p(\u_p,\v_p)$, respectively, and the same
interface term $I_{\Gamma_{fp}}$. Defining the bilinear forms
$b_{s}(\cdot ,\cdot )  :  \X_p \times \bS_e  \longrightarrow
\mathbb{R}$ and $a^s_p(\cdot ,\cdot )  :  \bS_e \times \bS_e
 \longrightarrow \mathbb{R}$,
\begin{align*}
	b_{s}(\v_s,\bt_e)  :=  (\D(\v_s), \bt_e)_{\O_p}, 
\qquad a^s_p(\bs_e,\bt_e)  :=   (A\bs_e,\bt_e)_{\O_p},
\end{align*}
we obtain the following weak formulation: \textit{given 
$\f_{f} \in
  W^{1,1}(0, T; \V_{f}')$, $\f_{p} \in W^{1,1}(0, T; \X_{p}')$, $q_f
  \in W^{1,1}(0, T; W_{f}')$, $q_p \in W^{1,1}(0, T; W_{p}')$, and
  $p_p(0) = p_{p,0} \in W_p$, 
$\bs_e(0) = A^{-1}\D(\bbeta_{p,0}) \in \bS_e$, for $t
  \in (0,T]$, find $(\u_f(t), p_f(t),\u_p(t),p_p(t),\u_s(t),
\bs_e(t),\lam(t))\, \in$ $ L^{\infty}(0,T;\V_f)\times
L^{\infty}(0,T;W_f)\times L^{\infty}(0,T;\V_p) \times
W^{1,\infty}(0,T;W_p) \times L^{\infty}(0,T;\X_p) \times
W^{1,\infty}(0,T;\bS_e) \times L^{\infty}(0,T;\Lambda)$, such that
for all $\v_f \in \V_f$, $w_f \in W_f$, $\v_p \in \V_p$, $w_p \in
W_p$, $\v_s \in \X_p$, $\btau_e \in \bS_e$, 
$\mu \in \Lambda$,}
\begin{align}
	&  
	a_{f}(\u_{f},\v_{f}) + a^d_{p}(\u_{p},\v_{p}) 
	+ a_{BJS}(\u_{f}, \u_{s};\v_{f},\v_{s}) + b_f(\v_{f},p_{f})+ b_p(\v_{p},p_{p})\nonumber
	\\
	& \qquad\quad   + 
	\alpha_p b_p(\v_{s},p_{p}) +b_{s}(\v_s,\bs_e)+ b_{\Gamma}(\v_{f},\v_{p},\v_{s};\lam)  = (\f_{f},\v_{f})_{\O_f} + (\f_{p},\v_{s})_{\O_p}, \label{h-cts-alt-1}\\
	& \left( s_0 \d_t p_{p},w_{p}\right)_{\O_p} + a^s_{p}(\d_t\bs_{e},\bt_{e})
	- \alpha_p b_p\left(\u_{s},w_{p}\right) - b_p(\u_{p},w_{p}) -b_{s}(\u_s,\bt_e)- b_f(\u_{f},w_{f}) \nonumber
	\\
	& \qquad\quad 
	= (q_{f},w_{f})_{\O_f} + (q_{p},w_{p})_{\O_p}, \label{h-cts-alt-2} \\
	& b_{\Gamma}\left(\u_{f},\u_{p},\u_{s};\mu\right) = 0. \label{h-cts-alt-gamma}
\end{align}
We can write \eqref{h-cts-alt-1}--\eqref{h-cts-alt-gamma} in an operator notation
as a degenerate evolution problem in a mixed form:
\begin{align}
	\frac{\d}{\d t} \Ec_1 \q(t) +\mathcal{A}\q(t)+ \Bc's(t) &= \f(t) && \text{in } \Q', \label{eq:1}\\
	\frac{\d}{\d t} \Ec_2 s(t) - \Bc\q(t) +\mathcal{C}s(t) &= g(t) && \text{in } S', \label{eq:2}
\end{align}
where we define $\Q$, the space of generalized displacement variables, as  
\begin{align*}
	\Q =\Big\{ \q=(\v_p,\v_s,\v_f) \, \in\,  \V_p \times \X_p \times \V_f \Big\},
\end{align*}
and, similarly, the space $S$, consisting of generalized stress variables, as
\begin{align*}
	S =\left\{ s=(w_p,\btau_e,w_f,\mu)\in W_p\times\bS_e \times W_f \times \Lambda \right \}.
\end{align*}
The spaces $\Q$ and $S$ are equipped with norms:
\begin{align*}
	\|\q\|_{\Q} &= \|\v_p\|_{\V_p} + \|\v_s\|_{\X_p} + \|\v_f\|_{\V_f}, \\
	\|s\|_{S} &= \|w_p\|_{W_p} +\|\btau_e\|_{\bS_e} + \|w_f\|_{W_f} 
+ \|\mu\|_{\Lambda}.
\end{align*}
The operators $ \Ac:\Q \rightarrow \Q',\,  \Bc: \Q \rightarrow S',\, \mathcal{C}: S \rightarrow S'$, and the functionals $\f \in \Q'$, $g \in S'$ are defined as follows:
$$
\Ac = 
\begin{pmatrix}
\nu_{eff}\kappa^{-1}  & 0  & 0 \\
0 & \alpha_{BJS} \, \gamma_t' \, \nu_I \, \sqrt{\kappa^{-1}} \, \gamma_t & -\alpha_{BJS} \, \gamma_t' \, \nu_I \, \sqrt{\kappa^{-1}} \, \gamma_t \\
0 & -\alpha_{BJS} \, \gamma_t' \, \nu_I \, \sqrt{\kappa^{-1}} \, \gamma_t  & 2\nu\D:\D+ \alpha_{BJS} \, \gamma_t' \, \nu_I \, \sqrt{\kappa^{-1}} \, \gamma_t 
\end{pmatrix},  
$$
$$
\Bc = 
\begin{pmatrix}
\nabla \cdot & \alpha_p\nabla \cdot  & 0  \\
0& -\D & 0 \\
0 & 0 &\nabla \cdot \\
\gamma_n & \gamma_n  & \gamma_n 
\end{pmatrix},  \quad 
\mathcal{C} = 
\begin{pmatrix}
0 & 0  & 0 & 0  \\
0 & 0 & 0 & 0 \\
0 & 0 & 0 & 0 \\
0 & 0 & 0 & 0 
\end{pmatrix},
\quad
\f = \begin{pmatrix}
{\bf 0} \\ \f_p \\ \f_f
\end{pmatrix},
\quad
g = \begin{pmatrix}
q_p \\ 0 \\ q_f \\ 0
\end{pmatrix},
$$
where $\gamma_t$ and $\gamma_n$ denote the tangential and normal trace operators, 
respectively, and $\gamma_t'$ is the adjoint operator of $\gamma_t$.
The operators $\Ec_1: \Q \rightarrow \Q',\,\Ec_2: S \rightarrow S' $ are given by:
$$
\Ec_1 = 
\begin{pmatrix}
0 & 0  & 0 \\
0 & 0 & 0 \\
0 & 0 & 0 
\end{pmatrix},  \quad 
\Ec_2 = 
\begin{pmatrix}
s_0 & 0  & 0 & 0 \\
0 & A & 0 & 0 \\
0 & 0 & 0 & 0 \\
0 & 0 & 0 & 0
\end{pmatrix}.
$$
%
%
\section{Well-posedness of the model}
In this section we establish the solvability of
\eqref{h-cts-1}-\eqref{h-cts-gamma}. We start with the analysis of the
alternative formulation \eqref{h-cts-alt-1}--\eqref{h-cts-alt-gamma}.

\subsection{Existence and uniqueness of a solution of the alternative formulation}
We first explore important properties of the operators introduced at the end of 
Section~\ref{sec:var}.  
\begin{lemma} \label{lmapropAE}
The operator $\Bc$ and its adjoint $\Bc'$ are bounded and continuous. 
Moreover, there exist constants $\beta_1,\, \beta_2>0$ such that
	\begin{align}
	\inf_{{\bf 0}\neq({\bf 0},\v_s, {\bf 0})\in \Q}\sup_{(0,\bt_e,0,0)\in S}\frac{b_s(\v_{s}, \bt_e)}{\|({\bf 0},\v_s, {\bf 0}) \|_{\Q}\|(0,\bt_e,0,0)\|_{S}} \geq \beta_1, \label{inf-sup-elasticity}\\
	\inf_{0\neq(w_p,{\bf 0},w_f,\mu)\in S}\sup_{(\v_p,{\bf 0},\v_f)\in \Q}\frac{b_f(\v_f,w_f)+b_p(\v_p,w_p)+ b_{\Gamma}(\v_f,\v_p,{\bf 0};\mu)}{\|(\v_p,{\bf 0},\v_f)\|_{\Q}\|(w_p,{\bf 0},w_f,\mu)\|_S } \geq \beta_2 \label{inf-sup-stokes-darcy}.
	\end{align}
\end{lemma}
\begin{proof}
The operator $\Bc$ is linear and satisfies for all 
$\q=(\v_p,\v_s,\v_f)\in \Q$ and $s=(w_p,\bt_e,w_f,\mu)\in S$,
	\begin{align*}
	\Bc(\q)(s) &= b_f(\v_{f},w_{f})+ b_p(\v_{p},w_{p}) + \alpha_p b_p(\v_{s},w_{p}) 
+ b_{s}(\v_s,\bt_e)+ b_{\Gamma}(\v_{f},\v_{p},\v_{s};\mu) \\
	& \leq \|\nabla \cdot \v_f\|_{L^r(\O_f)}\|w_f\|_{L^{r'}(\O_f)} +  \|\nabla \cdot \v_p\|_{L^r(\O_p)}\|w_p\|_{L^{r'}(\O_p)} +\alpha_p \|\nabla \cdot \v_s\|_{L^r(\O_p)}\|w_p\|_{L^{r'}(\O_p)}\\
	& \quad + \|\D(\v_s)\|_{L^2(\O_p)}\|\bt_e\|_{L^2(\O_p)}  
+ \| \v_f\cdot \n_f +(\v_p + \v_s)\cdot \n_p\|_{W^{-1/r,r}(\Gam_{fp})}\|\mu\|_{W^{1/r,r'}(\Gam_{fp})} \\
	&\leq C\Big(\|\v_f\|_{W^{1,r}(\O_f)}\|w_f\|_{L^{r'}(\O_f)} +\| \v_p\|_{L^r(\mbox{div};\O_p)}\|w_p\|_{L^{r'}(\O_p)} + \| \v_s\|_{H^1(\O_p)}\|w_p\|_{L^{r'}(\O_p)} \\
	& \quad
+ \| \v_s\|_{H^1(\O_p)}\|\bt_e\|_{L^{2}(\O_p)} 
+ \|\v_f\|_{W^{1,r}(\O_f)}\|\mu\|_{W^{1/r,r'}(\Gam_{fp})} + \| \v_p\|_{L^r(\mbox{div};\O_p)}\|\mu\|_{W^{1/r,r'}(\Gam_{fp})} \\
	& \quad 
+ \| \v_s\|_{H^1(\O_p)}\|\mu\|_{W^{1/r,r'}(\Gam_{fp})} \Big)  \leq C \|\q\|_{\Q}\|s\|_S,
	\end{align*}
	which implies that $\Bc$ and $\Bc'$ are bounded and continuous.
	
Next, let ${\bf 0}\neq({\bf 0},\v_s,{\bf 0})\in \Q$ be given. We
choose $\btau_e=\D(\v_s)$ and, using Korn's inequality, 
$ \| \D(\w) \|_{L^2(\O_p)} \ge C_{K,p} \|\w\|_{H^1(\O_p)}$, for $\w \in \X_p$,
we obtain 
	\begin{align*}
	\frac{b_s(\v_s,\bt_e) }{\|\btau_e\|_{L^2(\O_p)}} = \frac{\|\D(\v_s)\|^2_{L^2(\O_p)} }{\|\D(\v_s)\|_{L^2(\O_p)}} = \|\D(\v_s)\|_{L^2(\O_p)} \geq C_{K,p}\|\v_s\|_{H^1(\O_p)}.
	\end{align*} 
	Therefore, \eqref{inf-sup-elasticity} holds.
	
Finally, we note that \eqref{inf-sup-stokes-darcy} was proven in
\cite{ervin2009coupled} in the case of velocity boundary conditions with
restricted mean value of $W_f\times W_p$. However, it can be
shown that the result holds with no restriction on $W_f\times W_p$ since $|\Gam_D|>0$.
\end{proof}

Let us define, for $\v_f\in \V_f$ and $\v_s\in \X_p$,
\begin{align*}
|\v_f-\v_s|_{BJS} = \sum_{j=1}^{d-1}\alpha_{BJS}\|(\v_f-\v_s)\cdot \t_{f,j}\|_{L^r(\Gam_{fp})}.
\end{align*}

\begin{lemma} \label{lmapropBBa}
The operators $\Ac$ and $\Ec_2$ are bounded, continuous, and monotone.
In addition, the following continuity and coercivity estimates hold
with constants $c_f, \, \bar c_f, \, C_f, \, c_p, \, \bar c_p, \, C_p,
\, c_I, \, \bar c_I, \, C_I > 0$ for all $\u_f,\v_f\in \V_f,\,
\u_p,\v_p\in \V_p$ and $\u_s,\v_s\in \X_p$,
\begin{align}
& c_f\|\v_f\|^{r}_{W^{1,r}(\O_f)} - c*\bar c_f \leq a_f(\v_f,\v_f), \quad
a_f(\u_f,\v_f) \leq C_f \|\u_f\|^{r/r'}_{W^{1,r}(\O_f)}\|\v_f\|_{W^{1,r}(\O_f)}, 
\label{a-f-bounds}\\
& c_p\|\v_p\|^{r}_{L^r(\O_p)} - c*\bar c_p \leq a_p^d(\v_p,\v_p), \quad
a_p^d(\u_p,\v_p) \leq C_p \|\u_p\|^{r/r'}_{L^r(\O_p)}\|\v_p\|_{L^r(\O_p)}, 
\label{a-p-bounds}\\
& c_I|\v_f-\v_s|^{r}_{BJS} - c*\bar c_I \leq a_{BJS}(\v_f,\v_s;\v_f,\v_s), \,
a_{BJS}(\u_f,\u_s;\v_f,\v_s)\leq C_I |\u_f-\u_s|^{r/r'}_{BJS}\|\v_f-\v_s\|_{L^{r}(\Gam_{fp})},
\label{a-bjs-bounds}
\end{align}
where $c$ is the constant from \eqref{B1}--\eqref{B2}.
\end{lemma}
\begin{proof}
The operator $\Ec_2$ is linear and, using \eqref{defAten}, it satisfies
	\begin{align*}
	\Ec_{2}(s)(t) &= (s_0 p_p,w_p)_{\O_p} +(A\bs_e,\bt_e)_{\O_p} \leq C \left(\|p_p\|_{L^2(\O_p)}\|w_p\|_{L^2(\O_p)} + \|\bs_e\|_{L^2(\O_p)} \|\bt_e\|_{L^2(\O_p)}\right),  \\
	\Ec_{2}(s)(s) &= (s_0 p_p,p_p)_{\O_p} +(A\bs_e,\bs_e)_{\O_p} \geq C\left(\|p_p\|^2_{L^2(\O_p)} + \|\bs_e\|^2_{L^2(\O_p)} \right), \quad \forall s,t \in S,  
	\end{align*}
which imply that $\Ec_2$ is bounded, continuous and monotone.
The continuity and monotonicity of the operator $\Ac$
follow from \eqref{B1}--\eqref{B2}, see \cite{ervin2009coupled} and
\cite[Example 5.a, p.59]{showalter2013monotone}.

For the continuity of $a_f(\cdot,\cdot)$, we apply 
\eqref{continuity} with $\G(\x) = \nu(\x)\x$,
$\s=\D(\u_f),\, \t={\bf 0}$ and $\w =\D(\v_f)$:
\begin{align*}
a_f(\u_f,\v_f) \leq 2 \,C_6\left\|\frac{|\D(\u_f)|}{c+|\D(\u_f)|} \right\|^{\frac{2-r}{r}}_{L^{\infty}(\O_f)} 
(|\nu(\D(\u_f))\D(\u_f)|,|\D(\u_f)|)_{\O_f}^{1/r'}\|\D(\v_f)\|_{L^r(\O_f)}.
\end{align*}
Using \eqref{B2} with $\x = {\bf 0}$, $\h = \D(\u_f)$, we also have
	\begin{align*}
	|\nu(\D(\u_f))\D(\u_f)| \leq C_4 \frac{|\D(\u_f)|}{c+ |\D(\u_f)|^{2-r}} \leq C_4 \frac{|\D(\u_f)|^{r-1}}{c|\D(\u_f)|^{r-2}+ 1} \leq C_4 |\D(\u_f)|^{r-1}.
	\end{align*}
Combining the above two estimates, we obtain
	\begin{align*}
	a_f(\u_f,\v_f) \leq C \|\D(\u_f)\|_{L^r(\O_f)}^{r/r'}\|\D(\v_f)\|_{L^r(\O_f)} \leq C_f\|\u_f\|^{r/r'}_{W^{1,r}(\O_f)}\|\v_f\|_{W^{1,r}(\O_f)}.
	\end{align*}
To establish the coercivity bound for $a_f( \cdot , \cdot)$ given in
\eqref{a-f-bounds} we consider three cases.

(i) $c = 0$. From \eqref{monotonicity} we have
\begin{align}
  a_f(\v_f,\v_f) 
  \ge  2 \,C_{5}  \frac{\| \D(\v_f) \|_{L^r(\O_f)}^{2}}{\| \D(\v_f) \|_{L^r(\O_f)}^{2-r}}
   =  2\,C_{5} \| \D(\v_f) \|_{L^r(\O_f)}^{r} 
  \ge 2\,C_{5} C_{K,f}^{r} \, \|\v_f\|_{W^{1,r}(\O_f)}^{r},
  \label{eqrrt1}
\end{align}
where $C_{K,f}$ is the constant arising in Korn's inequality, 
$  \| \D(\w) \|_{L^r(\O_f)} \ge C_{K,f} \|\w\|_{W^{1,r}(\O_f)}$, for $\w \in \V_f$. 

(ii) $c \ne 0$ and $\v_f \in \V_f$ with  $\| \D(\v_f) \|_{L^r(\O_f)}^{2-r} \ge c$. 
Then from \eqref{monotonicity}  we have
\begin{align}
  a_f(\v_f,\v_f) 
  \ge  2\,C_{5}  \frac{\| \D(\v_f) \|_{L^r(\O_f)}^{2}}{c +  \| \D(\v_f) \|_{L^r(\O_f)}^{2-r}}
  \ge  C_5 \| \D(\v_f) \|_{L^r(\O_f)}^{r}
   \ge  C_5 C_{K}^{r} \, \|\v_f\|_{W^{1,r}(\O_f)}^{r}.
  \label{eqrrt2}
\end{align}

(iii) $c \ne 0$ and $\v_f \in \V_f$ with  $\| \D(\v_f) \|_{L^r(\O_f)}^{2-r} < c$. Then
$C_{K}^{r} \| \v_f \|_{W^{1,r}(\O_f)}^{r} \le \| \D(\v_f) \|_{L^r(\O_f)}^{r}  \le c^{r / (2 - r)}$.
Denote the coercivity constant from \eqref{eqrrt2} as
$c_{f} = C_5 C_{K}^{r}$ and let $\bar{c}_{f} = C_5 c^{(2r - 2) / (2 - r)}$. Now,
\[
    c_{f}  \| \v_f \|_{W^{1,r}(\O_f)}^{r}  \le C_5 \| \D(\v_f) \|_{L^r(\O_f)}^{r} 
    \le C_5  c^{r / (2 - r)}  = c \bar{c}_{f},
\]
hence
\begin{equation}
c_{f}  \| \v_f \|_{W^{1,r}(\O_f)}^{r}  - c \bar{c}_{f} \le 0 \le a_f(\v_f,\v_f).
  \label{eqrrt3}
\end{equation}
Combining \eqref{eqrrt1}-\eqref{eqrrt3} yields the coercivity estimate given in 
\eqref{a-f-bounds}. The reader is also
referred to \cite{necas1996weak}, where a similar result is proven
under slightly different assumptions, which are satisfied by the Carreau model 
with $\nu_{\infty} = 0$. 

The continuity and coercivity bounds \eqref{a-p-bounds} and \eqref{a-bjs-bounds} 
follow in the same way.
\end{proof}

\begin{remark}
The system \eqref{eq:1}--\eqref{eq:2} is a degenerate evolution
problem in a mixed form, which fits the structure of the problems
studied in \cite{Showalter-SIAMMA}. However, the analysis in
\cite{Showalter-SIAMMA} is restricted to the Hilbert space setting and
needs to be extended to the Sobolev space setting. Furthermore, the
analysis in \cite{Showalter-SIAMMA} is for monotone
operators, see \cite{showalter2013monotone}, and it is restricted to
$\f \in \Q_1'$ and $g \in S_2'$, where $\Q_1'$ and $S_2'$ are the spaces
$\Q$ and $S$ with semi-scalar products arising from $\Ec_1$ and $\Ec_2$,
respectively. In our case this translates to $\f_p = \f_f = {\bf 0}$ and
$q_f = 0$. To avoid this restriction, we take a different approach, based on
reformulating the problem as a parabolic problem for $p_p$ and $\bs_e$. The well 
posedness of the resulting problem is established using the coercivity of
the functionals established in Lemma~\ref{lmapropBBa}.
\end{remark}

Denote by $W_{p,2}$ and $\bS_{e,2}$ the closure of the spaces $W_p$ and $\bS_e$ with 
respect to the norms
\begin{align*}
\|w_p\|^2_{W_{p,2}} \, := \, (s_0w_p,w_p)_{L^2(\O_p)},\quad \|\bt_e\|^2_{\bS_{e,2}} \, := \, (A\bt_e,\bt_e)_{L^2(\O_p)} \, .
\end{align*}
Note that $W_{p,2} = L^{2}(\Omega_{p})$, and $\bS_{e,2} = \bS_{e}$. 
Let $S_2=W_{p,2}\times \bS_{e,2}$. We introduce the
inner product $(\cdot , \cdot)_{S_2}$ defined by $\left( (w_1 , \bt_1) , (w_2 , \bt_2) \right)_{S_2} := 
(s_0w_1 , w_2)_{L^2(\O_p)} \, + \, (A\bt_1 , \bt_2)_{L^2(\O_p)}$.

Define the domain 
\begin{align}
D \, := \, &
\left\{ (p_p,\bs_e)\in W_p\times \bS_e:\, \mbox{ for given } 
(\f_f,\f_p,q_f)\in \V_f'\times \X_p'\times W_f'   \right. \nonumber\\
& \quad \quad  \quad \exists \, ( (\u_{p},\u_{s},\u_{f}), p_f,\lam)\in \Q\times W_f\times \Lambda \mbox{ such that } 
\forall ( (\v_{p},\v_{s},\v_{f}), (w_p, \btau_e, w_f, \mu)) \in \Q\times S \mbox{:} \nonumber\\
&  
a_{f}(\u_{f},\v_{f}) + a^d_{p}(\u_{p},\v_{p}) 
+ a_{BJS}(\u_{f}, \u_{s};\v_{f},\v_{s}) + b_f(\v_{f},p_{f})+ b_p(\v_{p},p_{p})\nonumber
\\
& \qquad\quad   + 
\alpha_p b_p(\v_{s},p_{p}) +b_{s}(\v_s,\bs_e)+ b_{\Gamma}(\v_{f},\v_{p},\v_{s};\lam)  = (\f_{f},\v_{f})_{\O_f} + (\f_{p},\v_{s})_{\O_p}, \label{h-cts-alt-dom-1}\\
& \left( s_0 p_{p},w_{p}\right)_{\O_p} + a^s_{p}(\bs_{e},\bt_{e})
- \alpha_p b_p\left(\u_{s},w_{p}\right) - b_p(\u_{p},w_{p}) -b_{s}(\u_s,\bt_e)- b_f(\u_{f},w_{f}) \nonumber
\\
& \qquad\quad 
= (q_{f},w_{f})_{\O_f} + (s_0\bar{g}_{p},w_{p})_{\O_p}+ (A\bar{g}_{e},\bt_{e})_{\O_p}, \label{h-cts-alt-dom-2} \\
& b_{\Gamma}\left(\u_{f},\u_{p},\u_{s};\mu\right) = 0 , \label{h-cts-alt-dom-gamma} \\
& \left. \quad \quad  \quad \mbox{for some } (\bar{g}_p,\bar{g}_e)\in W_{p,2}^{\prime}\times \bS_{e,2}^{\prime} \,  \right\} 
\subset W_{p,2}\times \bS_{e,2}  \, . \label{deffD}
\end{align}
We note that \eqref{h-cts-alt-dom-1}--\eqref{h-cts-alt-dom-gamma} can
be written in an operator form as
\begin{align*}
\Ac \q + \Bc's & = \f  \quad \text{in } \Q', \\
- \Bc\q + \Ec_2 s & = \bar g \quad \text{in } S',
\end{align*}
where $\bar g \in S'$ is the functional on the right hand side of \eqref{h-cts-alt-dom-2}.

Note that there may be more than one 
$(\bar{g}_p,\bar{g}_e)\in W_{p,2}^{\prime}\times \bS_{e,2}^{\prime}$ that generate the same $ (p_p,\bs_e) \in D$.
In view of this, we introduce the multivalued operator $\mathcal{M}(\cdot)$ with domain $D$ defined by
\begin{equation}
\mathcal{M}( (p_p,\bs_e) ) \ := \ 
\left\{ (\bar{g}_p - p_p \, , \, \bar{g}_e - \bs_e) \, : \, (p_p,\bs_e)
\mbox{ satisfies (4.9)--(4.11) for }
 (\bar{g}_p,\bar{g}_e)\in W_{p,2}^{\prime}\times \bS_{e,2}^{\prime} \right\} \, .
\label{defmcM}
\end{equation}
Associated with  $\mathcal{M}(\cdot)$ we have the \textit{relation} $\mathcal{M} \subset \left( W_p\times \bS_e \right) \times 
 \left( W_{p,2} \times \bS_{e,2} \right)^{\prime}$ with domain $D$, where $[\v , \f] \in \mathcal{M}$
if $\v \in D$ and $\f \in \mathcal{M}( \v )$. \\

Consider the following problem: Given $h_p \in W^{1,1}(0, T; W_{p,2}')$ and
$h_e \in W^{1,1}(0, T; \bS_{e,2}')$, find $(p_p,\bs_e) \in D$ satisfying
\begin{align}
\frac{d}{dt}\begin{pmatrix}
p_p(t) \\
\bs_e(t)
\end{pmatrix} + \mathcal{M} \begin{pmatrix}
p_p(t) \\
\bs_e(t)
\end{pmatrix} \ni \begin{pmatrix}
h_p(t) \\
h_e(t)
\end{pmatrix}. \label{ode-problem}
\end{align}

A key result that we use to establish the existence of a solution to
\eqref{h-cts-alt-1}--\eqref{h-cts-alt-gamma} is the following theorem; for
details see \cite[Theorem 6.1(b)]{showalter2013monotone}.
\begin{theorem}  \label{thmsho61b}
	Let the linear, symmetric and monotone operator $\mathcal{N}$ be given for the real vector space $E$ to its algebraic
	dual $E^{*}$, and let $E_{b}'$ be the Hilbert space which is the dual of $E$ with the seminorm
	\[            | x |_{b} \ = \ (\mathcal{N} x \, (x))^{1/2}  \, , \quad x \in E \, .  \]
	Let $\mathcal{M} \subset E \times E_{b}'$ be a relation with domain 
$D = \{ x \in E \, : \, \mathcal{M}(x) \neq \emptyset \}$. \\
	Assume $\mathcal{M}$ is monotone and $Rg( \mathcal{N} + \mathcal{M}) \, = \, E_{b}'$. Then, for each $u_{0} \in D$ and
	for each $f \in W^{1 , 1}(0 , T ; E_{b}')$, there is a solution $u$ of
	\[     \frac{d}{dt}\left( \mathcal{N} u(t) \right) \ + \ \mathcal{M}\left( u(t) \right) \ni f(t) \, , \quad 0 < t < T \, , \]
	with
	\[ \mathcal{N} u \in W^{1 , \infty}(0 , T; E_{b}') \, , \ \ \ u(t) \in D \, , \ \ \ \mbox{\textit{for all} } 0 \le t \le T \, , 
	\ \ \mbox{ \textit{ and } } \mathcal{N}u(0) \, = \, \mathcal{N}u_{0} \, .   \]
\end{theorem}
Using Theorem \ref{thmsho61b}, we can show that the problem
\eqref{h-cts-alt-1}--\eqref{h-cts-alt-gamma} is well-posed.
\begin{theorem}\label{well-pos-alt}
For each $\f_{f} \in
  W^{1,1}(0, T; \V_{f}')$, $\f_{p} \in W^{1,1}(0, T; \X_{p}')$, $q_f
  \in W^{1,1}(0, T; W_{f}')$, \linebreak $q_p \in W^{1,1}(0, T; W_{p}')$, and
  $p_p(0) = p_{p,0} \in W_p$, $\bs_e(0) = A^{-1}\D(\bbeta_{p,0}) \in \bS_e$,
there exists a solution of
\eqref{h-cts-alt-1}--\eqref{h-cts-alt-gamma} with 
$(\u_f$,
$p_f$, $\u_p$, $p_p$, $\u_s$, $\bs_e$, $\lam) \in
L^{\infty}(0,T;\V_f)\times L^{\infty}(0,T;W_f)\times
L^{\infty}(0,T;\V_p) \times W^{1,\infty}(0,T;W_p) \times
L^{\infty}(0,T;\X_p) \times W^{1,\infty}(0,T;\bS_e) \times
L^{\infty}(0,T;\Lambda)$.
\end{theorem}

To prove Theorem \ref{well-pos-alt} we proceed in the following manner. \\
\textbf{Step 1}. (Section \ref{sssec_one}) Establish that the domain $D$ 
given by \eqref{deffD} is nonempty. \\
	\textbf{Step 2}. (Section \ref{sssec_two}) Show solvability of the parabolic problem \eqref{ode-problem}. \\
	\textbf{Step 3}. (Section \ref{sssec_three}) Show that the original problem 
	\eqref{h-cts-alt-1}--\eqref{h-cts-alt-gamma} is a special case of \eqref{ode-problem}.	
	
	Each of the steps will be covered in details in the corresponding subsection.

\subsubsection{Step 1: The Domain $D$ is nonempty}
\label{sssec_one}
We begin with a number of preliminary results used in the proof.
We first introduce operators that will be used to regularize the problem. Let
$R_{s} \, : \, X_{p} \longrightarrow X_{p}'$,
$R_{p} \, : \, V_{p} \longrightarrow V_{p}'$, $L_{f} \, : \, W_{f}
\longrightarrow W_{f}'$, $L_{p} \, : \, W_{p} \longrightarrow W_{p}'$ be defined by
\begin{align}
R_{s}(\u_s)(\v_s) &:= r_s(\u_s,\v_s) = (\D(\u_s), \D(\v_s))_{\O_p}, \label{defRs} \\
R_{p}(\u_p)(\v_p) &:= r_p(\u_p,\v_p) = (|\nabla \cdot \u_p|^{r-2}\nabla \cdot \u_p, \nabla \cdot \v_p)_{\O_p} \, ,
\label{defRp}  \\
L_{f}(p_f)(w_f) &:=   l_f(p_f,w_f) = (|p_f|^{r'-2}p_f,w_f)_{\O_f}, \label{defLf} \\
L_{p}(p_p)(w_p) &:= l_p(p_p,w_p) = (|p_p|^{r'-2}p_p,w_p)_{\O_p} \, .  \label{defLp}
\end{align}

\begin{lemma} \label{lmapropRL}
The operators $R_{s}$, $R_{p}$, $L_{f}$, and $L_{p}$ are bounded,
continuous, coercive, and monotone.
\end{lemma}
\begin{proof}
The operators satisfy the following continuity and coercivity bounds:
\begin{align*} \allowdisplaybreaks
& R_s(\u_s)(\v_s) \leq \|\u_s\|_{H^1(\O_p)}\|\v_s\|_{H^1(\O_p)}, 
&& R_s(\u_s)(\u_s) \geq C_{K,p}\|\u_s\|^2_{H^1(\O_p)}, && \forall \u_s,\v_s \in \X_p,\\
& R_p(\u_p)(\v_p) \leq \|\nabla \cdot \u_p\|^{r/r'}_{L^r(\O_p)}\|\nabla \cdot \v_p\|_{L^r(\O_p)}, 
&& R_p(\u_p)(\u_p) \geq \|\nabla \cdot \u_p\|^r_{L^r(\O_p)}, && \forall \u_p,\v_p \in \V_p, \\
& L_f(p_f)(w_f) \leq \|p_f\|^{r'/r}_{L^{r'}(\O_f)}\|w_f\|_{L^{r'}(\O_f)}, 
&& L_f(p_f) (p_f)\geq \|p_f\|^{r'}_{L^{r'}(\O_f)}, && \forall p_f,w_f \in W_f, \\
& L_p(p_p)(w_p) \leq \|p_p\|^{r'/r}_{L^{r'}(\O_p)}\|w_p\|_{L^{r'}(\O_p)}, 
&& L_p(p_p)(p_p) \geq \|p_p\|^{r'}_{L^{r'}(\O_p)}, && \forall p_p,w_p \in W_p.
\end{align*}
The coercivity bounds follow directly from the definitions, using
Korn's inequality for $R_s$. The continuity bounds follow from the
Cauchy-Schwarz or H\"{o}lder's inequalities. The above bounds imply
that the operators are bounded, continuous, and coercive. Monotonicity
follows from bounds similar to \eqref{monotonicity}, which can be
established in a way similar to the Power law model \cite{Sandri}.
\end{proof} 
It was shown in \cite{ervin2009coupled} that there exists a bounded
extension of $\lam$ from $W^{1/r,r'}(\Gam_{fp})$ to
$W^{1/r,r'}(\partial \O_p)$, defined as $E_{\Gam}\lam
= \gamma \phi(\lam)$, where $\gamma$ is the trace operator from
$W^{1,r'}(\O_p)$ to $W^{1/r,r'}(\partial \O_p)$ and $\phi(\lam) \in
W^{1,r'}(\O_p)$ is the weak solution of
\begin{align}
-\nabla \cdot |\nabla \phi(\lam)|^{r'-2}\nabla \phi(\lam)  & = 0 , \quad\mbox{in } \O_p, \label{aux-prob-1}\\
\phi(\lam)& = 
\lam , \quad \mbox{on } \Gam_{fp} , \label{aux-prob-2} \\
|\nabla \phi(\lam)|^{r'-2}\nabla \phi(\lam)\cdot \n & = 0 , \quad \mbox{on } \partial \O_p \setminus\Gam_{fp}  
\label{aux-prob-3} \, .
\end{align}
We have the following equivalence of norms statement.
\begin{lemma} \label{lmaequiv1}
	For $\lam$ $\in$ $W^{1/r,r'}(\Gam_{fp})$ and $\phi(\lam)$ 
defined by \eqref{aux-prob-1}--\eqref{aux-prob-3}, there exists $c_1$, $c_2 > 0$ such that
	\begin{equation}
c_1 \| \phi(\lam)\|_{W^{1,r'}(\O_p)} \le
	\|\lam\|_{W^{1/r,r'}(\Gam_{fp})} \le c_2 \| \phi(\lam)\|_{W^{1,r'}(\O_p)}.
	\label{eqvlman}
	\end{equation}
\end{lemma}
\begin{proof}
For $\phi \in W^{1,r'}(\O)$, $ |\nabla \phi(\lam)|^{r'-2}\nabla
\phi(\lam) \in L^{r'}(\text{div};\O)$ and, therefore, from
\eqref{aux-prob-1}--\eqref{aux-prob-3}, we have
	\begin{align}
	(|\nabla \phi(\lam)|^{r'-2}\nabla \phi(\lam), \nabla \phi(\lam))_{\O_p} 
	&= \langle |\nabla \phi(\lam)|^{r'-2}\nabla \phi(\lam)\cdot \n,  E_{\Gam}\lam \rangle_{\partial \O_p} \nonumber \\
	&\leq \||\nabla \phi(\lam)|^{r'-2}\nabla \phi(\lam)\cdot \n\|_{W^{-1/r,r}(\partial \O_p)} 
	\|E_{\Gam}\lam\|_{W^{1/r,r'}(\partial \O_p)} \nonumber \\
	&\leq  C \, \||\nabla \phi(\lam)|^{r'-2}\nabla \phi(\lam)\cdot \n\|_{W^{-1/r,r}(\partial \O_p)} \|\lam\|_{W^{1/r,r'}(\Gam_{fp})}. \label{energy-bound}
	\end{align}
	Now, for $\psi \in W^{1,r'}(\O_p)$,
	\begin{align}
	\int_{\partial \O_p} |\nabla \phi(\lam)|^{r'-2}\nabla \phi(\lam)\cdot \n \, \psi \, ds 
	&= \int_{\O_p} \nabla \cdot | \nabla \phi(\lam)|^{r'-2}\nabla \phi(\lam) \, \psi \, d\x 
	\ + \ \int_{\O_p} |\nabla \phi(\lam)|^{r'-2}\nabla \phi(\lam) \cdot \nabla \psi \, d\x    \nonumber \\
	&\leq
	\| |\nabla \phi(\lam)|^{r'-2}\nabla \phi(\lam) \|_{L^{r}(\O_p)} \, \| \psi \|_{W^{1,r'}(\O_p)} \quad \quad
	\mbox{(using \eqref{aux-prob-1})}  \nonumber \\
	&= \| \nabla \phi \|_{L^{r'}(\O_p)}^{r' / r} \,    \| \psi \|_{W^{1,r'}(\O_p)} \, .  \label{rrrty1}
	\end{align}
	Using the fact the trace operator, $\gamma(\cdot)$, is a bounded, linear, bijective 
operator for the quotient space
	$W^{1,q}(\O_p) / W_{0}^{1,q}(\O_p)$ onto $W^{1 - \frac{1}{q} \, , \, q}(\partial \O_p)$ \cite{galdi2011introduction},
	we have
	\begin{align}
	\| |\nabla \phi(\lam)|^{r'-2}\nabla \phi(\lam)\cdot \n \|_{W^{-1/r,r}(\partial \O_p)}
	&= \sup_{\xi \in W^{1/r , r'}(\partial \O_p)}
	\frac{ \langle |\nabla \phi(\lam)|^{r'-2} \nabla \phi(\lam)\cdot \n \, , \xi 
		\rangle_{W^{-1/r,r}(\partial \O_p) \, , W^{1/r , r'}(\partial \O_p)}}{ \| \xi \|_{W^{1/r , r'}(\partial \O_p)}}  \nonumber \\
	&\le C \,   \sup_{\psi  \in W^{1 , r'}(\O_p)}
	\frac{ \int_{\partial \O_p} |\nabla \phi(\lam)|^{r'-2}\nabla \phi(\lam)\cdot \n \, \gamma(\psi) \, ds}%
	{ \| \psi\|_{W^{1 , r'}(\O_p)}}  \nonumber \\
	& \le C \, \| \nabla \phi \|_{L^{r'}(\O_p)}^{r' / r} \, , \mbox{ (using \eqref{rrrty1})}.  \label{kjhte}
	\end{align}
Combining  \eqref{energy-bound} and \eqref{kjhte} with the Poincare inequality
implies that
	\begin{align}
	\|\phi(\lam)\|_{W^{1,r'}(\O)}\leq C \|\lam\|_{W^{1/r,r'}(\Gam_{fp})}. \label{ijn1}
	\end{align}
	On the other hand, due to \eqref{aux-prob-2} and the trace inequality, we have
	\begin{align}
	\|\lam\|_{W^{1/r,r'}(\Gam_{fp})} \leq C \|\phi(\lam)\|_{W^{1,r'}(\O)}.
	\label{ijn2}
	\end{align}
	Combining \eqref{ijn1} and \eqref{ijn2}, we obtain \eqref{eqvlman}.
\end{proof}

Introduce $L_{\Gamma} :  \Lambda \longrightarrow \Lambda'$ defined by
\begin{equation}
L_{\Gamma}(\lambda)(\mu) \, := \, l_{\Gamma}(\lambda , \mu) \ = \
\left( | \nabla \phi(\lambda) |^{r - 2} \, \nabla \phi(\lambda) , \nabla \phi(\mu) \right)_{\Omega_{p}}.
\label{defLg}
\end{equation}

\begin{lemma} \label{lmapropLg}
	The operator $L_{\Gamma}$ is bounded, continuous, coercive, and monotone.
\end{lemma}
\begin{proof} 
The result can be obtained in a similar manner to the proof of Lemma
\ref{lmapropRL}, using the equivalence of norms proved in Lemma
\ref{lmaequiv1}. In particular, it holds that
\begin{equation}\label{L-Gamma-bounds}
L_{\Gamma}(\lambda)(\mu) 
\le C_{\Gamma}\|\lam\|_{W^{1/r,r'}(\Gam_{fp})}^{r'/r}\|\mu\|_{W^{1/r,r'}(\Gam_{fp})}, \qquad
L_{\Gamma}(\lambda)(\lambda)  \ge c_{\Gamma} \|\lam\|_{W^{1/r,r'}(\Gam_{fp})}^{r'}.
\end{equation}
\end{proof}    

To establish that the domain $D$ is nonempty we first show that
there exists a solution to a regularization of
\eqref{h-cts-alt-dom-1}--\eqref{h-cts-alt-dom-gamma}. Then a solution
to \eqref{h-cts-alt-dom-1}--\eqref{h-cts-alt-dom-gamma} is established
by analyzing the regularized solutions as the regularization parameter
goes to zero.
\begin{lemma} \label{thmDdef}
	The domain $D$ specified by \eqref{deffD} is nonempty.
\end{lemma}
\begin{proof}
We will focus on the case \eqref{B1}--\eqref{B2} with $c = 0$, which
holds for the Power law model. The argument for the case $c > 0$
is similar, with an extra constant term on the right-hand side of the energy bound
\eqref{regularized-estimate-1}, due to coercivity estimates 
\eqref{a-f-bounds}--\eqref{a-bjs-bounds}.
	
For $\q^{(i)} \, = \, (\v_{p , i} , \v_{s , i} , \v_{f , i}) \in \Q$, $s^{(i)} \, = \ (w_{p , i} , \btau_{e , i} , w_{f , i} , \mu_{i}) \in S$,
	$ i = 1, 2$, 
	define the operators $\mathcal{R}:\Q \rightarrow \Q'$ and $\mathcal{L}: S \rightarrow S'$ as
	\begin{align*}
	\mathcal{R}( \q^{(1)})(\q^{(2)}) &:= R_{s}(\v_{s , 1}) (\v_{s , 2}) \ + \ R_{p}(\v_{p , 1}) (\v_{p , 2}) \ = \ 
	r_s (\v_{s,1},\v_{s,2}) + r_p (\v_{p,1},\v_{p,2}),\\
	\mbox{ and } \ 
	\mathcal{L}(s^{(1)})(s^{(2)}) &:=
	L_{f}(w_{f,1}) (w_{f,2}) \ + \  L_{p}(w_{p,1}) (w_{p,2}) \ + \  L_{\Gamma}(\mu_{1}) (\mu_{2}) \\
	&= \ 
	l_f(w_{f,1},w_{f,2}) + l_p(w_{p,1},w_{p,2})+ l_{\Gamma}(\mu_1,\mu_2).
	\end{align*}
	%
For $\epsilon>0$, consider a regularization of
\eqref{h-cts-alt-dom-1}--\eqref{h-cts-alt-dom-gamma} defined by:
\textit{Given $\f \in \Q'$, $\bar g \in S'$, determine $\q_{\epsilon} \in
  \Q,\, s_{\epsilon}\in S$ satisfying}
\begin{align}
(\epsilon \mathcal{R} + \Ac) \q_{\epsilon} + \Bc's_{\epsilon} &=\f \quad \mbox{in } \Q', 
\label{regularized-problem-1}\\
- \Bc\q_{\epsilon} + (\epsilon\mathcal{L} + \Ec_2) s_{\epsilon} &= \bar g \quad \mbox{in } S'. 
\label{regularized-problem-2}
\end{align}
	Introduce the operator $\mathcal{O}: \Q\times S \rightarrow (\Q\times S)'$ defined as
	\begin{align*}
	\mathcal{O}  \left( \begin{array}{c}  \q \\ s \end{array} \right)                           
	=\begin{pmatrix}
	\epsilon \mathcal{R} +\Ac   & \Bc' \\
	-\Bc  & \epsilon \mathcal{L} + \Ec_2
	\end{pmatrix}  \left[ \begin{array}{c}  \q \\ s \end{array} \right] \, .
	\end{align*}
	Note that
	\begin{equation}
	\mathcal{O} \left( \begin{array}{c}  \q^{(1)} \\ s^{(1)} \end{array} \right)  
	\left( \left( \begin{array}{c}  \q^{(2)} \\ s^{(2)} \end{array} \right) \right) 
	\ = \ 
	(\epsilon \mathcal{R}+\Ac)(\q^{(1)}) (\q^{(2)}) \ + \ \Bc'(s^{(1)}) (\q^{(2)}) 
	\ - \ \Bc (\q^{(1)}) (s^{(2)}) \ + \ (\epsilon \mathcal{L} + \Ec_2) (s^{(1)}) (s^{(2)}) \, ,
	\label{eqwwr}
	\end{equation}
	and
	\begin{align*}
	&\left(\mathcal{O}\begin{pmatrix}
	\q^{(1)}\\
	s^{(1)}
	\end{pmatrix}-\mathcal{O}\begin{pmatrix}
	\q^{(2)}\\
	s^{(2)}
	\end{pmatrix}\right)\left(\begin{pmatrix}
	\q^{(1)}\\
	s^{(1)}
	\end{pmatrix}-\begin{pmatrix}
	\q^{(2)}\\
	s^{(2)}
	\end{pmatrix}\right) \\
	& \qquad =( (\epsilon \mathcal{R}+\Ac)\q^{(1)} - (\epsilon \mathcal{R}+\Ac)\q^{(2)})(\q^{(1)} -\q^{(2)}) 
	+ ( (\epsilon \mathcal{L} + \Ec_2)s^{(1)} - ( \epsilon \mathcal{L} + \Ec_2 )s^{(2)})(s^{(1)} -s^{(2)}).
	\end{align*}
From Lemmas \ref{lmapropAE}, \ref{lmapropBBa}, \ref{lmapropRL}, and
\ref{lmapropLg} we have that $\mathcal{O}$ is a bounded, continuous,
and monotone operator.  Moreover, using the coercivity bounds from 
\eqref{a-f-bounds}--\eqref{a-bjs-bounds} and \eqref{L-Gamma-bounds}, we also have
\begin{align}
\mathcal{O}\begin{pmatrix} \q\\ s \end{pmatrix}
\left(\begin{pmatrix} \q\\ s \end{pmatrix}\right) 
& = (\epsilon \mathcal{R}+\Ac)\q (\q) + (\Ec_2 +\epsilon \mathcal{L})s(s)  \nonumber \\ 
& =\epsilon r_s (\v_{s},\v_{s}) + \epsilon r_p (\v_{p},\v_{p}) + a_f(\v_f,\v_f)
+ a^d_p(\v_p,\v_p) +a_{BJS}(\v_f,\v_s;\v_f,\v_s) 
\nonumber  \\
& \qquad + (s_0 w_p,w_p)_{\O_p} + a^s_p(\bt_e,\bt_e) 
+ \epsilon l_f(w_{f},w_{f}) +\epsilon l_p(w_{p},w_{p})
+ \epsilon l_{\Gamma}(\mu,\mu)  \nonumber \\
& \geq C\Big(\epsilon\|\D(\v_s)\|^2_{L^2(\O_p)} 
+\epsilon\|\nabla \cdot \v_p\|^{r}_{L^r(\O_p)} + \|\D(\v_f)\|^r_{L^r(\O_f)} 
+\|\v_p\|^r_{L^r(\O_p)} + |\v_f-\v_s|^r_{BJS}  \nonumber \\
& \qquad + s_0\|w_p\|^2_{L^2(\O_p)} +\|\bt_e\|^2_{L^2(\O_p)} 
+\epsilon\|w_f\|^{r'}_{L^{r'}(\O_f)} +\epsilon \|w_p\|^{r'}_{L^{r'}(\O_p)} 
+\epsilon\|\mu\|^{r'}_{W^{1/r,r'}(\Gam_{fp})} \Big).
\label{eqwws1}
\end{align}
In the case of \eqref{B1}--\eqref{B2} with $c > 0$, we have an extra
term $- c( \bar c_f + \bar c_p + \bar c_I)$ on the right-hand side of
\eqref{eqwws1} due to the coercivity estimates from
\eqref{a-f-bounds}--\eqref{a-bjs-bounds}. The argument in this case doesn't change
and we omit this term for simplicity.
It follows from \eqref{eqwws1} that $\mathcal{O}$ is coercive.  Thus, an
application of the Browder-Minty theorem
\cite{renardy2006introduction} establishes the existence of a solution
$(\q_{\epsilon},s_{\epsilon}) \in \Q\times S$ of
\eqref{regularized-problem-1}--\eqref{regularized-problem-2}, where
$\q_{\epsilon} =(\u_{p,\epsilon},\u_{s,\epsilon}, \u_{f,\epsilon})$
and $s_{\epsilon} = (p_{p,\epsilon}, \bs_{e,\epsilon}, p_{f,\epsilon},
\lam_{\epsilon})$.
	
Now, from \eqref{eqwws1} and
\eqref{regularized-problem-1}--\eqref{regularized-problem-2}, we have
\begin{align}
& \epsilon\|\u_{s,\epsilon}\|^2_{H^1(\O_p)} 
+ \epsilon\|\nabla \cdot \u_{p,\epsilon}\|^{r}_{L^r(\O_p)} 
+ \|\u_{f,\epsilon}\|^r_{W^{1,r}(\O_f)} +\|\u_{p,\epsilon}\|^r_{L^r(\O_p)} 
+ |\u_{f,\epsilon}-\u_{s,\epsilon}|^r_{BJS}\nonumber \\
& \qquad + s_0\|p_{p,\epsilon}\|^2_{L^2(\O_p)} +\|\bs_{e,\epsilon}\|^2_{L^2(\O_p)} +\epsilon\|p_{f,\epsilon}\|^{r'}_{L^{r'}(\O_f)} +\epsilon \|p_{p,\epsilon}\|^{r'}_{L^{r'}(\O_p)} +\epsilon\|\lam_{\epsilon}\|^{r'}_{W^{1/r,r'}(\Gam_{fp})} \nonumber \\
& \quad \leq C\Big( \|\f_p\|_{H^{-1}(\O_p)}\|\u_{s,\epsilon}\|_{H^1(\O_p)} + \|\f_f\|_{W^{-1,r'}(\O_f)}\|\u_{f,\epsilon}\|_{W^{1,r}(\O_f)} \nonumber\\
& \qquad + \|q_f\|_{L^r(\O_f)}\|p_{f,\epsilon}\|_{L^{r'}(\O_f)}+ \|\bar{g}_p\|_{L^r(\O_p)}\|p_{p,\epsilon}\|_{L^{r'}(\O_p)}+ \|\bar{g}_e\|_{L^2(\O_p)}\|\bs_{e,\epsilon}\|_{L^{2}(\O_p)}\Big).
\label{regularized-estimate-1}
\end{align}
From \eqref{h-cts-alt-dom-2},  $\bs_{e,\epsilon}$ and $\u_{s,\epsilon}$ satisfy
\begin{align*}
a^{s}_{p}(\bs_{e,\epsilon},\bt_e) -b_{s}(\u_{s,\epsilon},\bt_e) = (A \bar{g}_e,\bt_e)_{\O_p}, 
\quad \forall \bt_e \in \bS_e.
\end{align*}
Therefore, applying the inf-sup condition \eqref{inf-sup-elasticity}, we obtain:
	\begin{align}
	\|\u_{s,\epsilon}\|_{H^1(\O_p)} &\leq C \sup_{(0,\bt_e,0,0)\in S}\frac{ b_s(\u_{s,\epsilon}, \bt_e)}{\|(0,\bt_e,0,0)\|_{S}} 
= C \sup_{(0,\bt_e,0,0)\in S}\frac{ a^{s}_{p}(\bs_{e,\epsilon},\bt_e) - 
(A \bar{g}_e,\bt_e)_{\O_p}}{\|(0,\bt_e,0,0)\|_{S}} 
	\nonumber\\
	&\leq  C\left(\|\bs_{e,\epsilon}\|_{L^2(\O_p)}+\|\bar{g}_{e}\|_{L^2(\O_p)}\right). \label{displacement-bound}
	\end{align}
Combining \eqref{displacement-bound} and \eqref{regularized-estimate-1}, and
using Young's inequality, for $a,b \ge 0$, $\frac1p + \frac1q = 1$, and $\delta> 0$,
\begin{equation}\label{young}
a b \le \frac{\delta^p a^p}{p} + \frac{b^q}{\delta^q q},
\end{equation}
we obtain
\begin{align}
&\|\u_{s,\epsilon}\|^2_{H^1(\O_p)} +\epsilon\|\nabla \cdot \u_{p,\epsilon}\|^{r}_{L^r(\O_p)} + \|\u_{f,\epsilon}\|^r_{W^{1,r}(\O_f)} +\|\u_{p,\epsilon}\|^r_{L^r(\O_p)} 
+ |\u_{f,\epsilon}-\u_{s,\epsilon} |^r_{BJS} +\epsilon\|\u_{s,\epsilon}\|^2_{H^1(\O_p)}
\nonumber \\
& \qquad + s_0\|p_{p,\epsilon}\|^2_{L^2(\O_p)} +\|\bs_{e,\epsilon}\|^2_{L^2(\O_p)} +\epsilon\|p_{f,\epsilon}\|^{r'}_{L^{r'}(\O_f)} +\epsilon \|p_{p,\epsilon}\|^{r'}_{L^{r'}(\O_p)} +\epsilon\|\lam_{\epsilon}\|^{r'}_{W^{1/r,r'}(\Gam_{fp})} \nonumber \\
& \quad 
\leq C\Big(\|q_f\|_{L^r(\O_f)}\|p_{f,\epsilon}\|_{L^{r'}(\O_f)} + \|\bar{g}_p\|_{L^r(\O_p)}\|p_{p,\epsilon}\|_{L^{r'}(\O_p)} +\|\f_p\|^2_{H^{-1}(\O_p)} \nonumber \\
& \qquad + \|\f_f\|^{r'}_{W^{-1,r'}(\O_f)}+\|\bar{g}_e\|^2_{L^2(\O_p)} \Big) 
+ \frac{1}{2}\left( \|\u_{s,\epsilon}\|^2_{H^1(\O_p)}  
+\|\u_{f,\epsilon}\|^r_{W^{1,r}(\O_f)} + \|\bs_{e,\epsilon}\|^2_{L^2(\O_p)}
\right),\label{regularized-estimate-2}
\end{align}
from which it follows that 
\begin{align}
\|\u_{s,\epsilon}\|^2_{H^1(\O_p)} +  \epsilon\|\nabla \cdot \u_{p,\epsilon}\|^{r}_{L^r(\O_p)} +
\|\u_{f,\epsilon}\|^r_{W^{1,r}(\O_f)} +\|\u_{p,\epsilon}\|^r_{L^r(\O_p)}  +\|\bs_{e,\epsilon}\|^2_{L^2(\O_p)}+|\u_{f,\epsilon}-\u_{s,\epsilon}|^r_{BJS} \nonumber \\
\leq C\left( \|\f_p\|^2_{H^{-1}(\O_p)} + \|\f_f\|^{r'}_{W^{-1,r'}(\O_f)}  +\|q_f\|_{L^r(\O_f)}\|p_{f,\epsilon}\|_{L^{r'}(\O_f)}+ \|\bar{g}_e\|^2_{L^2(\O_p)}+ \|\bar{g}_p\|_{L^r(\O_p)}\|p_{p,\epsilon}\|_{L^{r'}(\O_p)}\right) .\label{regularized-estimate-3}
\end{align}
To obtain bounds for $p_{p,\epsilon}$, $p_{f,\epsilon}$, and $\lam_{\epsilon}$ we use 
\eqref{inf-sup-stokes-darcy}. With 
$s=(p_{p,\epsilon}, \mathbf{0}, p_{f,\epsilon}, \lam_{\epsilon}) \in S$, we have
\begin{align}
& \|p_{f,\epsilon}\|_{L^{r'}(\O_f)} + \|p_{p,\epsilon}\|_{L^{r'}(\O_p)} 
+ \|\lam_{\epsilon}\|_{W^{1/r,r'}(\Gam_{fp})} \nonumber \\
& \qquad
\leq C \sup_{(\v_p,{\bf 0},\v_f)\in \Q}\frac{b_f(\v_f,p_{f,\epsilon})+b_p(\v_p,p_{p,\epsilon})+ b_{\Gamma}(\v_f,\v_p,{\bf 0};\lam_{\epsilon})}{\|(\v_p,{\bf 0},\v_f)\|_{\Q}} 
\nonumber \\
& \qquad
\leq C\sup_{(\v_p,{\bf 0},\v_f) \in \Q}\frac{- \epsilon \, r_{p}(\u_{p,\epsilon} , \v_{p}) - a_{f}(\u_{f,\epsilon},\v_{f}) - a^d_{p}(\u_{p,\epsilon},\v_{p}) 
		- a_{BJS}(\u_{f,\epsilon}, \u_{s,\epsilon};\v_{f},{\bf 0}) +  (\f_{f},\v_{f})_{\O_f} }{\|(\v_p,{\bf 0},\v_f)\|_{\Q}}  \nonumber \\
&\qquad
\leq C\left( \epsilon\|\nabla \cdot \u_{p,\epsilon}\|^{r/r'}_{L^r(\O_p)} +
	\|\u_{f,\epsilon}\|^{r/r'}_{W^{1,r}(\O_f)} + \|\u_{p,\epsilon}\|^{r/r'}_{L^{r}(\O_p)} + |\u_{f,\epsilon}-\u_{s,\epsilon}|^{r/r'}_{BJS} +\|\f_f\|_{W^{-1,r'}(\O_f)}\right). \label{nhgt1}
\end{align}
Using \eqref{regularized-estimate-3}, \eqref{young}, and \eqref{nhgt1},
we obtain
\begin{align}
\|\u_{s,\epsilon}\|^2_{H^1(\O_p)} & + \epsilon\|\nabla \cdot \u_{p,\epsilon}\|^{r}_{L^r(\O_p)} + \|\u_{f,\epsilon}\|^r_{W^{1,r}(\O_f)} +\|\u_{p,\epsilon}\|^r_{L^r(\O_p)}  +\|\bs_{e,\epsilon}\|^2_{L^2(\O_p)}+|\u_{f,\epsilon}-\u_{s,\epsilon}|^r_{BJS} \nonumber \\
	&+\|p_{f,\epsilon}\|^{r'}_{L^{r'}(\O_f)}+\|p_{p,\epsilon}\|^{r'}_{L^{r'}(\O_p)}+\|\lam_{\epsilon}\|^{r'}_{W^{1/r,r'}(\Gam_{fp})}  \nonumber \\
	&\leq C\left( \|\f_p\|^2_{H^{-1}(\O_p)} + \|\f_f\|^{r'}_{W^{-1,r'}(\O_f)}  + \|\bar{g}_p\|^r_{L^r(\O_p)} + \|\bar{g}_e\|^2_{L^2(\O_p)}+\|q_f\|^r_{L^r(\O_f)}\right) ,\label{regularized-estimate-4}
\end{align}
which implies that $\|\u_{s,\epsilon}\|_{H^1(\O_p)},\, \|\u_{f,\epsilon}\|_{W^{1,r}(\O_f)},\, \|\bs_{e,\epsilon}\|_{L^2(\O_p)},\,\|p_{f,\epsilon}\|_{L^{r'}(\O_f)},\,\|p_{p,\epsilon}\|_{L^{r'}(\O_p)}$ and 
	$\|\lam_{\epsilon}\|_{W^{1/r,r'}(\Gam_{fp})}$ are bounded independently of $\epsilon$.
	
	Also, as $\nabla \cdot \V_p =(W_p)'$, we have from \eqref{regularized-problem-2},
\eqref{h-cts-alt-dom-2}, and the continuity of $L_p$ stated in Lemma~\ref{lmapropRL}:
\begin{align*}
\|\nabla \cdot \u_{p,\epsilon}\|_{L^r(\O_p)} &
\leq s_0\|\bar{g}_p\|_{L^{r}(\O_p)} + s_0\|p_{p,\epsilon}\|_{L^{r}(\O_p)} +
\alpha_p\|\nabla \cdot \u_{s,\epsilon}\|_{L^r(\O_p)} 
+ \epsilon\|p_{p,\epsilon}\|_{L^{r'}(\O_p)}\\
& \leq s_0\|\bar{g}_p\|_{L^{r}(\O_p)} + s_0\|p_{p,\epsilon}\|_{L^{r'}(\O_p)} 
+\alpha_p\|\u_{s,\epsilon}\|_{H^1(\O_p)} + \epsilon\|p_{p,\epsilon}\|_{L^{r'}(\O_p)}  .
\end{align*}
Therefore $\|\u_{p,\epsilon}\|_{L^r(\text{div};\O_p)}$ is also bounded independently of $\epsilon$.
	
Since $\Q$ and $S$ are reflexive Banach spaces, as $\epsilon
\rightarrow 0$ we can extract weakly convergent subsequences $\{
\q_{\epsilon , n} \}_{n = 1}^{\infty}$, $\{ s_{\epsilon , n} \}_{n =
  1}^{\infty}$, and $\{ \Ac\q_{\epsilon , n} \}_{n = 1}^{\infty}$,
such that $\q_{\epsilon , n} \rightharpoonup \q \mbox{ in } \Q$, 
$s_{\epsilon , n} \rightharpoonup s \mbox{ in } S$, 
$\Ac\q_{\epsilon , n} \rightharpoonup  \zeta \mbox{ in } \Q'$, and
\begin{align*}
	\zeta &+ \Bc's=\f   \quad \mbox{in } \Q', \\
	\Ec_2 s&- \Bc\q = \bar g \quad \mbox{in } S'. 
\end{align*}
Moreover, from \eqref{regularized-problem-1}--\eqref{regularized-problem-2} we have
\begin{align*}
\lim\sup\limits_{\epsilon\rightarrow 0}\left( \mathcal{A}(\q_{\epsilon})(\q_{\epsilon}) +\Ec_2(s_{\epsilon})(s_{\epsilon})\right) &=  \lim\sup\limits_{\epsilon\rightarrow 0}(-\epsilon \mathcal{R}(\q_{\epsilon}) (\q_{\epsilon}) - \epsilon \mathcal{L}(s_{\epsilon})(s_{\epsilon})+\f(\q_{\epsilon})+ \bar g(s_{\epsilon})) \\
& \leq \f(\q) + \bar g(s) = \zeta(\q) + \Ec_2(s)(s).
\end{align*}
Since $\Ac+\Ec_2$ is monotone and continuous, it follows, see 
\cite[p. 38]{showalter2013monotone}, that $\Ac \q = \zeta$. Hence, $\q$ and
$s$ solve \eqref{h-cts-alt-dom-1}--\eqref{h-cts-alt-dom-gamma}, which
establishes that $D$ is nonempty.
\end{proof}

\begin{corollary}  \label{corONTO}
For $\mathcal{M}$ defined by \eqref{defmcM} we have that $Rg(I + \mathcal{M}) \, = \, W_{p , 2}' \times \bS_{e , 2}'$.
\end{corollary}
\begin{proof}
To show $Rg(I + \mathcal{M}) \, = \, W_{p , 2}' \times \bS_{e , 2}'$ we need to show that for 
$\f \in W_{p , 2}' \times \bS_{e , 2}'$ there is a $\v \in D$ such that $\f \in (I + \mathcal{M})(\v)$. 

Let $(\bar{g}_p,\bar{g}_e) \in W_{p,2}^{\prime}\times \bS_{e,2}^{\prime}$ be given. Lemma~\ref{thmDdef} 
establishes that there exists $(\tilde p_p , \tilde \bs_e) \in D$ such that 
\eqref{h-cts-alt-dom-1}--\eqref{h-cts-alt-dom-gamma} are satisfied. Hence
$(\bar{g}_p - \tilde p_p \, , \, \bar{g}_e -  \tilde \bs_e) \in
\mathcal{M}(\tilde p_p , \tilde \bs_e)$
and therefore it immediately follows that
$(\bar{g}_p \, , \, \bar{g}_e) \in  (I + \mathcal{M}) (\tilde p_p , \tilde \bs_e)$. 
\end{proof}

\subsubsection{Step 2: Solvability of the parabolic problem \eqref{ode-problem}}
\label{sssec_two}
In this section we establish the existence of a solution to
\eqref{ode-problem}. We begin by showing that $\mathcal{M}$ defined by
\eqref{defmcM} is a monotone operator.

\begin{lemma} \label{lmaLmte}
The operator $\mathcal{M}$ defined by \eqref{ode-problem} is monotone.
\end{lemma}
\begin{proof}
To show that $\mathcal{M}$ is monotone we need to show for $\f \in \mathcal{M}(\v)$,
$\tilde{\f} \in \mathcal{M}(\tilde{\v})$ that $( \f \, - \, \tilde{\f} \, , \, \v \, - \, \tilde{\v} )_{S_{2}} \ge 0$. 

For $(p_{p},\bs_e) \in D$, $(\bar{g}_p -  p_p \, , \, \bar{g}_e -  \bs_e) \in \mathcal{M}(p_{p},\bs_e)$
and $(w_{p},\bt_e)\in S_2$, 
 we have from \eqref{h-cts-alt-dom-2}
	\begin{align}
	\left( (\bar{g}_p - p_p \, , \, \bar{g}_e -  \bs_e) \ , \ 
	(w_{p},\bt_e) \right)_{S_2}
& = (s_0  \bar{g}_p, w_p)+(A \bar{g}_e, \bt_e)-  (s_0 p_p,w_p)- a^s_p(\bs_e,\bt_e)  \nonumber \\
& =- \alpha_p b_p\left(\u_{s},w_{p}\right) - b_p(\u_{p},w_{p}) -b_{s}(\u_s,\bt_e).  \label{ippp1}
	\end{align}
Also, from \eqref{h-cts-alt-dom-1}--\eqref{h-cts-alt-dom-gamma},
	the corresponding $(\u_f,p_f,\u_p,\u_s,\lam)$ satisfy
\begin{align}
	&  
	a_{f}(\u_{f},\v_{f}) + a^d_{p}(\u_{p},\v_{p}) 
	+ a_{BJS}(\u_{f}, \u_{s};\v_{f},\v_{s}) + b_f(\v_{f},p_{f})+ b_p(\v_{p},p_{p})\nonumber
	\\
	& \qquad\quad   + 
	\alpha_p b_p(\v_{s},p_{p}) +b_{s}(\v_s,\bs_e)+ b_{\Gamma}(\v_{f},\v_{p},\v_{s};\lam)  = (\f_{f},\v_{f})_{\O_f} + (\f_{p},\v_{s})_{\O_p}, \label{p-1}\\
	& \left( s_0 p_{p},w_{p}\right)_{\O_p} + a^s_{p}(\bs_{e},\bt_{e})
	- \alpha_p b_p\left(\u_{s},w_{p}\right) - b_p(\u_{p},w_{p}) -b_{s}(\u_s,\bt_e)- b_f(\u_{f},w_{f}) \nonumber
	\\
	& \qquad\quad 
	=  (s_0\bar{g}_{p},w_{p})_{\O_p}+ (A\bar{g}_{e},\bt_{e})_{\O_p} +(q_f,w_f)_{\O_f}, \label{p-2} \\
	& b_{\Gamma}\left(\u_{f},\u_{p},\u_{s};\mu\right) = 0 , \label{p-gamma} 
\end{align}	
Next, for $(\tilde{g}_p -  \tilde{p}_p \, , \, \tilde{g}_e -  \tilde{\bs}_e) \in \mathcal{M}(\tilde{p}_{p} , \tilde{\bs}_e)$ 
the corresponding $(\tilde{\u}_f,\tilde{p}_f,\tilde{\u}_p, \tilde{\u}_s,\tilde{\lam})$	satisfy
\begin{align}
	&  
	a_{f}(\tilde{\u}_{f},\v_{f}) + a^d_{p}(\tilde{\u}_{p},\v_{p}) 
	+ a_{BJS}(\tilde{\u}_{f}, \tilde{\u}_{s};\v_{f},\v_{s}) + b_f(\v_{f},\tilde{p}_{f})+ b_p(\v_{p},\tilde{p}_{p})\nonumber
	\\
	& \qquad\quad   + 
	\alpha_p b_p(\v_{s},\tilde{p}_{p}) +b_{s}(\v_s,\tilde{\bs}_e)+ b_{\Gamma}(\v_{f},\v_{p},\v_{s};\tilde{\lam})  = (\f_{f},\v_{f})_{\O_f} + (\f_{p},\v_{s})_{\O_p}, \label{p-tilde-1}\\
	& \left( s_0 \tilde{p}_{p},w_{p}\right)_{\O_p} +  a^s_{p}(\tilde{\bs}_{e},\bt_{e})
	- \alpha_p b_p\left(\tilde{\u}_{s},w_{p}\right) - b_p(\tilde{\u}_{p},w_{p}) -b_{s}(\tilde{\u}_s,\bt_e)- b_f(\tilde{\u}_{f},w_{f}) \nonumber
	\\
	& \qquad\quad 
	=  (s_0\tilde{g}_{p},w_{p})_{\O_p}+ (A\tilde{g}_{e},\bt_{e})_{\O_p} +(q_f,w_f)_{\O_f}, \label{p-tilde-2} \\
	& b_{\Gamma}\left(\tilde{\u}_{f},\tilde{\u}_{p},\tilde{\u}_{s};\mu\right) = 0. \label{p-tilde-gamma} 
\end{align}
With the association $\v = (p_{p},\bs_e)$, $\tilde{\v} = (\tilde{p}_{p} , \tilde{\bs}_e)$, $\f = (\bar{g}_p -  p_p \, , \, \bar{g}_e -  \bs_e)$,
$\tilde{\f} = (\tilde{g}_p -  \tilde{p}_p \, , \, \tilde{g}_e -  \tilde{\bs}_e)$, using \eqref{ippp1}
\begin{align*}
( \f \, - \, \tilde{\f} \, , \, \v \, - \, \tilde{\v} )_{S_{2}}  = 
& - \alpha_p b_p\left(\u_{s},p_{p}-\tilde{p}_p\right) - b_p(\u_{p},p_{p}-\tilde{p}_p) -b_{s}(\u_s,\bs_e-\tilde{\bs}_e)\\
	& + \alpha_p b_p\left(\tilde{\u}_{s},p_{p}-\tilde{p}_p\right) +b_p(\tilde{\u}_{p},p_{p}-\tilde{p}_p) +b_{s}(\tilde{\u}_s,\bs_e-\tilde{\bs}_e).
\end{align*}	
	Testing equation \eqref{p-1} with $(\v_f,\v_p,\v_s) =(\u_f,\u_p,\u_s)$, we obtain
	\begin{align*}
	&  
	a_{f}(\u_{f},\u_{f}) + a^d_{p}(\u_{p},\u_{p}) 
	+ a_{BJS}(\u_{f}, \u_{s};\u_{f},\u_{s}) + b_f(\u_{f},p_{f})+ b_p(\u_{p},p_{p})\nonumber
	\\
	& \qquad\quad   + 
	\alpha_p b_p(\u_{s},p_{p}) +b_{s}(\u_s,\bs_e)+ b_{\Gamma}(\u_{f},\u_{p},\u_{s};\lam)  = (\f_{f},\u_{f})_{\O_f} + (\f_{p},\u_{s})_{\O_p}.
	\end{align*}
	On the other hand, choosing $w_f=p_f$ and $\mu=\lam$ in \eqref{p-2} and \eqref{p-gamma}, we get
	\begin{align*}
	-b_f(\u_{f},p_{f})-b_{\Gamma}(\u_{f},\u_{p},\u_{s};\lam)=(q_f,p_f)_{\O_f}.
	\end{align*}
	Hence,
	\begin{align} 
	& a_{f}(\u_{f},\u_{f})+ a^d_{p}(\u_{p},\u_{p}) 
	+ a_{BJS}(\u_{f}, \u_{s};\u_{f},\u_{s}) + b_p(\u_{p},p_{p}) + 
	\alpha_p b_p(\u_{s},p_{p})\nonumber\\
	& \qquad\quad 
+b_{s}(\u_s,\bs_e)=  (\f_{f},\u_{f})_{\O_f} + (\f_{p},\u_{s})_{\O_p}+(q_f,p_f)_{\O_f}. \label{dif-1}
	\end{align}
	Repeating the same argument for problem \eqref{p-tilde-1}--\eqref{p-tilde-gamma}, we obtain
	\begin{align}
	& a_{f}(\tilde{\u}_{f},\tilde{\u}_{f}) + a^d_{p}(\tilde{\u}_{p},\tilde{\u}_{p}) 
	+ a_{BJS}(\tilde{\u}_{f}, \tilde{\u}_{s};\tilde{\u}_{f},\tilde{\u}_{s}) + b_p(\tilde{\u}_{p},\tilde{p}_{p}) + 
	\alpha_p b_p(\tilde{\u}_{s},\tilde{p}_{p}) \nonumber \\
	& \qquad\quad +b_{s}(\tilde{\u}_s,\tilde{\bs}_e)=  (\f_{f},\tilde{\u}_{f})_{\O_f} + (\f_{p},\tilde{\u}_{s})_{\O_p}+(q_f,\tilde{p}_f)_{\O_f}. \label{dif-2}
	\end{align}
	Next, we test \eqref{p-1} with $(\v_f,\v_p,\v_s) =(\tilde{\u}_f,\tilde{\u}_p,\tilde{\u}_s)$:
	\begin{align*}
	&  
	a_{f}(\u_{f},\tilde{\u}_{f}) + a^d_{p}(\u_{p},\tilde{\u}_{p}) 
	+ a_{BJS}(\u_{f}, \u_{s};\tilde{\u}_{f},\tilde{\u}_{s}) + b_f(\tilde{\u}_{f},p_{f})+ b_p(\tilde{\u}_{p},p_{p})\nonumber
	\\
	& \qquad\quad   + 
	\alpha_p b_p(\tilde{\u}_{s},p_{p}) +b_{s}(\tilde{\u}_s,\bs_e)+ b_{\Gamma}(\tilde{\u}_{f},\tilde{\u}_{p},\tilde{\u}_{s};\lam)  = (\f_{f},\tilde{\u}_{f})_{\O_f} + (\f_{p},\tilde{\u}_{s})_{\O_p}.
	\end{align*}
	Choosing $w_f=p_f$ and $\mu=\lam$ in \eqref{p-tilde-2}--\eqref{p-tilde-gamma}, we conclude that
	\begin{align*}
	-b_f(\tilde{\u}_f,p_f)-b_{\Gamma}(\tilde{\u}_{f},\tilde{\u}_{p},\tilde{\u}_{s};\lam) =(q_f,p_f)_{\O_f},
	\end{align*}
	which implies that
	\begin{align}
	& a_{f}(\u_{f},\tilde{\u}_{f}) + a^d_{p}(\u_{p},\tilde{\u}_{p}) 
	+ a_{BJS}(\u_{f}, \u_{s};\tilde{\u}_{f},\tilde{\u}_{s})+ b_p(\tilde{\u}_{p},p_{p}) + 
	\alpha_p b_p(\tilde{\u}_{s},p_{p}) \nonumber \\
	&\qquad\quad +b_{s}(\tilde{\u}_s,\bs_e)=  (\f_{f},\tilde{\u}_{f})_{\O_f} + (\f_{p},\tilde{\u}_{s})_{\O_p}+(q_f,p_f)_{\O_f}. \label{dif-3}
	\end{align}
	Similarly,
	\begin{align} 
	& a_{f}(\tilde{\u}_{f},\u_{f}) + a^d_{p}(\tilde{\u}_{p},\u_{p}) 
	+ a_{BJS}(\tilde{\u}_{f}, \tilde{\u}_{s};\u_{f},\u_{s})+ b_p(\u_{p},\tilde{p}_{p}) + 
	\alpha_p b_p(\u_{s},\tilde{p}_{p}) \nonumber \\
	&\qquad\quad +b_{s}(\u_s,\tilde{\bs}_e)=  (\f_{f},\u_{f})_{\O_f} + (\f_{p},\u_{s})_{\O_p}+(q_f,\tilde{p}_f)_{\O_f}. \label{dif-4}
	\end{align}
	Manipulating \eqref{dif-1}--\eqref{dif-4}, we finally obtain
	\begin{align*}
	\left(  \f \, - \, \tilde{\f} \, , \, \v \, - \, \tilde{\v} \right)_{S_2}  &=  a_{f}(\u_{f},\u_{f}) + a^d_{p}(\u_{p},\u_{p}) 
	+ a_{BJS}(\u_{f}, \u_{s};\u_{f},\u_{s}) \\
	&-a_{f}(\tilde{\u}_{f},\u_{f}) - a^d_{p}(\tilde{\u}_{p},\u_{p}) 
	- a_{BJS}(\tilde{\u}_{f}, \tilde{\u}_{s};\u_{f},\u_{s}) \\
	&- a_{f}(\u_{f},\tilde{\u}_{f}) - a^d_{p}(\u_{p},\tilde{\u}_{p}) 
	- a_{BJS}(\u_{f}, \u_{s};\tilde{\u}_{f},\tilde{\u}_{s}) \\
	& +a_{f}(\tilde{\u}_{f},\tilde{\u}_{f}) + a^d_{p}(\tilde{\u}_{p},\tilde{\u}_{p}) 
	+ a_{BJS}(\tilde{\u}_{f}, \tilde{\u}_{s};\tilde{\u}_{f},\tilde{\u}_{s}) \\
	= a_{f}(\u_{f},\u_{f}-\tilde{\u}_f) &+ a^d_{p}(\u_{p},\u_{p}-\tilde{\u}_p) 
	+ a_{BJS}(\u_{f}, \u_{s};\u_{f}-\tilde{\u}_f,\u_{s}-\tilde{\u}_s) \\
	-a_{f}(\tilde{\u}_{f},\u_{f}-\tilde{\u}_f) & -a^d_{p}(\tilde{\u}_{p},\u_{p}-\tilde{\u}_p) 
	- a_{BJS}(\tilde{\u}_{f}, \tilde{\u}_{s};\u_{f}-\tilde{\u}_f,\u_{s}-\tilde{\u}_s) \geq 0,
	\end{align*}
	due to the monotonicity of $a_f(\cdot,\cdot),\, a_p^d(\cdot,\cdot)$ and $a_{BJS}(\cdot,\cdot;\cdot,\cdot)$.
	
\end{proof}

\begin{lemma}  \label{lmaExode}
For each $h_p \in W^{1,1}(0, T; W_{p,2}')$, $h_e \in W^{1,1}(0, T; \bS_{e,2}')$, 
and $p_p(0) \in W_p$, $\bs_e(0) \in \bS_e$,
there exists a solution to \eqref{ode-problem} with $p_p \in W^{1,\infty}(0,T;W_p)$ and
$\bs_e \in W^{1,\infty}(0,T;\bS_e)$.
\end{lemma}
\begin{proof}
Applying Theorem \ref{thmsho61b} with $\mathcal{N} = I$, 
$\mathcal{M} = \mathcal{M}$, 
$E = W_{p , 2} \times \bS_{e , 2}$, $E_{b}' = W'_{p , 2} \times
\bS'_{e , 2}$, and using Lemma~\ref{lmaLmte} and Corollary
\ref{corONTO}, we obtain existence of a solution to
\eqref{ode-problem}.
\end{proof}
%

\subsubsection{Step 3: The original problem 
\eqref{h-cts-alt-1}--\eqref{h-cts-alt-gamma} is a special case of \eqref{ode-problem}}
\label{sssec_three}

Finally, we establish the existence of a solution to \eqref{h-cts-alt-1}--\eqref{h-cts-alt-gamma}
as a corollary of Lemma \ref{lmaExode}.

\begin{lemma}\label{lem:equiv}
If $(p_p(t),\bs_e(t))\in D$ solves \eqref{ode-problem} 
for $h_{p} = s_0^{-1} q_p$ and $h_{e} = 0$, then it
also solves \eqref{h-cts-alt-1}--\eqref{h-cts-alt-gamma}.
\end{lemma}
\begin{proof}
  Let $(p_p(t),\bs_e(t))\in D$ solve \eqref{ode-problem} for $h_{p} = s_0^{-1} q_p$ and
  $h_{e} = 0$.
Note that
\eqref{h-cts-alt-dom-1} and \eqref{h-cts-alt-dom-gamma} from the
definition of the domain $D$ directly imply \eqref{h-cts-alt-1}
and \eqref{h-cts-alt-gamma}. Also, \eqref{h-cts-alt-dom-2} and
\eqref{h-cts-alt-2} are the same when tested only with $w_f$. Thus it
remains to show \eqref{h-cts-alt-2} with $w_f = 0$.  

Since $(p_p(t),\bs_e(t))$ solve \eqref{ode-problem} for $h_{p} = s_0^{-1} q_p$ and
$h_{e} = 0$, there exist $(\bar{g}_p,\bar{g}_e) \in W_{p,2}^{\prime}\times \bS_{e,2}^{\prime}$
such that
$(\bar{g}_p -  p_p \, , \, \bar{g}_e -  \bs_e) \in \mathcal{M}(p_{p},\bs_e)$ satisfy
\[ 
	\frac{d}{dt} \begin{pmatrix}
	p_p\\
	\bs_e
	\end{pmatrix} \  + \ \begin{pmatrix}
	\bar{g}_p -  p_p \\
	\bar{g}_e -  \bs_e
	\end{pmatrix}  =  \begin{pmatrix}
	s_0^{-1} q_p \\
	0
	\end{pmatrix}  \, .
\]
Then,
\begin{equation}
	\left(\frac{d}{dt} \begin{pmatrix}
	p_p\\
	\bs_e
	\end{pmatrix}, \begin{pmatrix}
	w_p\\
	\bt_e
	\end{pmatrix}\right)_{S_2} +\left(\begin{pmatrix}
	\bar{g}_p -  p_p \\
	\bar{g}_e -  \bs_e
	\end{pmatrix}, \begin{pmatrix}
	w_p\\
	\bt_e
	\end{pmatrix}\right)_{S_2} = \left(\begin{pmatrix}
	s_0^{-1} q_p \\
	0
	\end{pmatrix}, \begin{pmatrix}
	w_p\\
	\bt_e
	\end{pmatrix}\right)_{S_2}
		= (q_p,w_p) ,
\label{prte11}
\end{equation}
and,  using \eqref{ippp1}, \eqref{prte11} becomes
\begin{align*}
	(s_0\d_t p_p,w_p) +a^s_p(\d_t \bs_e,\bt_e) - \alpha_p b_p\left(\u_{s},w_{p}\right) - b_p(\u_{p},w_{p}) -b_{s}(\u_s,\bt_e) = (q_p,w_p),
\end{align*}
which is \eqref{h-cts-alt-2} with $w_f = 0$. 
\end{proof}

\begin{proof}[Proof of Theorem~\ref{well-pos-alt}]
Existence of a solution of \eqref{h-cts-alt-1}--\eqref{h-cts-alt-gamma}
follows from Lemma~\ref{lmaExode} and Lemma~\ref{lem:equiv}. From Lemma~\ref{lmaExode}
we have that $p_p \in W^{1,\infty}(0,T;W_p)$ and $\bs_e \in W^{1,\infty}(0,T;\bS_e)$.
By taking \linebreak $(\v_{f}, w_{f},\v_{p},w_{p}, \v_s, \btau_e, \mu) 
= (\u_{f}, p_{f},\u_{p}, p_{p}, \u_s, \bs_e, \lambda)$ in
\eqref{h-cts-alt-1}--\eqref{h-cts-alt-gamma}, we obtain that $\u_f \in L^{\infty}(0,T;\V_f)$
and $\u_p \in L^{\infty}(0,T;\V_p)$. The inf-sup condition \eqref{inf-sup-elasticity} and
\eqref{h-cts-alt-2} imply that
$\u_s \in L^{\infty}(0,T;\X_p)$, while the inf-sup condition \eqref{inf-sup-stokes-darcy}
and \eqref{h-cts-alt-1}
imply that $p_f \in L^{\infty}(0,T;W_f)$ and $\lambda \in L^{\infty}(0,T;\Lambda)$.
\end{proof}

\subsection{Existence and uniqueness of solution of the original formulation}
In this section we discuss how the well-posedness of the original
formulation \eqref{h-cts-1}--\eqref{h-cts-gamma} follows from the
existence of a solution of the alternative formulation
\eqref{h-cts-alt-1}--\eqref{h-cts-alt-gamma}. Recall that $\u_s$ is
the structure velocity, so the displacement solution can be recovered from
\begin{align}
\bbeta_p(t) = \bbeta_{p,0} +\int_0^t \u_s(s)\, ds, && \forall t\in (0,T]. \label{disp-solution}
\end{align} 
Since $\u_s(t) \in L^{\infty}(0,T;\X_p)$, then $\bbeta_p(t) \in
W^{1,\infty}(0,T;\X_p)$ for any $\bbeta_{p,0} \in \X_p$. By construction, 
$\u_s = \d_t \bbeta_p$ and $\bbeta_p(0) = \bbeta_{p,0}$.

\begin{theorem}\label{th-exist-uniq-weak-2}
For each $\f_{f} \in
  W^{1,1}(0, T; \V_{f}')$, $\f_{p} \in W^{1,1}(0, T; \X_{p}')$, $q_f
  \in W^{1,1}(0, T; W_{f}')$, \linebreak $q_p \in W^{1,1}(0, T; W_{p}')$, and
  $p_p(0) = p_{p,0} \in W_p$, $\bbeta_p(0) = \bbeta_{p,0} \in \X_p$,
there exists a
unique solution $(\u_f(t), p_f(t),\u_p(t),p_p(t),\bbeta_p(t),\lam(t))$
$\in L^{\infty}(0,T;\V_f)\times L^{\infty}(0,T;W_f)\times
L^{\infty}(0,T;\V_p) \times W^{1,\infty}(0,T;W_p) \times
W^{1,\infty}(0,T;\X_p) \times L^{\infty}(0,T;\Lambda)$ of
\eqref{h-cts-1}--\eqref{h-cts-gamma}.
\end{theorem}
\begin{proof}
We begin by using the existence of a solution of the alternative
formulation \eqref{h-cts-alt-1}--\eqref{h-cts-alt-gamma} to establish
solvability of the original formulation
\eqref{h-cts-1}--\eqref{h-cts-gamma}. Let $(\u_f$, $p_f$, $\u_p$,
$p_p$, $\u_s$, $\bs_e$, $\lam)$ be a solution to
\eqref{h-cts-alt-1}--\eqref{h-cts-alt-gamma}. Let $\bbeta_p$ be
defined in \eqref{disp-solution}, so $\u_s = \d_t \bbeta_p$. Then
\eqref{h-cts-alt-2} with $\btau_e = {\bf 0}$ implies \eqref{h-cts-2}
and \eqref{h-cts-alt-gamma} implies \eqref{h-cts-gamma}. We further
note that \eqref{h-cts-1} and \eqref{h-cts-alt-1} differ only in their
respective terms $a^e_{p}(\bbeta_{p},\bxi_{p})$ and
$b_{s}(\v_s,\bs_e)$.  Testing \eqref{h-cts-alt-2} with $\btau_e \in
\bS_e$ gives $(\d_t(A\bs_e-\D(\bbeta_p)),\btau_e)_{\O_p} = 0$, which,
using that $\D(\X_p) \subset \bS_e$, implies that
$\d_t(A\bs_e-\D(\bbeta_p)) = {\bf 0}$.  Integrating from 0 to $t \in
(0,T]$ and using that $\bs_e(0) = A^{-1}\D(\bbeta_p(0))$ implies that
  $\bs_e(t) = A^{-1}\D(\bbeta_p(t))$. Therefore, with \eqref{defAten},
\begin{align*}
b_{s}(\v_s,\bs_e) = (\bs_e, \D(\v_s))_{\O_p} 
= (A^{-1}\D(\bbeta_p), \D(\v_s))_{\O_p} = a^e_p(\bbeta_p,\v_s).
\end{align*}
Therefore \eqref{h-cts-1} implies \eqref{h-cts-alt-1}, which establishes
that $(\u_f,p_f,\u_p,p_p,\bbeta_{p,0} +\int_0^t \u_s(s)\, ds,\lam)$ is a solution
of \eqref{h-cts-1}--\eqref{h-cts-gamma}. The stated regularity of the solution
follows from the established regularity in Theorem~\ref{well-pos-alt}.
	
Now, assume that the solution of \eqref{h-cts-1}--\eqref{h-cts-gamma}
is not unique. Let
$(\u^i_f,p^i_f,\u^i_p,p^i_p,\bbeta^i_p,\lam^i)$, $i=1,2$, be two solutions
corresponding to the same data. Using the monotonicity property
\eqref{monotonicity} with $G(\x)= \nu(\x)\x$, $\s = \D(\u^1_f)$ and
$\t = \D(\u^2_f)$, we have
\begin{align}
C\frac{\| \D (\u^1_f) - \D (\u^2_f) \|^2_{L^r(\O_f)}}
{c + \| \D ( \u^1_f ) \|^{2-r}_{L^r(\O_f)} + \| \D (\u^2_f) \|^{2-r}_{L^r(\O_f)}}
& \leq  \left( 2\nu (\D(\u^1_f)) \D(\u^1_f) - 2\nu (\D(\u^2_f)) \D(\u^2_f), \D(\u^1_f) - \D(\u^2_f)\right)_{\O_f} \nonumber \\
& =  
\left( a_{f}(\u^1_f , \u^1_f - \u^2_f) - a_{f}(\u^2_f , \u^1_f - \u^2_f) \right) 
=:  I_1. \label{stokes-dif-1}
\end{align}
Similarly, we use \eqref{monotonicity} with $G(\x)=\nu_{eff}(\x)\x$, $\s=\u^1_p$ and 
$\t=\u^2_p$, to obtain
\begin{align} 
C\frac{\| \u^1_p - \u^2_p \|^2_{L^r(\O_p)}}
{c + \| \u^1_p \|^{2-r}_{L^r(\O_p)} + \| \u^2_p \|^{2-r}_{L^r(\O_p)}} 
& \leq (\kappa^{-1}  (\nu_{eff} (\u^1_p ) \u^1_p - \nu_{eff} (\u^2_p ) \u^2_p )
,\u^1_p - \u^2_p )_{\O_p} \nonumber \\
& = \ a_{p}^{d}(\u^1_f ,  \u^1_f - \u^2_f) - a_{p}^{d}(\u^2_f ,  \u^1_f - \u^2_f) 
=: I_2. \label{darcy-dif-1}
\end{align}
We apply \eqref{monotonicity} one more time to bound the terms coming
from BJS condition. Set $G(\x) = \nu_{I}(\x)\x$,
$\s=((\u^1_f-\d_t\bbeta^1_p)\cdot \t_{f,j})\t_{f,j}$ and
$\t=((\u^2_f-\d_t\bbeta^2_p)\cdot \t_{f,j})\t_{f,j}$, then
\begin{align} 
& \alpha_{BJS}C \sum_{j=1}^{d-1}\frac{\| (\u^1_f - \d_t \bbeta^1_p)
\cdot \t_{f,j} - (\u^2_f - \d_t \bbeta^2_p)\cdot \t_{f,j} \|^2_{L^r(\Gamma_{fp})}}
{c + \| (\u^1_f - \d_t \bbeta^1_p)\cdot \t_{f,j} \|^{2-r}_{L^r(\Gamma_{fp}) } + \| (\u^2_f - \d_t \bbeta^2_p)\cdot \t_{f,j} \|_{\Gamma_{fp}}^{2-r}} \nonumber \\
& \qquad\quad 
\leq  a_{BJS} (\u^1_f, \d_t \bbeta^1_p; \u^1_f - \u^2_f, \d_t \bbeta^1_p -\d_t \bbeta^2_p) -  a_{BJS} (\u^2_f, \d_t \bbeta^2_p;\u^1_f - \u^2_f, \d_t \bbeta^1_p -\d_t \bbeta^2_p)  =:I_3.
\label{bjs-dif-1}
\end{align}
From \eqref{h-cts-1} we have
	\begin{align}
	I_1+I_2+I_3+a^e_p(\bbeta^1_p- \bbeta^2_p,\d_t\bbeta^1_p- \d_t\bbeta^2_p) = -b_f(\u^1_f-\u^2_f,p^1_f-p^2_f) -b_p(\u^1_p-\u^2_p,p^1_p-p^2_p) \nonumber\\
	-\alpha_p b_p(\d_t\bbeta^1_p- \d_t\bbeta^2_p,p^1_p-p^2_p)-b_{\Gam}(\u^1_f-\u^2_f,\u^1_p-\u^2_p,\d_t\bbeta^1_p- \d_t\bbeta^2_p;\lam^1-\lam^2). \label{eqdif1}
	\end{align}
	 On the other hand, it follows from \eqref{h-cts-2} and \eqref{h-cts-gamma}, with $w_f =p^1_f-p^2_f,\,w_p =p^1_p-p^2_p,\, \mu =\lam^1-\lam^2$, that
	 \begin{align}
	 ( s_0 \, \d_t\left( p^1_p-p^2_p\right),p^1_p-p^2_p) 
	 - \alpha_p b_p(\d_t \left(\bbeta^1_p-\bbeta^2_p\right),p^1_p-p^2_p) - b_p(\u^1_p-\u^2_p,p^1_p-p^2_p) \nonumber \\
	 - b_f(\u^1_f-\u^2_f,p^1_f-p^2_f) 
	 -	 b_{\Gamma}(\u^1_f-\u^2_f,\u^1_p-\u^2_p,\d_t \left(\bbeta^1_p-\bbeta^2_p\right);\lam^1-\lam^2) = 0 \, .\label{eqdif2}
	 \end{align}
 Combining \eqref{eqdif1} and \eqref{eqdif2}, we obtain
 \begin{align*}
 I_1+I_2+I_3 +a^e_p(\bbeta^1_p- \bbeta^2_p,\d_t\bbeta^1_p- \d_t\bbeta^2_p) 
 = -( s_0 \, \d_t\left( p^1_p-p^2_p\right),p^1_p-p^2_p),
\end{align*}
which implies
\begin{align*}
 \frac{1}{2}\d_t \left(  a^e_p(\bbeta^1_p-\bbeta^2_p,\bbeta^1_p-\bbeta^2_p) + s_0\|p^1_p-p^2_p\|^2_{L^2(\O_p)} \right) 
 &+ I_1+I_2+I_3  = 0. 
  \end{align*}
  Integrating in time from 0 to $t\in (0,T]$, and using $p_p^1(0)=p_p^2(0), \ \bbeta_p^1(0)=\bbeta^2_p(0)$, we obtain
  \begin{align*}
  \frac{1}{2} \left(  a^e_p(\bbeta^1_p(t)-\bbeta^2_p(t),\bbeta^1_p(t)-\bbeta^2_p(t)) + s_0\|p^1_p(t)-p^2_p(t)\|^2_{L^2(\O_p)} \right) +\int_0^t\left( I_1+I_2+I_3\right)\, ds   = 0. 
  \end{align*} 
  Hence, using \eqref{stokes-dif-1}--\eqref{bjs-dif-1}, we have
   \begin{align}
&   \frac{1}{2} \left(  a^e_p(\bbeta^1_p(t)-\bbeta^2_p(t),\bbeta^1_p(t)-\bbeta^2_p(t)) + s_0\|p^1_p(t)-p^2_p(t)\|^2_{L^2(\O_p)} \right) 
   \nonumber \\
& \qquad   + C \int_0^t \left(\frac{\| \D (\u^1_f) - \D (\u^2_f) \|^2_{L^2(\O_f)}}{c + \| \D ( \u^1_f ) \|^{2-r}_{L^r(\O_f)} + \| \D (\u^2_f) \|^{2-r}_{L^r(\O_f)}} +	\frac{\| \u^1_p - \u^2_p \|^2_{L^r(\O_p)}}{c + \| \u^1_p \|^{2-r}_{L^r(\O_p)} + \| \u^2_p \|^{2-r}_{L^r(\O_p)}} \right)\, ds   \leq 0.  \label{uniq-vel}
   \end{align}
We note that $a_p^e(\cdot,\cdot)$ satisfies the bounds, for some $c_e$, $C_e > 0$, for
all $\bbeta_p$, $\bxi_p \in \X_p$,
\begin{equation}\label{ae-bound}
c_e \|\bxi_p\|_{H^1(\O_p)}^2 \leq a_p^e(\bxi_p,\bxi_p), \quad 
a_p^e(\bbeta_p,\bxi_p) \leq C_e \|\bbeta_p\|_{H^1(\O_p)}\|\bxi_p\|_{H^1(\O_p)},
\end{equation}
where the coercivity bound follows from Korn's inequality. Therefore,
it follows from \eqref{uniq-vel}, together with
the established regularity $\u_f^i \in L^{\infty}(0,T;\V_f)$ and
$\u_p^i \in L^{\infty}(0,T;\V_p)$, that $\u^1_f(t) =\u^2_f(t),
\u^1_p(t)=\u^2_p, \bbeta^1(t) =\bbeta^2_p, \, \forall t\in (0,T]$.
  Finally, we use the inf-sup condition \eqref{inf-sup-stokes-darcy}
  for $p^1_f-p^2_f,p^1_p-p^2_p,\lam^1-\lam^2$ together with
  \eqref{h-cts-1} to obtain
  \begin{align*}
 & \|(p^1_f-p^2_f,p^1_p-p^2_p,\lam^1-\lam^2)\|_{W_f\times W_p\times \Lambda}\\
 & \qquad \leq C\sup_{(\v_f,\v_p)\in \V_f\times \V_p}\frac{b_f(\v_f,p^1_f-p^2_f) 
+b_p(\v_p,p^1_p-p^2) +
  b_{\Gamma}(\v_f,\v_p, \mathbf{0};\lam^1-\lam^2)}{\|(\v_f,\v_p)\|_{\V_f\times \V_p}} \\
& \qquad   = C\sup_{(\v_f,\v_p)\in \V_f\times \V_p} \Big( \frac{a_{f}(\u^2_{f},\v_{f})-a_{f}(\u^1_{f},\v_{f}) 
+a^d_{p}(\u^2_{p},\v_{p})- a^d_{p}(\u^1_{p},\v_{p}) }{\|(\v_f,\v_p)\|_{\V_f\times \V_p}}  \\
& \qquad\qquad\qquad\qquad + \frac{a_{BJS}(\u^2_{f},\d_t \bbeta^2_{p};\v_{f}, \mathbf{0})- a_{BJS}(\u^1_{f},\d_t \bbeta^1_{p};\v_{f}, \mathbf{0}) 
  	}{\|(\v_f,\v_p)\|_{\V_f\times \V_p}}\Big)=0.
  \end{align*}
 Therefore, for all $t\in (0,T]$, $p^1_f = p^2_f, \ p^1_p = p^2_p, \ \lam^1 = \lam^2$, and we can conclude
 that \eqref{h-cts-1}--\eqref{h-cts-gamma} has a unique solution.
\end{proof}

We conclude with a stability bound for the solution of \eqref{h-cts-1}--\eqref{h-cts-gamma}.
\begin{theorem} \label{stability}
For the solution of \eqref{h-cts-1}--\eqref{h-cts-gamma}, 
assuming sufficient regularity of the data,
there exists $C>0$ such that
\begin{align*}
& \| \u_{f} \|_{L^r(0,T; W^{1,r}(\O_f))}^r + \| \u_{p} \|_{L^r(0,T; L^{r}(\Omega_p))}^r 
+ | \u_{f} - \d_t \bbeta_{p} |^r_{L^r(0,T;BJS)} + \| p_{f} \|^{r'}_{L^{r'} (0,T; L^{r'}(\O_f))} \\
& \qquad +  \| p_{p} \|^{r'}_{L^{r'} (0,T; L^{r'}(\O_p))} 
+ \| \lam \|^{r'}_{L^{r'} (0,T; W^{1/r,r'}(\Gamma_{fp}))}
+ \| \bbeta_{p} \|^2_{L^\infty (0,T; H^1({\Omega_p}))} 
+ s_0\| p_{p} \|^2_{L^\infty (0,T; L^2(\O_p))}\\
& \quad \leq 
C \exp(T) \Big( \| \f_{p}\|^{2}_{L^\infty (0,T; H^{-1}(\Omega_p))}  +\| \bbeta_{p} (0)\|^2_{H^1(\Omega_p)} 
	+ s_0 \| p_{p}(0) \|_{L_2(\Omega_p)}^2	
+ \| \d_t \f_{p} \|_{L^2(0,T;H^{-1}(\Omega_p))}^2 
\\
& \qquad 
+ \| \f_{f} \|^{r'}_{L^{r'}(0,T;W^{-1,r'}(\Omega_f))} 
	+ \| q_{f} \|^{r}_{L^{r}(0,T;L^{r}(\Omega_f))}  + \| q_{p} \|^{r}_{L^{r}(0,T;L^{r}(\Omega_f))} 
+ c( \bar c_f + \bar c_p + \bar c_I)
\Big).
	\end{align*}
\end{theorem}

\begin{proof}
We first note that the term $c( \bar c_f + \bar c_p + \bar c_I)$ appears due to the use of
the coercivity bounds in \eqref{a-f-bounds}--\eqref{a-bjs-bounds} in the general case $c > 0$. 
For simplicity, we present the proof for $c = 0$, noting that the extra term appears in 
\eqref{energy-ineq-1} and the last inequality in the proof. We choose 
$(\v_{f}, w_{f},\v_{p},w_{p}, \bxi_{p},\mu) 
= (\u_{f}, p_{f},\u_{p}, p_{p},\d_t \bbeta_{p}, \lambda)$ 
in \eqref{h-cts-1}--\eqref{h-cts-gamma} to get
\begin{align} 
& \frac{1}{2}\d_t\left[  (s_0  p_{p}, p_{p})_{\O_p} + a_p^e (\bbeta_{p},  \bbeta_{p}) \right] 
+ a_f(\u_{f}, \u_{f}) + a_p^d(\u_{p}, \u_{p}) 
+ a_{BJS}(\u_{f},  \d_t\bbeta_{p}; \u_{f}, \d_t \bbeta_{p}) \nonumber \\	
&\qquad\quad
= (\f_{f},\u_{f})_{\O_f} + (\f_{p},\d_t \bbeta_{p})_{\O_p} 
+ (q_{f},p_{f})_{\O_f} + (q_{p},p_{p})_{\O_p}.\label{energy-eq-1}
	\end{align}
Next, we integrate \eqref{energy-eq-1} from 0 to $t\in (0,T]$ and use the 
coercivity bounds in \eqref{a-f-bounds}--\eqref{a-bjs-bounds} and \eqref{ae-bound}:
\begin{align}
& 
s_0\| p_{p}(t) \|_{L^2(\Omega_p)}^2 + \| \bbeta_{p} (t)\|^2_{H^1(\O_p)} 
+ \int_0^t \left(\| \u_{f} \|_{W^{1,r}(\O_f)}^r + \| \u_{p} \|_{L^{r}(\Omega_p)}^r 
+ | \u_{f} - \d_t \bbeta_{p} |^r_{BJS}\right)\, ds 
\nonumber \\
& \qquad \leq C\Big( \int_0^t (\f_{f},\u_{f})_{\O_f}\, ds
+ (\f_{p}(t), \bbeta_{p}(t))_{\Omega_p} - (\f_{p}(0), \bbeta_{p}(0))_{\Omega_p}  -\int_0^t(\d_t\f_{p}, \bbeta_{p})_{\O_p}\, ds \nonumber \\ 
& \qquad\quad +  \int_0^t \left(  (q_{f},p_{f})_{\O_f} 
+ (q_{p},p_{p})_{\O_p}\right)\, ds 
+ s_0 \| p_{p}(0) \|_{L_2(\Omega_p)}^2 
+\| \bbeta_{p} (0)\|^2_{H^1(\O_p)}
\Big) \nonumber \\
& \qquad \leq  C\left( \| \f_{p}(0) \|_{H^{-1}(\O_p)}^{2} 
+ \| \bbeta_{p} (0)\|_{H^1(\O_p)}^2 
+ s_0 \| p_{p}(0) \|_{L_2(\Omega_p)}^2
+ \| \f_{p}(t) \|_{H^{-1}(\Omega_p)}^{2} \right) 
\nonumber \\
& \qquad\quad	 + C \int_0^t \left( \| \f_{f} \|_{W^{-1, r'}(\O_f)}^{r'} 
+\| \d_t\f_{p} \|_{H^{-1}(\O_p)}^{2}
+ \| \bbeta_{p} \|^2_{H^1(\O_p)}
+\| q_{f} \|^r_{L^{r}(\Omega_f)} 
+ \| q_{p} \|^r_{L^{r}(\Omega_p)} \right)\, ds \nonumber\\ 
& \qquad\quad 
+ \epsilon_1\| \bbeta_{p} (t)\|^2_{H^1(\O_p)} 
+ \epsilon_1\int_0^t \left( \| \u_{f} \|_{W^{1,r}(\O_f)}^r
+ \| p_{f} \|^{r'}_{L^{r'}(\O_f)} 
+ \| p_{p} \|^{r'}_{L^{r'}(\O_p)}\right)\, ds,
\label{energy-ineq-1}
	\end{align}
using Young's inequality \eqref{young} for the last inequality.  We
next apply the inf-sup condition \eqref{inf-sup-stokes-darcy} for
$(p_f, p_p, \lambda)$ to obtain
\begin{align}
& \|(p_f,p_p,\lam)\|_{W_f \times W_p \times \Lambda}
\leq C \sup_{(\v_f,\v_p)\in \V_f\times \V_p} 
\frac{b_f(\v_{f},p_{f}) + b_p(\v_{p},p_{p})  
+ b_{\Gamma}(\v_{f},\v_{p},{\bf 0};\lam)}
{\|(\v_f,\v_p)\|_{\V_f\times \V_p}} \nonumber \\
&\qquad
= C\sup_{(\v_f,\v_p)\in \V_f\times \V_p}
\frac{-a_{f}(\u_{f},\v_{f}) - a^d_{p}(\u_{p},\v_{p})
- a_{BJS}(\u_{f},\d_t \bbeta_{p};\v_{f},0)+(\f_{f},\v_{f})_{\O_f}}
{\|(\v_f,\v_p)\|_{\V_f\times \V_p}}. \label{pressure-stab-1}
\end{align}
Using the continuity bounds in \eqref{a-f-bounds}--\eqref{a-bjs-bounds}, we have from 
\eqref{pressure-stab-1},
	\begin{align*} 
\|(p_f,p_p,\lam)\|_{W_f \times W_p \times \Lambda}
\leq C  \left( \|\f_{f}\|_{W^{-1,r'}(\O_f)} +  \|\u_{f}\|^{r/r'}_{W^{1,r}(\O_f)}
+  \| \u_{p}\|^{r/r'}_{L^{r}(\O_p)}
	+ | \u_{f} - \d_t \bbeta_{p} |^{r/r'}_{BJS} \right),
	\end{align*}
implying
\begin{align}  
& \epsilon_2 \int_0^t \big( \|p_{f}\|^{r'}_{L^{r'}(\O_f)} + \|p_{p}\|^{r'}_{L^{r'}(\O_p)}  
+ 	\|\lam\|^{r'}_{W^{1/r,r'}(\Gamma_{fp})} \big) \, ds \nonumber \\
& \qquad\quad \leq  C \epsilon_2 \int_0^t \left(  \|\f_{f}\|^{r'}_{W^{-1,r'}(\O_f)} 
+  \|\u_{f}\|^{r}_{W^{1,r}(\O_f)}
	+  \| \u_{p}\|^{r}_{L^{r}(\O_p)}
	+ | \u_{f} - \d_t \bbeta_{p}|^r_{BJS} \right) \, ds.\label{pressure-stab-2}
	\end{align}
Adding \eqref{energy-ineq-1} and \eqref{pressure-stab-2} and choosing $\epsilon_2$ 
small enough, 
and then $\epsilon_1$ small enough, implies
\begin{align*} 
& s_0\| p_{p}(t) \|_{L^2(\Omega_p)}^2 
+ \| \bbeta_{p} (t)\|^2_{H^1(\O_p)} 
+ \int_0^t \left(\| \u_{f} \|_{W^{1,r}(\O_f)}^r + \| \u_{p} \|_{L^{r}(\Omega_p)}^r 
+ | \u_{f} - \d_t \bbeta_{p} |^r_{BJS} \right)\, ds \\
&\qquad\qquad
+ \int_0^t \left(\|p_{f}\|^{r'}_{L^{r'}(\O_f)} + \|p_{p}\|^{r'}_{L^{r'}(\O_p)}  
+ \|\lam\|^{r'}_{W^{1/r,r'}(\Gamma_{fp})}\right)\, ds \\
&\qquad
\leq C\Big( \int_0^t\|\bbeta_{p}\|^2_{H^1(\O_p)}\, ds+ \| \f_{p}(t)\|^{2}_{H^{-1}(\O_p)} + \| \f_{p}(0) \|^{2}_{H^{-1}(\Omega_p)} +\| \bbeta_{p} (0)\|^2_{H^1(\O_p)}+ s_0 \| p_{p}(0) \|_{L_2(\Omega_p)}^2  \\ 
&\qquad\qquad
+ \int_0^t \left( \| \f_{f} \|_{W^{-1, r'}(\O_f)}^{r'} +\| \d_t\f_{p} \|_{H^{-1}(\O_p)}^{2}+\| q_{f} \|^r_{L^{r}(\Omega_f)} + \| q_{p} \|^r_{L^{r}(\Omega_p)} \right)\, ds \Big).
\end{align*}
The assertion of the theorem now follows from applying Gronwall's inequality.
\end{proof}

\section{Semidiscrete continuous-in-time approximation}
We assume that $\O_f$ and $\O_p$ are polytopal domains and that the Laplace 
problem in $\O_p$ has a solution with $W^{1+1/r,r}(\Omega_p)$ regularity. We refer
to \cite{Dauge,Grisvard} for sufficient conditions on $\O_p$.
Let $\mathcal{T}^f_h$ and $\mathcal{T}^p_h$
be shape-regular and quasi-uniform affine finite element partitions of
$\O_f$ and $\O_p$, respectively, not necessarily matching along the
interface $\Gamma_{fp}$. We consider the conforming finite element
spaces $\V_{f,h} \subset \V_f, \, W_{f,h} \subset W_f,\, \V_{p,h}
\subset \V_p,\,W_{p,h}\subset W_p$ and $\X_{p,h} \subset \X_p$. We
assume that $\V_{f,h},\, W_{f,h}$ is any inf-sup stable Stokes pair,
e.g., Taylor-Hood or the MINI elements. We choose $\V_{p,h},\,
W_{p,h}$ to be any of well-known inf-sup stable mixed finite element
Darcy spaces, e.g., the Raviart-Thomas or the Brezzi-Douglas-Marini
spaces \cite{BoffiBrezziFortin}. We employ a Lagrangian finite element
space $\X_{p,h} \subset \X_p$ to approximate the structure
displacement.  Note that the finite element spaces $\V_{f,h}$,
$\V_{p,h}$, and $\X_{p,h}$ satisfy the prescribed homogeneous boundary
conditions on the external boundaries $\Gam_f$ and $\Gam_p$. Finally,
following \cite{LSY,AKYZ-LM}, we choose a nonconforming approximation
for the Lagrange multiplier:
$$
\Lambda_h = \V_{p,h}\cdot \n_p|_{\Gam_{fp}}. 
$$ 
We equip $\Lambda_h$ with the norm 
$\|\cdot\|_{\Lambda_h} = \|\cdot\|_{L^{r'}(\Gamma_{fp})}$.

The semi-discrete continuous-in-time problem reads: for $t \in (0,T]$, find $(\u_{f,h}(t), p_{f,h}(t),\u_{p,h}(t),p_{p,h}(t),$ $\bbeta_{p,h}(t),\lam_h(t))$ $\in L^{\infty}(0,T;\V_{f,h})\times L^{\infty}(0,T;W_{f,h})\times  L^{\infty}(0,T;\V_{p,h}) \times  W^{1,\infty}(0,T;W_{p,h}) \times W^{1,\infty}(0,T;\X_{p,h})  \times L^{\infty}(0,T;\Lambda_h)$, such that $\forall$
$\v_{f,h} \in \V_{f,h}$, $w_{f,h} \in W_{f,h}$, $\v_{p,h} \in \V_{p,h}$, 
$w_{p,h} \in W_{p,h}$, 
$\bxi_{p,h} \in \X_{p,h}$, and $\mu_h \in \Lambda_h$,
\begin{align}
 & a_{f}(\u_{f,h},\v_{f,h}) + a^d_{p}(\u_{p,h},\v_{p,h})  
+ a^e_{p}(\bbeta_{p,h},\bxi_{p,h})
+ a_{BJS}(\u_{f,h},\d_t\bbeta_{p,h};\v_{f,h},\bxi_{p,h})
+ b_f(\v_{f,h},p_{f,h})\nonumber
\\
& \qquad\quad
   + b_p(\v_{p,h},p_{p,h}) + 
\alpha b_p(\bxi_{p,h},p_{p,h}) 
+ b_{\Gamma}(\v_{f,h},\v_{p,h},\bxi_{p,h};\lam_h) = (\f_{f},\v_{f,h})_{\O_f} 
+ (\f_{p},\bxi_{p,h})_{\O_p}, \label{h-weak-1} \\
&	( s_0 \d_t p_{p,h},w_{p,h})_{\O_p} 
- \alpha b_p(\d_t\bbeta_{p,h},w_{p,h}) 
- b_p(\u_{p,h},w_{p,h})  - b_f(\u_{f,h},w_{f,h}) \nonumber \\
& \qquad\quad
= (q_{f,h},w_{f,h})_{\O_f} + (q_{p,h},w_{p,h})_{\O_p}, \label{h-weak-2} \\
& b_{\Gamma}(\u_{f,h},\u_{p,h},\d_t\bbeta_{p,h};\mu_h) = 0. \label{h-b-gamma}
\end{align}
The initial conditions $p_{p,h}(0)$ and $\bbeta_{p,h}(0)$ 
are chosen as suitable approximations of 
$p_{p,0}$ and $\bbeta_{p,0}$.

In order to prove that the semi-discrete formulation
\eqref{h-weak-1}--\eqref{h-b-gamma} is well-posed, we will follow the
same strategy as in the fully continuous case. For the purpose of the
analysis only, we consider a discretization of the weak formulation
\eqref{h-cts-alt-1}--\eqref{h-cts-alt-gamma}. Let $\X_{p,h}$
consist of polynomials of degree at most $k_s$. We introduce the
stress finite element space $\bS_{e,h} \subset \bS_e$ as symmetric tensors
with elements that are discontinuous 
polynomials of degree at most $k_{s-1}$:
\begin{align*}
\bS_{e,h} =\{\bs_e \in \bS_e: \bs_e|_{T\in \mathcal{T}^p_h} 
\in \mathcal{P}^{\text{sym}}_{k_s-1}(T)^{d\times d} \}.
\end{align*}
 
Then the corresponding semi-discrete formulation is: for $t \in (0,T]$, find
$(\u_{f,h}(t), p_{f,h}(t),\u_{p,h}(t),p_{p,h}(t),$ $\u_{s,h}(t), \bs_{e,h}(t),\lam_h(t))$ $\in L^{\infty}(0,T;\V_{f,h})\times L^{\infty}(0,T;W_{f,h})\times  L^{\infty}(0,T;\V_{p,h}) \times  W^{1,\infty}(0,T;W_{p,h}) \times L^{\infty}(0,T;\X_{p,h})$ $\times W^{1,\infty}(0,T;\bS_{e,h}) \times L^{\infty}(0,T;\Lambda_h)$, such that for all
$\v_{f,h} \in \V_{f,h}$, $w_{f,h} \in W_{f,h}$, $\v_{p,h} \in \V_{p,h}$, $w_{p,h} \in W_{p,h}$, 
$\v_{s,h} \in \X_{p,h}$, $\btau_{e,h} \in \bS_{e,h}$, and $\mu_h \in \Lambda_h$,
\begin{align}
&  
a_{f}(\u_{f,h},\v_{f,h}) + a^d_{p}(\u_{p,h},\v_{p,h}) 
+ a_{BJS}(\u_{f,h}, \u_{s,h};\v_{f,h},\v_{s,h}) + b_f(\v_{f,h},p_{f,h})+ b_p(\v_{p,h},p_{p,h})\nonumber
\\
& \qquad\quad   + 
\alpha_p b_p(\v_{s,h},p_{p,h}) +b_{s}(\v_{s,h},\bs_{e,h})+ b_{\Gamma}(\v_{f,h},\v_{p,h},\v_{s,h};\lam_h)  = (\f_{f},\v_{f,h})_{\O_f} + (\f_{p},\v_{s,h})_{\O_p}, \label{h-alt-1}\\
& \left( s_0 \d_t p_{p,h},w_{p,h}\right)_{\O_p} + a^s_{p}(\d_t\bs_{e,h},\btau_{e,h})
- \alpha_p b_p\left(\u_{s,h},w_{p,h}\right) - b_p(\u_{p,h},w_{p,h}) -b_{s}(\u_{s,h},\btau_{e,h})- b_f(\u_{f,h},w_{f,h}) \nonumber
\\
& \qquad\quad 
= (q_{f},w_{f,h})_{\O_f} + (q_{p},w_{p,h})_{\O_p}, \label{h-alt-2} \\
& b_{\Gamma}\left(\u_{f,h},\u_{p,h},\u_{s,h};\mu_h\right) = 0. \label{h-alt-gamma}
\end{align}
The initial conditions $p_{p,h}(0)$ and $\bs_{e,h}(0)$ 
are suitable approximations of 
$p_{p,0}$ and $\bs_{e,0} = A^{-1}\D(\bbeta_{p,0})$.

We define the spaces of generalized velocities and pressures, $\Q_h=\V_{p,h}\times \X_{p,h}\times \V_{f,h}$ and $S_h=W_{p,h}\times \bS_{e,h}\times W_{f,h}\times \Lambda_h$, respectively, equipped with the corresponding norms, 
\begin{align*}
\|\q_h\|_{\Q_h} = \|\v_{p,h}\|_{\V_p} + \| \v_{s,h}\|_{\X_{p}} + \| \v_{f,h}\|_{\V_f},\qquad
\|s_h\|_{S_h} = \|w_{p,h}\|_{W_p} +\|\btau_{e,h}\|_{\bS_e} + \|w_{f,h}\|_{W_f} +\|\mu_h\|_{\Lambda_h}.
\end{align*}

\subsection{Discrete inf-sup conditions}
We first recall the inf-sup conditions for the individual Stokes and
Darcy problems \cite{ervin2009coupled}. Since $|\Gamma_p^D| > 0$,
it is sufficient to consider $\v_{p,h} \in \V_{p,h,\Gamma_{fp}}^0=
\{\v_{p,h} \in \V_{p,h}: \v_{p,h}\cdot \n_p\big|_{\Gamma_{fp}}=0\}$. 
There exist constant $C_{p,1} > 0$ and $C_{f,1} > 0$ 
independent of $h$ such that
\begin{align} 
\inf_{w_{p,h} \in W_{p,h}} \sup_{\v_{p,h} \in \V_{p,h,\Gamma_{fp}}^0} 
\frac{b_p(\v_{p,h},w_{p,h})}{\|\v_{p,h}\|_{\V_p}\|w_{p,h}\|_{W_p}} \ge C_{p,1},
\quad
\inf_{w_{f,h} \in W_{f,h}} \sup_{\v_{f,h} \in \V_{f,h}} 
\frac{b_f(\v_{f,h},w_{f,h})}{\|\v_{f,h}\|_{\V_f}\|w_{f,h}\|_{W_f}} \ge C_{f,1}.
\label{inf-sup-1-h}
\end{align}
We next prove inf-sup condition for $b_{\Gamma}(\cdot;\cdot)$. We
recall the mixed finite element interpolant $\Pi_{p,h}$ onto $\V_{p,h}$ \cite{BoffiBrezziFortin}, 
which satisfies for all $\v_p \in \V_{p}\cap (W^{s,r}(\Omega_p))^d$, $s>0$,
\begin{align}
(\nabla \cdot \Pi_{p,h}\v_p,w_{p,h})_{\O_p}&=(\nabla \cdot \v_p, w_{p,h})_{\O_p}, && \forall w_{p,h} \in W_{p,h}, \label{mfe-operator}
\\
\left\langle \Pi_{p,h} \v_p\cdot\n_p, \v_{p,h}\cdot\n_p \right\rangle_{\Gamma_{fp}} &= \left\langle \v_p\cdot\n_p, \v_{p,h}\cdot\n_p \right\rangle_{\Gamma_{fp}}, && \forall \v_{p,h} \in \V_{p,h},
\label{mfe-operator-bis}
\end{align}
as well as the continuity bound \cite{acosta2011error, duran1988error}
\begin{align}
\|\Pi_{p,h}\v_p\|_{L^r(\O_p)} \leq C\left(\|\v_p\|_{W^{s,r}(\O_p)} 
+ \|\div \v_p\|_{L^{r}(\O_p)} \right). \label{eq:mfe-approx-prop2}
\end{align}
Let
$\V_{p,h}^0 = \{\v_{p,h} \in \V_{p,h}: \div \v_{p,h} = 0\}$.

\begin{lemma}
There exists a constant $C_{2} >0$ independent of $h$ such that
	\begin{align} 
	\inf_{\mu_h\in \Lambda_h} \sup_{\v_{p,h} \in \V^0_{p,h}}
	\frac{ b_{\Gamma}(\v_{p,h},{\bf 0},{\bf 0};\mu_h)}
	{\|\v_{p,h}\|_{\V_p} \|\mu_h\|_{\Lambda_h}}
	\geq C_{2}. \label{inf-sup-2-h}
	\end{align}
\end{lemma}
\begin{proof}
Let $\mu_h \in \Lambda_h$ be given. Consider the auxiliary problem
\begin{align}
\nabla \cdot \nabla \phi  & = 0 , &&\mbox{in } \O_p, \label{1aux-prob-1}\\
\phi & = 0 && \mbox{ on } \Gam_{p}^D,  \\
\nabla \phi\cdot \n_p& = \mu_h^{r'-1}, &&\mbox{on } \Gam_{fp} , \label{1aux-prob-2} \\
\nabla \phi\cdot \n_p & = 0 , && \mbox{on } \Gam^N_{p}.   \label{eq:1axil-problem}
\end{align}	
Let $\v = \grad \phi$. Elliptic regularity for \eqref{1aux-prob-1}--\eqref{eq:1axil-problem} \cite{Dauge,Grisvard} implies that 
\begin{equation}\label{ell-reg}
\|\v\|_{W^{1/r,r}(\O_p)} \leq C \|\mu_h^{r'-1}\|_{L^{r}(\Gamma_{fp})}.
\end{equation}
Let $\v_{p,h} = \Pi_{p,h} \v$. Note that, due to \eqref{mfe-operator}, 
$\v_{p,h} \in \V^0_{p,h}$. We have
$$
\frac{b_{\Gamma}(\v_{p,h},0,0;\mu_h)}{\|\v_{p,h}\|_{\V_p}}
= \frac{\<\Pi_{p,h} \v\cdot\n_p,\mu_h\>_{\Gamma_{fp}}}{\|\Pi_{p,h} \v\|_{\V_p}}
= \frac{\<\v\cdot\n_p,\mu_h\>_{\Gamma_{fp}}}{\|\Pi_{p,h} \v\|_{\V_p}}
= \frac{\|\mu_h\|_{L^{r'}(\Gamma_{fp})}^{r'}}{\|\Pi_{p,h}\v\|_{L^r(\O_p)}},
$$
and, using \eqref{eq:mfe-approx-prop2} with $s = 1/r$ and \eqref{ell-reg},
$$
\|\Pi_{p,h}\v\|_{L^r(\O_p)} \le C \|\v\|_{W^{1/r,r}(\O_p)} 
\le C \|\mu_h^{r'-1}\|_{L^r(\Gamma_{fp})} = C \|\mu_h\|_{L^{r'}(\Gamma_{fp})}^{r'-1}.
$$
The proof is completed by combining the above two inequalities.
\end{proof}
We next prove the inf-sup conditions for the formulation 
\eqref{h-alt-1}--\eqref{h-alt-gamma}.
\begin{theorem} \label{inf-sup}
	There exist constants $\beta_1, \, \beta_2 >0$ independent of $h$ such that
	\begin{align} 
	\inf_{(w_{p,h},{\bf 0},w_{f,h},\mu_h)\in S_h} \sup_{(\v_{p,h},{\bf 0},\v_{f,h}) \in \Q_h}
	\frac{b(\q_h;s_h) + b_{\Gamma}(\q_h;s_h)}
	{\|(\v_{p,h},{\bf 0},\v_{f,h})\|_{\Q_h} \|(w_{p,h},0,w_{f,h},\mu_h)\|_{S_h}}
	&\geq \beta_1, \label{inf-sup-stokes-darcy-h} \\
	\inf_{({\bf 0},\v_{s,h}, {\bf 0})\in \Q_h}\sup_{(0,\btau_{e,h},0,0)\in S_h}\frac{ b_s(\v_{s,h}, \btau_{e,h})}{\|({\bf 0},\v_{s,h}, {\bf 0}) \|_{\Q}\|(0,\btau_{e,h},0,0)\|_{S_h}} &\geq \beta_2, \label{inf-sup-elasticity-h}
	\end{align}
where
\begin{align*}
b(\q_h;s_h) = b_f(\v_{f,h},w_{f,h})+ b_p(\v_{p,h},w_{p,h}), \quad
b_{\Gamma}(\q_h;s_h) = b_{\Gamma}(\v_{p,h},{\bf 0},\v_{f,h};\mu_h).
\end{align*}
\end{theorem}
\begin{proof}
Let $s_h = (w_{p,h},{\bf 0},w_{f,h},\mu_h)\in S_h$ be given. It follows from
\eqref{inf-sup-1-h} and \eqref{inf-sup-2-h}, respectively, that there exist
$\q^1_h=(\v^1_{p,h},{\bf 0},\v^1_{f,h}) \in \Q_h$ with $\|\v^1_{p,h}\|_{\V_p} = 1$,
$\|\v^1_{f,h}\|_{\V_f} = 1$, as well as
$\q^2_h=(\v^2_{p,h},{\bf 0},{\bf 0}) \in \Q_h$ with $\|\v^2_{p,h}\|_{\V_p} = 1$
such that 
$$
b_p(\v^1_{p,h},w_{p,h}) \ge \frac{C_{p,1}}{2}\|w_{p,h}\|_{W_p}, \quad
b_f(\v^1_{f,h},w_{f,h}) \ge \frac{C_{f,1}}{2}\|w_{f,h}\|_{W_f}, \quad
b_{\Gamma}(\v^2_{p,h},{\bf 0},{\bf 0};\mu_h) \ge \frac{C_2}{2}\|\mu_h\|_{\Lambda_h}.
$$
Since $\v^1_{p,h}\cdot \n_p\big|_{\Gamma_{fp}}=0$, we have
	\begin{align*}
	b_{\Gam}(\q^1_h;s_h) &= \langle \v^1_{f,h} \cdot \n_f + \v^1_{p,h} \cdot \n_p,\mu_h\rangle_{\Gamma_{fp}} =\langle \v^1_{f,h} \cdot \n_f ,\mu_h\rangle_{\Gamma_{fp}} 
\leq C \|\v^1_{f,h}\|_{L^r(\Gamma_{fp})}\|\mu_h\|_{L^{r'}(\Gamma_{fp})} \\
	& \leq C\|\v^1_{f,h}\|_{W^{1-1/r,r}(\partial \O_f)}\|\mu_h\|_{L^{r'}(\Gamma_{fp})} 
\leq C_{\Gam}\|\v^1_{f,h}\|_{W^{1,r}(\O_f)}\|\mu_h\|_{L^{r'}(\Gamma_{fp})}
= C_{\Gam}\|\v^1_{f,h}\|_{\V_f}\|\mu_h\|_{\Lambda_h},
	\end{align*}
where we used the trace inequality. Let 
$\r_h =\q^1_h+ (1 + 2 C_{\Gam}C_{2}^{-1})\q^2_h$. Since $\div \v^2_{p,h} = 0$, we
obtain
\begin{align*}
b(\r_h;s_{h})& = b_f(\v^1_{f,h},w_{f,h}) + b_p(\v^1_{p,h},w_{p,h}) 
+ \left(1+2C_{\Gam}C_{2}^{-1}\right) b_p(\v^2_{p,h},w_{p,h}) \\
	&= b_f(\v^1_{f,h},w_{f,h}) + b_p(\v^1_{p,h},w_{p,h}) 
\geq \frac{\min (C_{f,1},C_{p,1})}{2}(\|(w_{p,h}\|_{W_p} + \|w_{f,h}\|_{W_f}), \\
b_{\Gamma}(\r_h;s_h) &= b_{\Gamma}(\q^1_h;s_h) + \left(1+2C_{\Gam}C_{2}^{-1}
\right)b_{\Gamma}(\q^2_h;s_h) \\
& \geq - C_{\Gam}\|\mu_h\|_{\Lambda_h} + 
\frac{C_{2}}{2}\left(1+2C_{\Gam}C_{2}^{-1}\right)\|\mu_h\|_{\Lambda_h} 
= \frac{C_{2}}{2}\|\mu_h\|_{\Lambda_h}.
\end{align*}
Hence, using that $\|\r_h\|_{\Q_h} \le 3 + 2C_{\Gam}C_{2}^{-1}$, we obtain
\begin{align*}
b(\r_h;s_h) + b_{\Gamma}(\r_h;s_h) 
\ge \frac{\min (C_{f,1},C_{p,1},C_2)}{2}\|s_h\|_{S_h}
\geq \frac{\min (C_{f,1},C_{p,1},C_2)}{6 + 4C_{\Gam}C_{2}^{-1}}
\|\r_h\|_{\Q_h}\|s_h\|_{S_h}, 
\end{align*}
which completes the proof of \eqref{inf-sup-stokes-darcy-h}. To show
\eqref{inf-sup-elasticity-h}, let $({\bf 0},\v_{s,h},{\bf
  0})\in \Q_h$ be given. We choose $\btau_{e,h}=\D(\v_{s,h}) \in \bS_{e,h}$ and,
using Korn's inequality, we obtain
\begin{align*}
\frac{b_s(\v_{s,h},\btau_{e,h}) }{\|\btau_{e,h}\|_{L^2(\O_p)}} = \frac{\|\D(\v_{s,h})\|^2_{L^2(\O_p)} }{\|\D(\v_{s,h})\|_{L^2(\O_p)}} = \|\D(\v_{s,h})\|_{L^2(\O_p)} \geq 
\beta_2\|\v_{s,h}\|_{H^1(\O_p)}.
\end{align*} 
\end{proof}

\subsection{Existence and uniqueness of a solution}
In order to show well-posedness of
\eqref{h-alt-1}--\eqref{h-alt-gamma}, we proceed as in the case of
the continuous problem. We introduce $W^2_{p,h}$ and $\bS^2_{e,h}$ 
as the closures of the spaces $W_{p,h}$ and $\bS_{e,h}$ with respect to the norms
\begin{align*}
\|w_{p,h}\|^2_{W^2_{p,h}} \, := \, (s_0w_{p,h},w_{p,h})_{L^2(\O_p)},\quad \|\btau_{e,h}\|^2_{\bS^2_{e,h}} \, := \, (A\btau_{e,h},\btau_{e,h})_{L^2(\O_p)}.
\end{align*}
Define the domain 
\begin{align}
& D_h \, := \, 
\left\{ (p_{p,h},\bs_{e,h})\in W_{p,h}\times \bS_{e,h}:\, \mbox{ for given } 
(\f_f,\f_p,q_f)\in \V_f'\times \X_p'\times W_f' \right. \nonumber\\
& \qquad \exists \,(\u_{p,h},\u_{s,h},\u_{f,h}), p_{f,h},\lam_h)\in \Q_h\times W_{f,h}\times \Lambda_h \mbox{ such that} \nonumber \\
& \qquad \quad 
\forall ( (\v_{p,h},\v_{s,h},\v_{f,h}), (w_{p,h}, \btau_{e,h}, w_{f,h}, \mu_h)) 
\in \Q_h\times S_h \mbox{:}
\nonumber\\
&  
a_{f}(\u_{f,h},\v_{f,h}) + a^d_{p}(\u_{p,h},\v_{p,h}) 
+ a_{BJS}(\u_{f,h}, \u_{s,h};\v_{f,h},\v_{s,h}) + b_f(\v_{f,h},p_{f,h})+ b_p(\v_{p,h},p_{p,h})\nonumber
\\
& \qquad   + 
\alpha_p b_p(\v_{s,h},p_{p,h}) +b_{s}(\v_{s,h},\bs_{e,h})+ b_{\Gamma}(\v_{f,h},\v_{p,h},\v_{s,h};\lam_h)  = (\f_{f,h},\v_{f,h})_{\O_f} + (\f_{p,h},\v_{s,h})_{\O_p}, 
\nonumber \\
& \left( s_0 p_{p,h},w_{p,h}\right)_{\O_p} + a^s_{p}(\bs_{e,h},\btau_{e,h})
- \alpha_p b_p\left(\u_{s,h},w_{p,h}\right) - b_p(\u_{p,h},w_{p,h}) \nonumber \\
& \qquad -b_{s}(\u_{s,h},\btau_{e,h})- b_f(\u_{f,h},w_{f,h}) 
= (q_{f},w_{f,h})_{\O_f} + (s_0\bar{g}_{p},w_{p,h})_{\O_p}+ (A\bar{g}_{e},\btau_{e,h})_{\O_p}, \nonumber \\
& b_{\Gamma}\left(\u_{f,h},\u_{p,h},\u_{s,h};\mu_h\right) = 0. 
\nonumber \\
& \left. \quad \quad  \quad \mbox{for some } (\bar{g_p},\bar{g}_e)\in \left(W^2_{p,h}\right)^{\prime}\times \left(\bS^2_{e,h}\right)^{\prime} \,  \right\} 
\subset W^2_{p,h}\times \bS^2_{e,h}  \, . \label{h-deffD}
\end{align}
Analogous to the continuous formulation, we introduce the multivalued operator $\mathcal{M}_{h}$ with domain $D_h$,
and its associated \textit{relation} $\mathcal{M}_{h} \subset \left( W_{p,h}\times \bS_{e,h} \right) \times 
 \left( W_{p,h}^{2} \times \bS_{e,h}^{2}\right)^{\prime}$, where
\begin{equation}
\mathcal{M}( (p_{p,h},\bs_{e,h}) ) :=  
\left\{ (\bar{g}_p - p_{p,h} , \bar{g}_e - \bs_{e,h})  :  (p_p,\bs_e)
\mbox{ satisfies (4.9)--(4.11) for }
 (\bar{g}_p,\bar{g}_e)\in W_{p,2}^{\prime}\times \bS_{e,2}^{\prime} \right\},
\label{h-deffL}
\end{equation}
and consider the problem 
\begin{align}
\frac{d}{dt}\begin{pmatrix}
p_{p,h}(t) \\
\bs_{e,h}(t)
\end{pmatrix} + \mathcal{M} \begin{pmatrix}
p_p(t) \\
\bs_e(t)
\end{pmatrix} \ni \begin{pmatrix}
s_{o}^{-1}q_{p} \\
0
\end{pmatrix}. \label{h-ode-problem}
\end{align}

We can establish the following well-posedness result.
\begin{theorem}\label{well-pos-alt-h}
For each $\f_{f} \in
	W^{1,1}(0, T; \V_{f}')$, $\f_{p} \in W^{1,1}(0, T; \X_{p}')$, $q_f
	\in W^{1,1}(0, T; W_{f}')$, \linebreak $q_p \in W^{1,1}(0, T; W_{p}')$, and
	$p_{p,0}  \in W_p$, $\bs_{e,0} = A^{-1}\D(\bbeta_{p,0}) \in \bS_{e}$,
	there exists a solution of
	\eqref{h-alt-1}--\eqref{h-alt-gamma} with 
	$(\u_{f,h}$,
	$p_{f,h}$, $\u_{p,h}$, $p_{p,h}$, $\u_{s,h}$, $\bs_{e,h}$, $\lam_h) \in
	L^{\infty}(0,T;\V_{f,h})\times L^{\infty}(0,T;W_{f,h})\times
	L^{\infty}(0,T;\V_{p,h}) \times W^{1,\infty}(0,T;W_{p,h}) \times
	L^{\infty}(0,T;\X_{p,h}) \times W^{1,\infty}(0,T;\bS_{e,h}) \times
	L^{\infty}(0,T;\Lambda_h)$.
\end{theorem}
The proof of Theorem \ref{well-pos-alt-h} uses the following steps: 

\noindent
\textbf{Step 1}. Establish that the domain $D_h$ given by \eqref{h-deffD} 
is nonempty. \\
\textbf{Step 2}. Show solvability of the parabolic problem \eqref{h-ode-problem}. \\
\textbf{Step 3}. Show that the solution to \eqref{h-ode-problem} satisfies
\eqref{h-alt-1}--\eqref{h-alt-gamma}.	

With the established discrete inf-sup conditions \eqref{inf-sup-stokes-darcy-h}
and \eqref{inf-sup-elasticity-h}, the proof follows closely the proof
of Theorem~\ref{well-pos-alt}. In particular,
the proofs of Step 2 and Step 3 in the discrete setting are identical
to the continuous case. The proof of Step
1 is also very similar. The only difference is that the operator
$L_{\Gamma}$ from Lemma~\ref{lmapropLg} is now defined as
$L_{\Gam} : \, \Lambda_h \rightarrow \Lambda_h'$, $L_{\Gam}(\mu_{h,1})
(\mu_{h,2}): = \<|\mu_{h, 1}|^{r'-2} \mu_{h, 1}, \mu_{h,2}
\>_{\Gam_{fp}}$. One needs to establish that $L_{\Gam}$ is a bounded,
continuous, coercive and monotone operator, which follows immediately
from its definition, since $\left(L_{\Gam}(\mu_h)(\mu_h)\right)^{1/r'}
= \|\mu_h\|_{\Lambda_h}$.

As a corollary of Theorem \ref{well-pos-alt-h}, we obtain
the following well-posedness result for the original semi-discrete problem
\eqref{h-weak-1}--\eqref{h-b-gamma}. The proof is identical to the proof
of Theorem~\ref{th-exist-uniq-weak-2}.
\begin{theorem}\label{well-pos}
	For each $\f_{f} \in
	W^{1,1}(0, T; \V_{f}')$, $\f_{p} \in W^{1,1}(0, T; \X_{p}')$, $q_f
	\in W^{1,1}(0, T; W_{f}')$, \linebreak $q_p \in W^{1,1}(0, T; W_{p}')$, and
	$p_{p,0}\in W_{p,h}$, $\bbeta_{p,0} \in \X_{p,h}$,
	there exists a unique solution $(\u_{f,h}(t), p_{f,h}(t),$ $\u_{p,h}(t), p_{p,h}(t),\bbeta_{p,h}(t),\lam_h(t))$
	$\in L^{\infty}(0,T;\V_{f,h})\times L^{\infty}(0,T;W_{f,h})\times
	L^{\infty}(0,T;\V_{p,h}) \times W^{1,\infty}(0,T;W_{p,h}) \times
	W^{1,\infty}(0,T;\X_{p,h}) \times L^{\infty}(0,T;\Lambda_h)$ of
	\eqref{h-weak-1}--\eqref{h-b-gamma}. 
\end{theorem}

The proof of the following stability result is identical to the proof
of Theorem~\ref{stability}.
\begin{theorem} \label{stability-discrete}
For the solution of \eqref{h-weak-1}--\eqref{h-b-gamma}, 
assuming sufficient regularity of the data,
there exists $C>0$ such that
\begin{align*}
& \| \u_{f,h} \|_{L^r(0,T; W^{1,r}(\O_f))}^r + \| \u_{p,h} \|_{L^r(0,T; L^{r}(\Omega_p))}^r +  |\u_{f,h} - \d_t \bbeta_{p,h} |^r_{L^r(0,T;BJS)} + \| p_{f,h} \|^{r'}_{L^{r'} (0,T; L^{r'}(\O_f))} \\
&\qquad
+  \| p_{p,h} \|^{r'}_{L^{r'} (0,T; L^{r'}(\O_p))} + \| \lam_h \|^{r'}_{L^{r'} (0,T; \Lambda_h)}
	 + \| \bbeta_{p,h} \|^2_{L^\infty (0,T; H^1({\Omega_p}))} + s_0\| p_{p,h} \|^2_{L^\infty (0,T; L^2(\O_p))}\\
& \quad
\leq 
C \exp(T) \Big( \| \f_{p}\|^{2}_{L^\infty (0,T; H^{-1}(\Omega_p))}  
+\| \bbeta_{p,h} (0)\|^2_{H^1(\Omega_p)} 
+ s_0 \| p_{p,h}(0) \|_{L_2(\Omega_p)}^2	
+ \| \d_t \f_{p} \|_{L^2(0,T;H^{-1}(\Omega_p))}^2 \\
& \qquad 
+ \| \f_{f} \|^{r'}_{L^{r'}(0,T;W^{-1,r'}(\Omega_f))} 
+ \| q_{f} \|^{r}_{L^{r}(0,T;L^{r}(\Omega_f))}  
+ \| q_{p} \|^{r}_{L^{r}(0,T;L^{r}(\Omega_f))} + c( \bar c_f + \bar c_p + \bar c_I)\Big).
\end{align*}
\end{theorem}

\section{Error analysis}

In this section we analyze the spatial discretization error. Let 
$k_f$ and $s_f$ be the degrees of polynomials in
$\V_{f,h}$ and $W_{f,h}$, let $k_p$ and $s_p$ be the
degrees of polynomials in $\V_{p,h}$ and $W_{p,h}$
respectively, and let $k_s$ be the polynomial degree in $\X_{p,h}$.

\subsection{Preliminaries}
We introduce $Q_{f,h}$, $Q_{p,h}$, and $Q_{\lambda,h}$ as the $L^2$ projection 
operators onto $W_{f,h}$, $W_{p,h}$, and $\Lambda_h$, respectively, satisfying:
\begin{align}
& (p_{f}-Q_{f,h}p_{f},w_{f,h})_{\O_f}=0,&&\forall w_{f,h} \in W_{f,h}, \label{fluid-pressure-int}\\
& (p_{p}-Q_{p,h}p_{p},w_{p,h})_{\O_p}=0,&&\forall w_{p,h} \in W_{p,h}, \label{darcy-pressure-int}\\
& \langle\lambda-Q_{\lambda,h}\lambda,\mu_h\rangle_{\Gamma_{fp}}=0,&&\forall \mu_h \in \Lambda_h, \label{l-multuplier-int}
\end{align}
with approximation properties \cite{di2017hybrid},
	\begin{eqnarray}
	&& \|p_f - Q_{f,h}p_f\|_{L^{r'}(\O_f )} \leq Ch^{s_f+1}\|p_f\|_{W^{s_f+1,r'}(\O_f)}, \label{stokesPresProj}\\
	&& \|p_p - Q_{p,h}p_p\|_{L^{r'}(\O_p )} \leq Ch^{s_p+1}\|p_p\|_{W^{s_p+1,r'}(\O_p)}, \label{darcyPresProj}\\
	&&  \|\lambda - Q_{\lambda,h} \lambda\|_{L^{r'}(\Gamma_{fp} )} \leq Ch^{k_p+1} \|\lambda\|_{W^{k_p+1,r'}(\Gamma_{fp})}. \label{LMProj}
	\end{eqnarray}	
In the error analysis we will use an interpolant 
$I_h=(I_{f,h},I_{p,h},I_{s,h}):\, \U \rightarrow \U_{h}$, where
$$
\U = \left\{ (\v_f,\v_p,\bxi_p) \in \V_f \times \V_p \times \X_p : 
b_{\Gamma}\left(\v_f,\v_p,\bxi_p;\mu\right)=0, \forall \mu \in \Lambda \right\},
$$
$$
\U_h = \left\{ (\v_{f,h},\v_{p,h},\bxi_{p,h}) \in \V_{f,h} \times \V_{p,h} 
\times \X_{p,h} : b_{\Gamma}\left(\v_{f,h},\v_{p,h},\bxi_{p,h};\mu_h\right)=0, 
\forall \mu_h \in \Lambda_h \right\}.
$$
We construct the interpolant by combining sub-problem interpolants with correction
on the interface for the flux continuity. 
We recall the mixed finite element interpolant $\Pi_{p,h}$ onto $\V_{p,h}$ introduced in 
\eqref{mfe-operator}. It satisfies the approximation property 
\cite{acosta2011error, duran1988error},
\begin{align}
\|\v_p-\Pi_{p,h}\v_p\|_{L^r(\O_p)} \leq Ch^{k_p+1}\|\v_p\|_{W^{k_p+1,r}(\O_p)}. 
\label{eq:mfe-approx-prop1}
\end{align}
Let $S_{f,h},\, S_{s,h}$ be the Scott-Zhang interpolation 
operators onto $\V_{f,h}$ and $\X_{p,h}$, respectively, satisfying 
\cite{Scott_Zhang_1990}
\begin{align}
\|\v_f-S_{f,h}\v_f\|_{L^r(\O_f)} + h\|\nabla(\v_f-S_{f,h}\v_f)\|_{L^r(\O_f)}
&\leq Ch^{k_f+1}\|\v_f\|_{W^{k_f+1,r}(\O_f)}, \label{eq:scott-zhang-approx-prop1} \\
\|\bxi_p-S_{s,h}\bxi_p\|_{L^2(\O_p)} + h\|\nabla(\bxi_p-S_{s,h}\bxi_p)\|_{L^2(\O_p)}&\leq Ch^{k_s+1}\|\bxi_p\|_{H^{k_s+1}(\O_p)}. \label{eq:scott-zhang-approx-prop2} 
\end{align}
We set $I_{f,h}= S_{f,h}$ and $I_{s,h}= S_{s,h}$. We next
construct $I_{p,h}\v_p$. Consider the auxiliary problem:
for $\v_f$ and $\bxi_p$ given, find $\phi \in W^{1+1/r,r}(\O_p)$ satisfying
\begin{align}
\nabla \cdot \nabla \phi  & = 0 , &&\mbox{in } \O_p, \label{aux-prob-1-h}\\
\phi & = 0 && \mbox{on } \Gam_{p}^D,  \\
\nabla \phi\cdot \n_p& = 
(\v_f-I_{f,h}\v_f)\cdot\n_f + (\bxi_p-I_{s,h}\bxi_p)\cdot \n_p, &&\mbox{on } \Gam_{fp} , \label{aux-prob-2-h} \\
\nabla \phi\cdot \n_p & = 0 , && \mbox{on } \Gam^N_{p}.   \label{eq:axil-problem-h}
\end{align}
Let $\z=\nabla \phi$ and define $\w=\z+\v_p$. 
Using \eqref{aux-prob-1-h}--\eqref{eq:axil-problem-h}, we obtain
\begin{align}
\nabla \cdot \w &= \nabla \cdot \z + \nabla \cdot \v_p = \nabla \cdot \v_p, &&\textrm{ in }\O_p, \label{div-w}\\
\w \cdot \n_p &= \z \cdot \n_p +   \v_p \cdot \n_p  = \v_f\cdot\n_f -I_{f,h}\v_f\cdot\n_f 
+ \bxi_p\cdot \n_p- I_{s,h}\bxi_p\cdot \n_p + \v_p \cdot \n_p \nonumber \\
& = -I_{f,h}\v_f\cdot\n_f - I_{s,h}\bxi_p\cdot \n_p,
&&\textrm{ on }\Gamma_{fp}. \label{propW}
\end{align}
We now set $I_{p,h}\v_p = \Pi_{p,h}\w$. Using \eqref{mfe-operator} and 
\eqref{div-w}, we have 
\begin{align}
(\nabla \cdot I_{p,h}\v_p, w_{p,h})_{\O_p}& =(\nabla \cdot \Pi_{p,h}\w, w_{p,h})_{\O_p}=(\nabla \cdot \w, w_{p,h})_{\O_p}=(\nabla \cdot \v_p, w_{p,h})_{\O_p}, && \forall w_{p,h} \in W_{p,h}. \label{darcy-velocity-int2}
\end{align}
Using \eqref{mfe-operator-bis} and \eqref{propW}, we have for all 
$\mu_h \in \Lambda_h$,
\begin{align*}
\langle I_{p,h}\v_p \cdot \n_p,\mu_h \rangle_{\Gamma_{fp}} = 
\langle \Pi_{p,h}\w \cdot \n_p,\mu_h \rangle_{\Gamma_{fp}} = 
\langle \w \cdot \n_p,\mu_h \rangle_{\Gamma_{fp}} =  
\langle -I_{f,h}\v_f\cdot\n_f - I_{s,h}\bxi_p\cdot \n_p,\mu_h \rangle_{\Gamma_{fp}},
\end{align*}
which implies that $I_h: \U \mapsto \U_h$ satisfies
\begin{align}
\left\langle I_{f,h}\v_f\cdot \n_f + I_{p,h}\v_p \cdot \n_p 
+ I_{s,h}\bxi_p\cdot \n_p,\mu_h \right\rangle_{\Gamma_{fp}} =0, && 
\forall \mu_h \in \Lambda_h. \label{new-operator}
\end{align}

We next present the approximation properties of $I_h$.

\begin{lemma}\label{l:interpolation}
  For $\v_f \in W^{k_f+1,r}(\O_f)$, $\v_p \in W^{k_p+1, r}(\O_p) $, and
  $\bxi_p \in H^{k_s+1}(\O_p)$, there exists $C > 0$ independent of $h$ such that
	\begin{align}
	&\|\v_f - I_{f,h}\v_f\|_{L^r(\O_f )} + h\|\nabla(\v_f - I_{f,h}\v_f)\|_{L^r(\O_f )} \leq Ch^{k_f+1}\|\v_f\|_{W^{k_f+1,r}(\O_f)}, \label{stokesVel}\\
	&\|\bxi_p - I_{s,h}\bxi_p\|_{L^2(\O_p )}+ h\|\nabla(\bxi_p - I_{s,h}\bxi_p)\|_{L^2(\O_p )} \leq Ch^{k_s+1}\|\bxi_p\|_{H^{k_s+1}(\O_p)},  \label{dispOperator}\\
	&\|\v_p - I_{p,h}\v_p\|_{L^r(\O_p )} \leq C(h^{k_p+1}\|\v_p\|_{W^{k_p+1, r}(\O_p)}+ h^{k_f}\|\v_f\|_{W^{k_f+1,r}(\O_f )} + h^{k_s}\|\bxi_p\|_{H^{k_s+1}(\O_p)}).
	\label{darcyVel} 
	\end{align}
\end{lemma}

\begin{proof}
The first two estimates \eqref{stokesVel}--\eqref{dispOperator} follow
immediately from
\eqref{eq:scott-zhang-approx-prop1}--\eqref{eq:scott-zhang-approx-prop2}.
Next,
\begin{align}\label{v-Ipv}
\|\v_p - I_{p,h}\v_p\|_{L^r(\O_p )} 
= \|\v_p - \Pi_{p,h} \v_p - \Pi_{p,h} \z\|_{L^r(\O_p )} 
\leq  \|\v_p - \Pi_{p,h} \v_p \|_{L^r(\O_p )} + \| \Pi_{p,h} \z\|_{L^r(\O_p )}.
\end{align}
Using \eqref{eq:mfe-approx-prop2}, 
elliptic regularity \eqref{ell-reg} for
\eqref{aux-prob-1-h}--\eqref{eq:axil-problem-h},
\eqref{stokesVel}, and \eqref{dispOperator}, we obtain
\begin{align}
\| \Pi_{p,h} \z\|_{L^r(\O_p )} &\leq C\|\z\|_{W^{1/r,r}(\O_p)} 
\leq C (\|(\v_f-I_{f,h}\v_f)\cdot\n_f \|_{L^{r}(\Gam_{fp})} 
+ \|(\bxi_p-I_{s,h}\bxi_p)\cdot \n_p\|_{L^{r}(\Gam_{fp})}) \nonumber \\
&\leq C (\|\v_f-I_{f,h}\v_f \|_{W^{1,r}(\O_{f})} 
+ \|\bxi_p-I_{s,h}\bxi_p\|_{H^{1}(\O_{p})}) \nonumber \\
& \leq C( h^{k_f}\|\v_f\|_{W^{k_f+1,r}(\O_f )} + h^{k_s}\|\bxi_p\|_{H^{k_s+1}(\O_p)}).
\label{Piz}
\end{align}
Bound \eqref{darcyVel} follows by combining \eqref{v-Ipv}, 
\eqref{eq:mfe-approx-prop1}, and \eqref{Piz}.
\end{proof}

\subsection{Error estimates}
For $\u = (\u_f, \u_p, \bbeta_p)$ and $\u_h = (\u_{f,h}, \u_{p,h}, \bbeta_{p,h})$, 
define
\begin{align}
\mathcal{E} (\u, \u_h) &= \norm{\frac{| \D (\u_f) - \D (\u_{f,h}) | }{c + | \D (\u_f) | + | \D (\u_{f,h}) | }}_{L^\infty(\O_f)}^ {\frac{2 - r}{r}} +  \norm{\frac{|\u_p - \u_{p,h} | }{c + | \u_p | + | \u_{p,h} | }}_{L^\infty(\O_p)}^ {\frac{2 - r}{r}} \nonumber\\
 & + \sum_{j=1}^{d-1}\norm{\frac{ | (\u_f - \d_t \bbeta_p)\cdot \t_{f,j} - (\u_{f,h} - \d_t \bbeta_{p,h})\cdot \t_{f,j} | }{c + | (\u_f - \d_t \bbeta_p)\cdot \t_{f,j} |  +  | (\u_{f,h} - \d_t \bbeta_{p,h})\cdot \t_{f,j}  | } }^{\frac{2-r}{r}}_{L^\infty(\Gamma_{fp})} \text{ and}\nonumber\\
\mathcal{G} (\u, \u_h) &= (| \nu (\D (\u_f)) \D (\u_f) - \nu (\D (\u_{f,h})) \D (\u_{f,h}) |, | \D (\u_f) - \D (\u_{f,h} ) )_{\O_f} \nonumber \\
 \quad & + (| \nu_{eff} (\u_p) \u_p - \nu_{eff} (\u_{p,h}) \u_{p,h} |, | \u_p - \u_{p,h} |)_{\O_p} \nonumber \\
 &+ \sum_{j=1}^{d-1} \alpha_{BJS}\<| \nu_{I} (((\u_f - \d_t \bbeta_p)\cdot \t_{f,j})\t_{f,j}) ((\u_f - \d_t \bbeta_p)\cdot \t_{f,j})\t_{f,j} \nonumber \\
&\qquad - \nu_{I} (((\u_{f,h} - \d_t \bbeta_{p,h})\cdot \t_{f,j})\t_{f,j}) ((\u_{f,h} - \d_t \bbeta_{p,h})\cdot \t_{f,j})\t_{f,j} |, \nonumber \\
& \qquad |   ((\u_f - \d_t \bbeta_p)\cdot \t_{f,j})\t_{f,j} - ((\u_{f,h} - \d_t \bbeta_{p,h})\cdot \t_{f,j})\t_{f,j} | \>_{\Gamma_{fp}}.
 \end{align}
The above quantities appear in the error analysis when applying the
continuity bound \eqref{continuity} to the difference of the true and
approximate velocities. Note that as each term in $\mathcal{E} (\u, \u_h)$
is less than 1, $\mathcal{E} (\u, \u_h) \le (d + 1)$.

\begin{theorem}\label{conv-thm}
Let $(\u_f, \u_p, \bbeta_p, p_f, p_p, \lambda)$ be the solution of 
\eqref{h-cts-1}--\eqref{h-cts-gamma} and $(\u_{f,h}, \u_{p,h}, \bbeta_{p,h}, p_{f,h}, p_{p,h}, \lambda_h)$ be the solution of \eqref{h-weak-1}--\eqref{h-b-gamma}. There exists a constant $C>0$ independent of $h$ such that
\begin{align*} 
& \| \u_f - \u_{f,h} \|^2_{L^2(0,T;W^{1,r}(\O_f))} 
+ \| \u_p - \u_{p,h} \|^2_{L^2(0,T;L^r(\O_p))}
+|(\u_f - \d_t \bbeta_p) - (\u_{f,h} - \d_t \bbeta_{p,h})|^2_{L^2(0,T;BJS)} \\ 
& \qquad
+\|p_f - p_{f,h}\|^{r'}_{L^{r'}(0,T;L^{r'}(\O_f))}+\|p_p - p_{p,h}\|^{r'}_{L^{r'}(0,T;L^{r'}(\O_p))} +\|Q_{\lam,h}\lam - \lam_{h}\|^{r'}_{L^{r'}(0,T;L^{r'}(\Gam_{fp}))}\\
& \qquad
 +\|\bbeta_p - \bbeta_{p,h}\|^2_{L^{\infty}(0,T;H^1(\O_p))} +s_0\|p_{p}-  p_{p,h}\|^2_{L^{\infty}(0,T;L^2(\O_p))}+\|\mathcal{G}(\u,\u_h)\|_{L^1(0,T)}\\
& \quad
\leq C\exp(T)\Big(h^{2k_f}\|\u_f\|^2_{L^2(0,T; W^{k_f+1,r}(\O_f))}+ h^{rk_f}\|\u_f\|^r_{L^r(0,T; W^{k_f+1,r}(\O_f))}  \\
& \qquad
+ h^{2(s_f+1)}\|p_f\|^2_{L^2(0,T; W^{s_f+1,r'}(\O_f))}  + h^{r'(s_f+1)}\|p_f\|^{r'}_{L^{r'}(0,T; W^{s_f+1,r'}(\O_f))} \\
& \qquad
+ h^{r(k_p+1)}\|\u_p\|^r_{L^r(0,T; W^{k_p+1,r}(\O_p))}+h^{r'(s_p+1)}\| p_p\|^{r'}_{L^{r'}(0,T; W^{s_p+1,r'}(\O_p))} \\
& \qquad
+h^{2(s_p+1)}(\|\d_t p_p\|^2_{L^2(0,T; W^{s_p+1,r'}(\O_p))}+\| p_p\|^2_{L^{\infty}(0,T; W^{s_p+1,r'}(\O_p))})  \\
& \qquad
+h^{2k_s} \left(\| \bbeta_p\|^{2}_{L^{2}(0,T; H^{k_s+1}(\O_p))} + \|\d_t \bbeta_p\|^{2}_{L^{2}(0,T; H^{k_s+1}(\O_p))}+  \| \bbeta_p\|^{2}_{L^{\infty}(0,T; H^{k_s+1}(\O_p))}\right) \\
& \qquad
+h^{rk_s}\| \d_t \bbeta_p\|^{r}_{L^{r}(0,T; H^{k_s+1}(\O_p))} + h^{r'(k_p+1)}\|\lam\|^{r'}_{L^{r'}(0,T; W^{k_p+1,r'}(\Gam_{fp}))}	\\
& \qquad
+ h^{2(k_p+1)}(\|\lam\|^2_{L^2(0,T; W^{k_p+1,r'}(\Gam_{fp}))} + \|\d_t\lam\|^2_{L^2(0,T; W^{k_p+1,r'}(\Gam_{fp}))}+ \|\lam\|^2_{L^{\infty}(0,T; W^{k_p+1,r'}(\Gam_{fp}))}) \Big).
\end{align*}
\end{theorem}
\begin{proof}
  The proof is comprised of four main steps. In \textbf{Step 1}, bounds for
  $\| \u_f - \u_{f,h} \|_{W^{1,r}(\O_f)}$
and  $\| \u_p - \u_{p,h} \|_{L^r(\O_p)}$ are obtained using the the monotonicity
\eqref{monotonicity} and continuity \eqref{continuity} assumptions. 
Bounds for $\|\bbeta_p(t) - \bbeta_{p,h}(t)\|_{H^1(\O_p)}$ and  
$\|p_{p}(t)-  p_{p,h}(t)\|_{L^2(\O_p)}$ are obtained in \textbf{Step 2}.
Using the discrete inf-sup condition \eqref{inf-sup-stokes-darcy-h},
bounds for $\|p_f - p_{f,h}\|_{L^{r'}(\O_f)}$,
$\|p_p - p_{p,h}\|_{L^{r'}(\O_p)}$, and  $\|\lam - \lam_{h}\|_{L^{r'}(\Gam_{fp})}$
are obtained in \textbf{Step 3}. In \textbf{Step 4} we combine the bounds,
apply Gronwall's inequality
and the approximation properties \eqref{stokesPresProj}--\eqref{LMProj} and
\eqref{stokesVel}--\eqref{darcyVel}, to complete the proof.

We note that the discretization error
is bounded in the same spatial norms as in the stability
bound of Theorem~\ref{stability-discrete}. The temporal norms for the
pressures and the Lagrange multiplier are also as in
Theorem~\ref{stability-discrete}. However, due to the use of the
monotonicity \eqref{monotonicity}, the temporal norm for the velocity and
displacement error is $L^2(0,T)$. This is in contrast to the
$L^r(0,T)$ norm in the stability estimate, which used the coercivity
bounds in \eqref{a-f-bounds}--\eqref{a-bjs-bounds}. \\

\textbf{Step 1}. Bounds for $\| \u_f - \u_{f,h} \|_{W^{1,r}(\O_f)}$
and  $\| \u_p - \u_{p,h} \|_{L^r(\O_p)}$. \\

Using \eqref{monotonicity} with $\G(\x)= \nu(\x)\x$, $\s = \D(\u_f)$ 
and  $\t = \D(\u_{f,h})$:
\begin{align}
& C\Big(\frac{\| \D (\u_f) - \D (\u_{f,h}) \|^2_{L^r(\O_f)}}{c + \| \D ( \u_f ) \|^{2-r}_{L^r(\O_f)} + \| \D (\u_{f,h}) \|^{2-r}_{L^r(\O_f)}} \nonumber \\
& \qquad\qquad + (|\nu (\D (\u_f)) \D (\u_f) - \nu (\D (\u_{f,h}))\D(\u_{f,h})|, 
| \D(\u_f) - \D(\u_{f,h})|_{\O_f} \Big) \nonumber\\
& \qquad \leq \left( 2\nu (\D(\u_f)) \D(\u_f) - 2\nu (\D(\u_{f,h})) \D(\u_{f,h}), 
\D(\u_f) - \D(\u_{f,h})\right)_{\O_f}  \label{2nu}\\
& \qquad = \left( 2\nu (\D(\u_f)) \D(\u_f) - 2\nu (\D(\u_{f,h})) \D(\u_{f,h}), 
\D(\u_f) - \D(\v_{f,h})\right)_{\O_f}  \nonumber\\
& \qquad\qquad +\left( 2\nu (\D(\u_f)) \D(\u_f) - 2\nu (\D(\u_{f,h})) \D(\u_{f,h}),
\D(\v_{f,h}) - \D(\u_{f,h})\right)_{\O_f}  \nonumber\\
& \qquad =: J_1 + J_2 ,\qquad \forall \v_{f,h}\in \V_{f,h}, \label{stokes-error-1}
\end{align}
where we used the factor $2\nu$ in \eqref{2nu} in order that the term $J_2$
may be expressed in terms of $a_f(\cdot,\cdot)$.
The term $J_1$ can be bounded using \eqref{continuity} with 
$\s=\D(\u_f),\, \t= \D(\u_{f,h})$ and $\w=\D(\u_f)-\D(\v_{f,h})$: 
\begin{align}
J_1 & \leq C\left(| \nu (\D(\u_f)) \D(\u_f ) - \nu (\D(u_{f,h})) \D (\u_{f,h}) |, 
|\D(\u_f) - \D(\u_{f,h}) |\right)_{\O_f}^{1/r'} \nonumber \\
& \qquad
\times\norm{ \frac{| \D(\u_f) - \D (\u_{f,h}) |}{c + | \D(\u_f) | + | \D (\u_{f,h})| }}^{\frac{2-r}{r}}_{L^\infty(\O_f)} 
\left\|\D (\u_f) - \D (\v_{f,h})\right\|_{L^r(\O_f)} \nonumber \\
& \leq \epsilon (\left| \nu (\D(\u_f)) \D(\u_f) - \nu (\D (\u_{f,h})) \D (\u_{f,h}) \right|, \left| \D (\u_f ) - \D(\u_{f,h} ) \right|)_{\O_f} \nonumber \\
& \qquad + C \left\|  \frac{| \D(\u_f) - \D (\u_{f,h} ) |}{c + | \D (\u_f ) | + | \D (\u_{f,h}) | }\right\|^{2-r}_{L^\infty(\O_f)} \| \D(\u_f ) - \D(\v_{f,h} ) \|^r_{L^r(\O_f)}, \label{stokes-error-2}
\end{align}
where we used Young's inequality \eqref{young}.
We choose $\epsilon$ small enough and combine 
\eqref{stokes-error-1}--\eqref{stokes-error-2} to obtain
\begin{align} 
& \frac{\| \D (\u_f) - \D (\u_{f,h}) \|^2_{L^r(\O_f)}}{c + \| \D ( \u_f ) \|^{2-r}_{L^r(\O_f)} + \| \D (\u_{f,h}) \|^{2-r}_{L^r(\O_f)}} 
+ (| \nu (\D (\u_f)) \D (\u_f) - \nu (\D (\u_{f,h}))\D (\u_{f,h})|, | \D(\u_f) - \D(\u_{f,h}) |)_{\O_f} \nonumber \\
& \qquad\qquad \leq C \left(\left\| \frac{| \D(\u_f ) - \D (\u_{f,h} ) |}{c + | \D (\u_f ) | + | \D (\u_{f,h} ) |}\right\|^{2-r}_{L^\infty(\O_f)} \| \D(\u_f) - \D(\v_{f,h} ) \|^r_{L^r(\O_f)} + J_2\right). \label{stokes-error-3}
\end{align}
Similarly, to bound the error in the Darcy velocity we use
\eqref{monotonicity} and \eqref{continuity} with $\G(\x)=
\nu_{eff}(\x)\x$, $\s=\u_p,\, \t=\u_{p,h}$, and $ \w = \u_p-\v_{p,h},\,
\v_{p,h}\in \V_{p,h}$, to obtain
\begin{align} 
& \frac{\| \u_p - \u_{p,h} \|^2_{L^r(\O_p)}}{c + \| \u_p \|^{2-r}_{L^r(\O_p)} + \| \u_{p,h} \|^{2-r}_{L^r(\O_p)}} + 
(|\nu_{eff} (\u_p ) \u_p - \nu_{eff} (\u_{p,h} ) \u_{p,h}|, |\u_p - \u_{p,h}|)_{\O_p} 
\nonumber \\
& \qquad\qquad \leq C\left( \norm {\frac{| \u_p - \u_{p,h} | }{c + | \u_p | 
+ | \u_{p,h} | }}_{L^\infty(\O_p)}^{2-r}  \norm{ \u_p - \v_{p,h}}^r_{L^r(\O_p)} + J_4\right), \label{darcy-error-1}
\end{align}
where 
$$
J_4 := (\nu_{eff} (\u_p ) \kappa^{-1}\u_p - \nu_{eff} (\u_{p,h}) \kappa^{-1}\u_{p,h},
\v_{p,h} - \u_{p,h})_{\O_p}.
$$
The factor $\kappa^{-1}$ is introduced in the definition of $J_4$ in 
order that it
may be expressed in terms of $a_p^d(\cdot,\cdot)$.
Similarly, to bound the terms coming from the BJS condition, we set in 
\eqref{monotonicity} and \eqref{continuity},
$\G(\x)=\nu_{I}(\x)\x$,
$\s=((\u_f-\d_t\bbeta_p)\cdot \t_{f,j})\t_{f,j},\,
\t=((\u_{f,h}-\d_t\bbeta_{p,h})\cdot \t_{f,j})\t_{f,j}$ and $ \w =
((\u_f-\d_t\bbeta_p)\cdot
\t_{f,j})\t_{f,j}-((\v_{f,h}-\bxi_{p,h})\cdot \t_{f,j})\t_{f,j},\,
\v_{f,h}\in \V_{f,h},\, \bxi_{p,h}\in \X_{p,h}$, to obtain
\begin{align} 
& C\sum_{j=1}^{d-1}\frac{\| (\u_f - \d_t \bbeta_p)\cdot \t_{f,j} - 
(\u_{f,h} - \d_t \bbeta_{p,h})\cdot \t_{f,j} \|^2_{L^r(\Gamma_{fp})}}
{c + \| (\u_f - \d_t \bbeta_p)\cdot \t_{f,j} \|^{2-r}_{L^r(\Gamma_{fp)} } 
+ \| (\u_{f,h} - \d_t \bbeta_{p,h})\cdot \t_{f,j} \|_{L^r(\Gamma_{fp})}^{2-r}} 
\nonumber \\
& \qquad + C\sum_{j=1}^{d-1}\alpha_{BJS}
\<| \nu_{I} (((\u_f - \d_t \bbeta_p)\cdot\t_{f,j})\t_{f,j}) ((\u_f - \d_t \bbeta_p)\cdot\t_{f,j})\t_{f,j} \nonumber \\
& \qquad\qquad\qquad - \nu_{I} (((\u_{f,h} - \d_t \bbeta_{p,h})\cdot\t_{f,j})\t_{f,j}) ((\u_{f,h} - \d_t \bbeta_{p,h})\cdot\t_{f,j})\t_{f,j} |, \nonumber \\
& \qquad\qquad\qquad |   ((\u_f - \d_t \bbeta_p)\cdot\t_{f,j})\t_{f,j} - ((\u_{f,h} - \d_t \bbeta_{p,h})\cdot\t_{f,j})\t_{f,j} | \>_{\Gamma_{fp}} \nonumber \\
& \quad \leq\sum_{j=1}^{d-1}\norm{\frac{ | (\u_f - \d_t \bbeta_p)\cdot\t_{f,j} - (\u_{f,h} - \d_t \bbeta_{p,h})\cdot \t_{f,j} | }{c + | (\u_f - \d_t \bbeta_p)\cdot \t_{f,j} |  +  | (\u_{f,h} - \d_t \bbeta_{p,h})\cdot \t_{f,j}  | } }^{2-r}_{L^\infty(\Gamma_{fp})} \nonumber \\
& \qquad\qquad \times \| (\u_f - \d_t \bbeta_p)\cdot \t_{f,j} - (\v_{f,h} -  \bxi_{p,h})\cdot \t_{f,j} \|^r_{L^r(\Gamma_{fp})} + J_6, \label{bjs-error-1}
\end{align}
where
\begin{align*}
 J_6 &:= \sum_{j=1}^{d-1}\alpha_{BJS}\<\sqrt{\kappa^{-1}}(\nu_{I} (((\u_f-\d_t \bbeta_p )\cdot \t_{f,j})\t_{f,j}) (\u_f-\d_t \bbeta_p )\cdot \t_{f,j}  \\
&- \nu_{I} (((\u_{f,h}-\d_t \bbeta_{p,h} )\cdot{\t_{f,j}})\t_{f,j}) (\u_{f,h}-\d_t \bbeta_{p,h} )\cdot \t_{f,j}) ,  
(\v_{f,h}-\bxi_{p,h})\cdot \t_{f,j} - (\u_{f,h}-\d_t \bbeta_{p,h} )\cdot \t_{f,j}\>_{\Gamma_{fp}}.
\end{align*}
Combining \eqref{stokes-error-3}--\eqref{bjs-error-1} together with
the regularity of the solution from Theorems \ref{th-exist-uniq-weak-2}
and \ref{well-pos}, we obtain
\begin{align}
& \| \u_f - \u_{f,h} \|^2_{W^{1,r}(\O_f)} + \| \u_p - \u_{p,h} \|^2_{L^r(\O_p)}
+ | (\u_f - \d_t \bbeta_p)
- (\u_{f,h} - \d_t \bbeta_{p,h})|_{BJS}^2 + \mathcal{G} (\u, \u_h)
\label{J246} \\
& \quad 
\le C \left(\mathcal{E} (\u, \u_h)^r(\| \u_f - \v_{f,h} \|^r_{W^{1,r}(\O_f)}
+ \|\u_p - \v_{p,h}\|^r_{L^r(\O_p)} 
+ \|\d_t \bbeta_p - \bxi_{p,h}\|_{H^1(\O_p)}^r) + J_2+J_4+J_6\right), \nonumber
\end{align}
where we used the trace inequality. To bound the last three terms above,
note that 
\begin{align*}
J_2 = a_f(\u_f, \v_{f,h} - \u_{f,h}) - a_f (\u_{f,h}, \v_{f,h} - \u_{f,h} ), \qquad J_4 = a_p^d (\u_p, \v_{p,h} - \u_{p,h}) - a_p^d (\u_{p,h}, \v_{p,h} - \u_{p,h}), \\
J_6 = a_{BJS} (\u_f, \d_t \bbeta_p; \v_{f,h} - \u_{f,h},  \bxi_{p,h} -\d_t \bbeta_{p,h}) -  a_{BJS} (\u_{f,h}, \d_t \bbeta_{p,h}; \v_{f,h} - \u_{f,h}, \bxi_{p,h} -\d_t  \bbeta_{p,h}).
\end{align*}

\textbf{Step 2}. Bounds for $\|\bbeta_p(t) - \bbeta_{p,h}(t)\|_{H^1(\O_p)}$ and  
$\|p_{p}(t)-  p_{p,h}(t)\|_{L^2(\O_p)}$.  \\

We subtract \eqref{h-weak-1} from \eqref{h-cts-1} and test with 
$ (\v_{f,h} - \u_{f,h}, \v_{p,h} - \u_{p,h}, \bxi_{p,h} -  \d_t \bbeta_{p,h})$, to
obtain
\begin{align}
& J_2 + J_4 + J_6 = a_p^e(\bbeta_{p,h} - \bbeta_p,  \bxi_{p,h} - \d_t\bbeta_{p,h})+ b_f (\v_{f,h} - \u_{f,h}, p_{f,h} - p_f)+ \alpha b_p ( \bxi_{p,h}  - \d_t\bbeta_{p,h}, p_{p,h} - p_p)  \nonumber \\
& \qquad    + b_p (\v_{p,h} - \u_{p,h}, p_{p,h} - p_p)
   + b_\Gamma ( \v_{f,h} - \u_{f,h}, \v_{p,h} - \u_{p,h}, \bxi_{p,h}  -\d_t \bbeta_{p,h}; \lambda_h  - \lambda )  \nonumber \\
& \quad  = a_p^e(\bbeta_{p,h} - \bbeta_p,  \bxi_{p,h} - \d_t\bbeta_{p})+a_p^e(\bbeta_{p,h} - \bbeta_p,  \d_t\bbeta_{p} - \d_t\bbeta_{p,h})  + b_f (\v_{f,h} - \u_{f,h}, p_{f,h} - Q_{f,h}p_f) \nonumber \\
& \qquad  + b_f (\v_{f,h} - \u_{f,h}, Q_{f,h}p_{f} - p_{f})+ \alpha b_p ( \bxi_{p,h}  - \d_t\bbeta_{p,h}, p_{p,h} - Q_{p,h}p_p)+ \alpha b_p ( \bxi_{p,h}  - \d_t\bbeta_{p,h}, Q_{p,h}p_p - p_{p}) \nonumber \\
& \qquad + b_p (\v_{p,h} - \u_{p,h}, p_{p,h} - Q_{p,h}p_p)  + b_p (\v_{p,h} - \u_{p,h}, Q_{p,h}p_p - p_p) \nonumber \\
& \qquad + b_\Gamma ( \v_{f,h} - \u_{f,h}, \v_{p,h} - \u_{p,h}, \bxi_{p,h}  -\d_t \bbeta_{p,h}; \lambda_h  - Q_{\lam,h}\lambda ) \nonumber \\
& \qquad + b_\Gamma ( \v_{f,h} - \u_{f,h}, \v_{p,h} - \u_{p,h}, \bxi_{p,h}  -\d_t \bbeta_{p,h}; Q_{\lam,h}\lambda  - \lambda ).\label{error-1}
 \end{align}
Since $\nabla \cdot \V_{p,h}=W_{p,h}$ and $\V_{p,h}\cdot
\n_p|_{\Gam_{fp}}=\Lambda_h$, \eqref{darcy-pressure-int} and 
\eqref{l-multuplier-int} imply that 
\begin{align*}
b_p (\v_{p,h} - \u_{p,h}, Q_{p,h}p_p - p_p)= 0, \quad 
b_{\Gam}(0,\v_{p,h}-\u_{p,h},0;Q_{\lam,h}\lam-\lam) =0.
\end{align*}
Now we take 
$
 (\v_{f,h}, \v_{p,h}, \bxi_{p,h}) = (I_{f,h}\u_{f}, I_{p,h}\u_{p}, I_{s,h}\d_t\bbeta_{p})$.
Then \eqref{error-1} can be written as follows:
\begin{multline}
J_2 + J_4 + J_6 + a_p^e(\bbeta_p - \bbeta_{p,h},  \d_t\bbeta_{p} - \d_t\bbeta_{p,h})
= a_p^e(\bbeta_{p,h} - \bbeta_{p},  I_{s,h}\d_t\bbeta_{p} - \d_t\bbeta_{p})   \\
+ b_f (I_{f,h}\u_{f} - \u_{f,h}, p_{f,h} - Q_{f,h}p_f)+ b_f (I_{f,h}\u_{f} - \u_{f,h}, Q_{f,h}p_{f} - p_{f})+ \alpha b_p ( I_{s,h}\d_t\bbeta_{p}  - \d_t\bbeta_{p,h}, p_{p,h} - Q_{p,h}p_p)  \\
+ \alpha b_p (I_{s,h}\d_t\bbeta_{p}  - \d_t\bbeta_{p,h}, Q_{p,h}p_p - p_{p})
+ b_\Gamma ( I_{f,h}\u_{f} - \u_{f,h}, I_{p,h}\u_{p} - \u_{p,h}, I_{s,h}\d_t\bbeta_{p}  -\d_t \bbeta_{p,h}; \lambda_h  - Q_{\lam,h}\lambda ) \\
+ b_\Gamma ( I_{f,h}\u_{f} - \u_{f,h}, 0, I_{s,h}\d_t\bbeta_{p}  -\d_t \bbeta_{p,h}; Q_{\lam,h}\lambda  - \lambda )+ b_p (I_{p,h}\u_{p} - \u_{p,h}, p_{p,h} - Q_{p,h}p_p).\label{error-2}
\end{multline}
Note that due to \eqref{h-b-gamma} and \eqref{new-operator}, we have
\begin{align}
b_\Gamma ( I_{f,h}\u_{f} - \u_{f,h}, I_{p,h}\u_{p} - \u_{p,h}, I_{s,h}\d_t\bbeta_{p}  -\d_t \bbeta_{p,h}; \lambda_h  - Q_{\lam,h}\lambda ) = 0.
\label{error-3}
\end{align}
We next subtract \eqref{h-weak-2} from \eqref{h-cts-2} with the choice
$(w_{f,h},w_{p,h}) = (Q_{f,h}p_{f}-p_{f,h},Q_{p,h}p_{p}-p_{p,h})$:
\begin{align}
& s_0 (\d_t p_{p}-Q_{p,h}\d_tp_p,Q_{p,h}p_{p}-p_{p,h})_{\O_p} +s_0 (Q_{p,h}\d_tp_p -\d_tp_{p,h},Q_{p,h}p_{p}-p_{p,h})_{\O_p} \nonumber \\ 
& \qquad - \alpha b_p(\d_t\bbeta_{p}-I_{s,h}\d_t \bbeta_{p} ,Q_{p,h}p_{p}-p_{p,h}) -\alpha b_p(I_{s,h}\d_t \bbeta_{p} -\d_t\bbeta_{p,h} ,Q_{p,h}p_{p}-p_{p,h}) \nonumber \\
& \qquad - b_p(\u_{p}-I_{p,h}\u_p, Q_{p,h}p_{p}-p_{p,h})-  b_p(I_{p,h}\u_p-\u_{p,h}, Q_{p,h}p_{p}-p_{p,h}) \nonumber \\
& \qquad - b_f(\u_{f}-I_{f,h}\u_f,Q_{f,h}p_{f}-p_{f,h})- b_f(I_{f,h}\u_f-\u_{f,h},Q_{f,h}p_{f}-p_{f,h} )=0. \label{error-4}
\end{align}
By \eqref{darcy-pressure-int} and \eqref{darcy-velocity-int2}, we have
\begin{align*}
s_0 (\d_t p_{p}-Q_{p,h}\d_tp_p,Q_{p,h}p_{p}-p_{p,h})_{\O_p} = 0, \quad 
b_p(\u_{p}-I_{p,h}\u_p, Q_{p,h}p_{p}-p_{p,h})=0.
\end{align*}
Then \eqref{error-4} becomes
\begin{align}
& s_0 (Q_{p,h}\d_tp_p -\d_tp_{p,h},Q_{p,h}p_{p}-p_{p,h})_{\O_p}
= \alpha b_p(\d_t\bbeta_{p}-I_{s,h}\d_t \bbeta_{p},Q_{p,h}p_{p}-p_{p,h}) 
\nonumber \\
& \qquad 
+ \alpha b_p(I_{s,h}\d_t \bbeta_{p} -\d_t\bbeta_{p,h} ,Q_{p,h}p_{p}-p_{p,h}) 
+ b_p(I_{p,h}\u_p-\u_{p,h}, Q_{p,h}p_{p}-p_{p,h})
\nonumber \\
& \qquad + b_f(\u_{f}-I_{f,h}\u_f,Q_{f,h}p_{f}-p_{f,h}) 
+ b_f(I_{f,h}\u_f-\u_{f,h},Q_{f,h}p_{f}-p_{f,h} ).\label{error-5}
\end{align}
We now combine \eqref{error-2}, \eqref{error-3}, and \eqref{error-5}, to obtain
\begin{align}
& J_2 + J_4 + J_6 + a_p^e(\bbeta_{p,h} - \bbeta_{p},  \d_t\bbeta_{p,h} - \d_t\bbeta_{p})+s_0 (Q_{p,h}\d_tp_p -\d_tp_{p,h},Q_{p,h}p_{p}-p_{p,h})_{\O_p}\nonumber \\
& \quad = a_p^e(\bbeta_{p,h} - \bbeta_{p},  I_{s,h}\d_t\bbeta_{p} - \d_t\bbeta_{p}) + b_f(\u_{f}-I_{f,h}\u_f,Q_{f,h}p_{f}-p_{f,h})
+ b_f (I_{f,h}\u_{f} - \u_{f,h}, Q_{f,h}p_{f} - p_{f})  \nonumber \\
& \quad \qquad + \alpha b_p (I_{s,h}\d_t\bbeta_{p}  - \d_t\bbeta_{p}, Q_{p,h}p_p - p_{p,h})+ \alpha b_p ( I_{s,h}\d_t\bbeta_{p}  - \d_t\bbeta_{p,h}, Q_{p,h} p_{p} - p_p) \nonumber \\
& \quad \qquad 
+ \<(I_{f,h}\u_{f}- \u_{f,h})\cdot\n_f,Q_{\lam,h}\lambda  - \lambda\>_{\Gamma_{fp}}
+ \<(I_{s,h}\d_t\bbeta_{p}  -\d_t \bbeta_{p,h})\cdot\n_p,Q_{\lam,h}\lambda  - \lambda\>_{\Gamma_{fp}}. \label{error-6}
\end{align}
We next bound the first four and the sixth terms of the right using
Young's inequality \eqref{young}. We note that the velocity and
displacement errors are controlled in $L^2(0,T)$, so the terms
involving such errors are bounded using \eqref{young} with $p = q = 2$.
The pressure and Lagrange multiplier errors are controlled in $L^{r'}(0,T)$,
so for such terms we use \eqref{young} with $p = r'$ and $q = r$. We have
 \begin{align}
& a_p^e(\bbeta_{p,h} - \bbeta_{p},  I_{s,h}\d_t\bbeta_{p} - \d_t\bbeta_{p}) 
\le C(\| \bbeta_{p,h} - \bbeta_{p} \|^2_{H^1(\O_p)} 
   + \| I_{s,h}\d_t\bbeta_{p}-\d_t \bbeta_p \|^2_{H^1(\O_p)}), \nonumber \\
& b_f(\u_{f}-I_{f,h}\u_f,Q_{f,h}p_{f}-p_{f,h}) 
\le \epsilon_1 \|p_{f,h}-Q_{f,h}p_f\|^{r'}_{L^{r'}(\O_f)}
+ C \|I_{f,h}\u_f-\u_f\|^r_{W^{1,r}(\O_f)}, \nonumber \\
& b_f (I_{f,h}\u_{f} - \u_{f,h}, Q_{f,h}p_{f} - p_{f}) \le
\epsilon_2 \| \u_{f} - \u_{f,h} \|^2_{W^{1,r}(\O_f)} \nonumber \\
& \qquad\qquad\qquad\qquad\qquad\qquad\qquad\quad
+ C(\|I_{f,h}\u_f-\u_f\|^2_{W^{1,r}(\O_f)} + \| Q_{f,h}p_f  - p_f \|^{2}_{L^{r'}(\O_f)}),
\nonumber \\
& \alpha b_p (I_{s,h}\d_t\bbeta_{p}  - \d_t\bbeta_{p}, Q_{p,h}p_p - p_{p,h})  \le
\epsilon_1 \|p_{p,h} - Q_{p,h}p_p\|^{r'}_{L^{r'}(\O_p)} 
+ C \|I_{s,h}\d_t\bbeta_{p}  - \d_t\bbeta_{p}\|^{r}_{H^1(\O_p)}, \nonumber \\
& \<(I_{f,h}\u_{f}- \u_{f,h})\cdot\n_f,Q_{\lam,h}\lambda  - \lambda\>_{\Gamma_{fp}} \le
\epsilon_2 \| \u_{f} - \u_{f,h} \|^2_{W^{1,r}(\O_f)} \nonumber \\
& \qquad\qquad\qquad\qquad\qquad\qquad\qquad\qquad\quad
+ C(\|I_{f,h}\u_f-\u_f\|^2_{W^{1,r}(\O_f)} + \|Q_{\lam,h}\lambda  - \lambda\|^2_{L^{r'}(\Gamma_{fp})}).
\label{error-7}
\end{align}
We combine \eqref{error-6} and \eqref{error-7} and integrate in time from 
$0$ to $t\in (0,T]$:
\begin{align}
 & \frac{1}{2}\left( a_p^e(\bbeta_p(t) - \bbeta_{p,h}(t), \bbeta_{p}(t) - \bbeta_{p,h}(t)) +s_0\|Q_{p,h}p_{p}(t)-  p_{p,h}(t)\|^2_{L^2(\O_p)} \right) 
 +\int_0^t\left(J_2 + J_4 + J_6\right)\, ds \nonumber \\
 & \quad \leq \int_0^t\left(\epsilon_1 \|p_{f,h}-Q_{f,h}p_f\|^{r'}_{L^{r'}(\O_f)}
 + \epsilon_1 \|p_{p,h} - Q_{p,h}p_p\|^{r'}_{L^{r'}(\O_p)}
+\epsilon_2\| \u_f - \u_{f,h} \|^2_{W^{1,r}(\O_f)}\right)\,ds \nonumber \\
& \qquad 
+ \frac{1}{2}\left( a_p^e(\bbeta_p(0) - \bbeta_{p,h}(0), \bbeta_{p}(0) - \bbeta_{p,h}(0)) +s_0\|Q_{p,h}p_{p}(0)-  p_{p,h}(0)\|^2_{L^2(\O_p)} \right) \nonumber \\
& \qquad
 +C \int_0^t \Big( \| \bbeta_{p,h} - \bbeta_p \|^2_{H^1(\O_p)} + \| I_{s,h}\d_t\bbeta_{p} -\d_t \bbeta_p \|^2_{H^1(\O_p)}+ \| I_{s,h}\d_t\bbeta_{p} -\d_t \bbeta_p \|^r_{H^1(\O_p)} \nonumber \\
& \qquad
 +  \| Q_{f,h}p_f  - p_f \|^{2}_{L^{r'}(\O_f)} 
+ \|Q_{\lam,h}\lambda  - \lambda\|^2_{L^{r'}(\Gamma_{fp})}
+\|I_{f,h}\u_{f}-\u_f\|^2_{W^{1,r}(\O_f)} +\|I_{f,h}\u_{f}-\u_f\|^r_{W^{1,r}(\O_f)} 
\Big) ds \nonumber \\
& \qquad
 +\int_0^t \left(\alpha b_p ( I_{s,h}\d_t\bbeta_{p} - \d_t\bbeta_{p,h}, Q_{p,h}p_p - p_{p}) 
+ \<(I_{s,h}\d_t\bbeta_{p}  -\d_t \bbeta_{p,h})\cdot\n_p,Q_{\lam,h}\lambda  - \lambda\>_{\Gamma_{fp}} \right) ds. \label{error-9}
 \end{align}
For the last two terms on the right hand side we use integration by parts:
 \begin{align}
& \int_0^t \left(\alpha b_p ( I_{s,h}\d_t\bbeta_{p}  - \d_t\bbeta_{p,h}, Q_{p,h}p_p - p_{p}) 
+ \<(I_{s,h}\d_t\bbeta_{p}  -\d_t \bbeta_{p,h})\cdot\n_p,Q_{\lam,h}\lambda  - \lambda\>_{\Gamma_{fp}} \right) ds \nonumber \\
& \quad =  \alpha b_p ( I_{s,h}\bbeta_{p}  - \bbeta_{p,h}, Q_{p,h}p_p - p_{p})\Big|_{0}^{t} 
 + \<(I_{s,h}\bbeta_{p}  - \bbeta_{p,h})\cdot\n_p,Q_{\lam,h}\lambda  - \lambda\>_{\Gamma_{fp}} \Big|_{0}^{t}\nonumber \\
& \qquad
  -\int_0^t \left(\alpha b_p (I_{s,h}\bbeta_{p}  - \bbeta_{p,h}, Q_{p,h}\d_tp_p - \d_tp_{p}) 
 + \<(I_{s,h}\bbeta_{p} - \bbeta_{p,h})\cdot\n_p,Q_{\lam,h}\d_t\lambda  - \d_t\lambda\>_{\Gamma_{fp}} \right) ds \label{error-9-10}
\end{align}
and bound the terms on the right hand side above as follows:
\begin{align}
& \alpha b_p ( I_{s,h}\bbeta_{p}  - \bbeta_{p,h}, Q_{p,h}p_p - p_{p})\Big|_{0}^{t} 
 + \<(I_{s,h}\bbeta_{p}  - \bbeta_{p,h})\cdot\n_p,Q_{\lam,h}\lambda  - \lambda\>_{\Gamma_{fp}} \Big|_{0}^{t} 
\leq \epsilon_2 \| \bbeta_{p}(t)  - \bbeta_{p,h}(t)\|^2_{H^1(\O_p)}
\nonumber \\
& \qquad +C\left(\|I_{s,h}\bbeta_{p}(t) - \bbeta_p(t)\|^2_{H^1(\O_p)}
+ \|Q_{p,h}p_p(t) - p_{p}(t)\|^2_{L^{r'}(\O_p)}
+ \|Q_{\lam,h}\lam(t) -\lam(t)\|^2_{L^{r'}(\Gam_{fp})}\right. \nonumber \\
 & \qquad +\left.
  \| I_{s,h}\bbeta_{p}(0)  - \bbeta_{p,h}(0)\|^2_{H^1(\O_p)}
+ \| Q_{p,h}p_p(0) - p_{p}(0)\|^2_{L^{r'}(\O_p)} 
+ \|Q_{\lam,h}\lam(0) -\lam(0)\|^2_{L^{r'}(\Gam_{fp})}\right),
\label{error-10-1}
\end{align}
\begin{align}
& \int_0^t \left(\alpha b_p (I_{s,h}\bbeta_{p}  - \bbeta_{p,h}, Q_{p,h}\d_tp_p - \d_tp_{p}) 
 + \<(I_{s,h}\bbeta_{p} - \bbeta_{p,h})\cdot\n_p,Q_{\lam,h}\d_t\lambda  - \d_t\lambda\>_{\Gamma_{fp}} \right) ds \nonumber \\
& \qquad \le 
C \int_0^t \left( \| \bbeta_{p}  - \bbeta_{p,h}\|^2_{H^1(\O_p)} 
+ \|I_{s,h}\bbeta_{p} - \bbeta_p\|^2_{H^1(\O_p)} \right. \nonumber \\
& \qquad\qquad\qquad \left.
+ \|Q_{p,h}\d_tp_p - \d_t p_{p}\|^2_{L^{r'}(\O_p)}
+ \|Q_{\lam,h}\d_t \lam -\d_t \lam\|^2_{L^{r'}(\Gam_{fp})}   \right) ds.
\label{error-10-2}
 \end{align}
We choose $p_{p,h}(0) = Q_{p,h}p_p(0),\, \bbeta_{p,h}(0) =I_{s,h}\bbeta_{p}(0)$.
Combining \eqref{error-9}--\eqref{error-10-2}, we obtain
\begin{align}
& \|\bbeta_p(t) - \bbeta_{p,h}(t)\|^2_{H^1(\O_p)} 
+  s_0\|Q_{p,h}p_{p}(t)-  p_{p,h}(t)\|^2_{L^2(\O_p)} 
 +\int_0^t\left(J_2 + J_4 + J_6\right) ds \nonumber \\
& \quad \leq \epsilon_2\left( \| \bbeta_{p}(t)  - \bbeta_{p,h}(t)\|^2_{H^1(\O_p)} 
+\int_0^t\  \| \u_f - \u_{f,h} \|^2_{W^{1,r}(\O_f)} \right) 
+ C\int_0^t \| \bbeta_p - \bbeta_{p,h} \|^2_{H^1(\O_p)} ds
\nonumber \\
& \qquad +\epsilon_1\int_0^t\left(\|p_{f,h}-Q_{f,h}p_f\|^{r'}_{L^{r'}(\O_f)} +\|p_{p,h} - Q_{p,h}p_p\|^{r'}_{L^{r'}(\O_p)}\right) ds \nonumber \\ 
& \qquad  + C\int_0^t \Big( 
\|I_{s,h}\bbeta_{p} - \bbeta_p\|^2_{H^1(\O_p)}  
+ \| I_{s,h}\d_t\bbeta_{p} -\d_t \bbeta_p \|^2_{H^1(\O_p)}
+ \| I_{s,h}\d_t\bbeta_{p} -\d_t \bbeta_p \|^r_{H^1(\O_p)} \nonumber \\
& \qquad  
+ \|Q_{\lam,h} \lam - \lam\|^2_{L^{r'}(\Gam_{fp})} 
 +\|Q_{p,h}\d_tp_p - \d_t p_{p}\|^2_{L^{r'}(\O_p)}+ \|Q_{\lam,h}\d_t \lam -\d_t \lam\|^2_{L^{r'}(\Gam_{fp})}  \nonumber \\
& \qquad  +  \| Q_{f,h}p_f  - p_f \|^{2}_{L^{r'}(\O_f)} +\|I_{f,h}\u_{f}-\u_f\|^2_{W^{1,r}(\O_f)} +\|I_{f,h}\u_{f}-\u_f\|^r_{W^{1,r}(\O_f)} \Big)\, ds \nonumber \\
& \qquad  + C \left(\|I_{s,h}\bbeta_{p}(t) - \bbeta_p(t)\|^2_{H^1(\O_p)}
+ \|Q_{p,h}p_p(t) - p_{p}(t)\|^2_{L^{r'}(\O_p)}
+ \|Q_{\lam,h}\lam(t) -\lam(t)\|^2_{L^{r'}(\Gam_{fp})}\right. \nonumber \\
& \qquad + \left. 
\|I_{s,h}\bbeta_{p}(0) - \bbeta_p(0)\|^2_{H^1(\O_p)} 
+\|Q_{p,h}p_p(0)-p(0) \|^2_{L^{r'}(\O_p)} 
+\|Q_{\lam,h}\lam(0) -\lam(0)\|^2_{L^{r'}(\Gam_{fp})} \right).\label{error-11}
\end{align}

\textbf{Step 3}. Bounds for $\|p_f - p_{f,h}\|_{L^{r'}(\O_f)}$,
$\|p_p - p_{p,h}\|_{L^{r'}(\O_p)}$ and  $\|\lam - \lam_{h}\|_{L^{r'}(\Gam_{fp})}$. \\

Next, using the inf-sup condition \eqref{inf-sup-stokes-darcy-h}, we obtain 
\begin{align*}
&\| (p_{f,h} - Q_{f,h}p_f,p_{p,h} - Q_{p,h}p_p,\lambda_h - Q_{\lam,h}\lambda)\|_{W_f\times W_p \times \Lambda_h} 
\nonumber \\
& \ \leq C \sup_{(\v_{f,h},\v_{p,h}) \in \V_{f,h}\times \V_{p,h}} \frac{b_f(\v_{f,h}, p_{f,h} -Q_{f,h}p_f) + b_p (\v_{p,h}, p_{p,h} - Q_{p,h}p_p ) + b_{\Gamma} (\v_{f,h}, \v_{p,h}, \textbf{0}; \lambda_h - Q_{\lam,h}\lam)}{\|(\v_{f,h},\v_{p,h})\|_{\V_f\times \V_p}} 
\nonumber \\
& \ =C \sup_{(\v_{f,h},\v_{p,h}) \in \V_{f,h}\times \V_{p,h}} - \Big[ \frac{a_f(\u_{f,h} \v_{f,h}) - a_f(\u_f, \v_{f,h})}{\|(\v_{f,h},\v_{p,h})\|_{\V_f\times \V_p}} + \frac{a_p^d (\u_{p,h}, \v_{p,h} ) - a_p^d(\u_p, \v_{p,h})}{\|(\v_{f,h},\v_{p,h})\|_{\V_f\times \V_p}} 
\nonumber \\
& \quad \quad + \frac{a_{BJS}(\u_{f,h}, \d_t \bbeta_{p,h}; \v_{f,h}, \textbf{0}) - a_{BJS} (\u_f, \d_t \bbeta_p; \v_{f,h}, \textbf{0})}{\|(\v_{f,h},\v_{p,h})\|_{\V_f\times \V_p}} 
\nonumber \\
& \quad \quad + \frac{b_f(\v_{f,h}, Q_{f,h}p_f-p_f) + b_p (\v_{p,h},Q_{p,h}p_p-p_p ) + b_{\Gamma} (\v_{f,h}, \v_{p,h}, \textbf{0};  Q_{\lam,h}\lam -\lam)}{\|(\v_{f,h},\v_{p,h})\|_{\V_f\times \V_p}}\Big]
\nonumber \\
& \  \leq C \left(\mathcal{E}(\u, \u_h) \mathcal{G}(\u, \u_h)^{1/r'} + \|Q_{f,h}p_f-p_f\|_{L^{r'}(\O_f)}+ \|Q_{p,h}p_p-p_p\|_{L^{r'}(\O_p)}+\|Q_{\lam,h}\lam-\lam\|_{L^{r'}(\Gam_{fp})}\right),  
\end{align*}
using \eqref{continuity} for the last inequality. 
Hence, as $\mathcal{E}(\u, \u_h) \le (d + 1)$,
\begin{align}
& \epsilon_1\int_0^t\left(
\|p_{f,h}-Q_{f,h}p_f\|^{r'}_{L^{r'}(\O_f)}
+ \|p_{p,h} - Q_{p,h}p_p\|^{r'}_{L^{r'}(\O_p)} 
+ \|\lambda_h - Q_{\lam,h}\lambda\|^{r'}_{L^{r'}(\Gam_{fp})}
\right)\, \nonumber \\
& \leq \epsilon_1C \int_0^t \left( \mathcal{G}(\u, \u_h) + \|Q_{f,h}p_f-p_f\|^{r'}_{L^{r'}(\O_f)}+ \|Q_{p,h}p_p-p_p\|^{r'}_{L^{r'}(\O_p)}+\|Q_{\lam,h}\lam-\lam\|^{r'}_{L^{r'}(\Gam_{fp})}\right) ds .\label{error-15}
\end{align}

\textbf{Step 4}. Combination of the bounds. \\

We now integrate \eqref{J246} in time, combine it with
\eqref{error-11} and \eqref{error-15}, take $\epsilon_1$ small
enough, then $\epsilon_2$ small enough, and apply Gronwall's
inequality, to obtain
\begin{align*} 
& \| \u_f - \u_{f,h} \|^2_{L^2(0,T;W^{1,r}(\O_f))}
+
\| \u_p - \u_{p,h} \|^2_{L^2(0,T;L^r(\O_p))}+|(\u_f - \d_t \bbeta_p) - (\u_{f,h} - \d_t \bbeta_{p,h}) |^2_{L^2(0,T;BJS)} \\ 
& \qquad +\|Q_{f,h}p_f - p_{f,h}\|^{r'}_{L^{r'}(0,T;L^{r'}(\O_f))}+\|Q_{p,h}p_p - p_{p,h}\|^{r'}_{L^{r'}(0,T;L^{r'}(\O_p))} +\|Q_{\lam,h}\lam - \lam_{h}\|^{r'}_{L^{r'}(0,T;L^{r'}(\Gam_{fp}))} \\
& \qquad +\|\bbeta_p - \bbeta_{p,h}\|^2_{L^{\infty}(0,T;H^1(\O_p))} +s_0\|Q_{p,h}p_{p}-  p_{p,h}\|^2_{L^{\infty}(0,T;L^2(\O_p))}+\|\mathcal{G}(\u,\u_h)\|_{L^1(0,T)}\\
& \quad \leq C\exp(T)\Big(
\|\u_f-I_{f,h}\u_{f}\|^2_{L^2(0,T;W^{1,r}(\O_f))}+\|\u_f-I_{f,h}\u_{f}\|^r_{L^r(0,T;W^{1,r}(\O_f))} \\
& \qquad
+ \| \bbeta_{p} -I_{s,h}\bbeta_{p} \|^2_{L^2(0,T;H^1(\O_p))}
+\| \u_p - I_{p,h}\u_{p}\|^r_{L^r(0,T;L^r(\O_p))}+ \|  \d_t \bbeta_p -I_{s,h}\d_t\bbeta_{p} \|^r_{L^r(0,T;H^1(\O_p))}\\
& \qquad  + \|  \d_t \bbeta_p -I_{s,h}\d_t\bbeta_{p} \|^2_{L^2(0,T;H^1(\O_p))}
+\| Q_{f,h}p_f  - p_f \|^2_{L^2(0,T;L^{r'}(\O_f))} + \|Q_{\lam,h}\lambda  - \lambda\|^2_{L^2(0,T;L^{r'}(\Gam_{fp}))}   \\
& \qquad + \| Q_{p,h}\d_tp_p  - \d_t p_p \|^2_{L^2(0,T;L^{r'}(\O_p))}
+\|Q_{\lam,h}\d_t\lambda  - \d_t \lambda\|^2_{L^2(0,T;L^{r'}(\Gam_{fp}))}+  \| \bbeta_{p}  - I_{s,h}\bbeta_{p}\|^2_{L^{\infty}(0,T;H^1(\O_p))}\\
& \qquad + \|Q_{p,h}p_p - p_{p}\|^2_{L^{\infty}(0,T;L^{r'}(\O_p))}+  \|Q_{\lam,h}\lam -\lam\|^2_{L^{\infty}(0,T;L^{r'}(\Gam_{fp}))}
 + \|Q_{f,h}p_f-p_f\|^{r'}_{L^{r'}(0,T;L^{r'}(\O_f))}\\
 & \qquad + \|Q_{p,h}p_p-p_p\|^{r'}_{L^{r'}(0,T;L^{r'}(\O_p))}+\|Q_{\lam,h}\lam-\lam\|^{r'}_{L^{r'}(0,T;L^{r'}(\Gam_{fp}))} \Big).
\end{align*}
The assertion of the theorem follows from the approximation bounds
\eqref{stokesPresProj}--\eqref{LMProj} and
\eqref{stokesVel}--\eqref{darcyVel} and the use of the triangle
inequality for the pressure error terms. 
\end{proof}

\section{Numerical results}
In this section we present numerical results that illustrate the
behavior of the method. We discretize the problem
\eqref{h-weak-1}--\eqref{h-b-gamma} in time using the Backward Euler
scheme with a time step $\tau$. For spatial discretization we use the
$\mathcal{P}_1b-\mathcal{P}_1b$ MINI elements for Stokes, the lowest
order Raviart-Thomas spaces $\mathcal{RT}_0-\mathcal{P}_0$ for Darcy
\cite{BoffiBrezziFortin}, continuous piecewise linears $\mathcal{P}_1$
for the displacement, and piecewise constants $\mathcal{P}_0$ for the
Lagrange multiplier.  We neglect the nonlinearity in the BJS condition
\eqref{Gamma-fp-1} and handle the nonlinearity in Stokes and Darcy
using the Picard iteration. The computations are performed on
triangular grids using the finite element package FreeFem++ \cite{freefem}. 

\subsection{Example 1: application to industrial filters}
Our first example is motivated by an application to industrial
filters, see \cite{ervin2011coupling}.  The units in this example are
dimensionless. We consider a computational domain
$\Omega=(0,2)\times(0,1)$, where $\O_f=(0,1)\times (0,1)$ is the fluid
region and $\O_p=(1,2)\times (0,1)$ is the poroelastic region, which
models the filter. The flow is driven by a pressure drop: on the left
boundary of $\O_f$ we set $p_{in}=1$ and on the right boundary of
$\O_p$, $p_{out}=0$, which is also chosen as the initial condition for
the Darcy pressure. Along the top and bottom boundaries, we impose a
no-slip boundary condition for the Stokes flow and a no-flow boundary
condition for the Darcy flow. We also set zero displacement initial
and boundary conditions. We set
$\lam_p=\mu_p=s_0=\alpha=\alpha_{BJS}=1.0$ and $\kappa=\bI$. We
consider the Cross model for the viscosity in Stokes and Darcy:
\begin{equation}\label{cross}
\nu_f(|\D(\textbf{u}_f)|) = \nu_{f,\infty}+\frac{\nu_{f,0}-\nu_{f,\infty}}{1+K_f|\D(\textbf{u}_f)|^{2-r_f}}, \quad
\nu_p(|\textbf{u}_p|) = \nu_{p,\infty}+\frac{\nu_{p,0}-\nu_{p,\infty}}{1+K_p|\textbf{u}_p|^{2-r_p}},
\end{equation}
where the parameters are chosen as $K_f=K_p=1,\,
\nu_{f,\infty}=\nu_{p,\infty}=1,\, \nu_{f,0}=\nu_{p,0}=10$,
$r_f=r_p=1.35$.
The simulation time is $T = 1.0$ and the time step $\tau = 0.01$. To
verify the convergence estimate from Theorem~\ref{conv-thm}, we
compute a reference solution, obtained on a mesh with characteristic
size $h = 1/320$. Table \ref{table:convergence} shows the relative
errors and rates of convergence for the solutions computed with mesh
sizes $h = 1/20, 1/40, 1/80$ and $1/160$ . Since we use bounded
functions to model the viscosity in both regions, we compute the error
norms using $r=r'=2$. As seen from the table, the results agree with
theory, i.e. we observe at least first convergence rate for all
variables. We note that the time step is sufficiently small, so that
the time discretization error does not have an effect on the convergence.

We also investigate the non-Newtonian effect by comparing to the
linear analogue of the method \eqref{h-weak-1}--\eqref{h-b-gamma}.
For visualization we use the solutions computed with mesh size
$h=1/40$. We set the viscosity in the linear case to be $\nu_f^{lin} =
\nu_{f,\infty} = 1$ and $\nu_p^{lin} = \nu_{p,\infty} = 1$.  In
Figure~\ref{fig:1} we plot the non-Newtonian pressure and velocity at
the final time. We observe channel-like flow in the fluid region,
which slows down and diffuses as the fluid enters the poroelastic
region. The pressure drop occurs mostly in the fluid region. In
Figure~\ref{fig:2} we plot the nonlinear viscosity at the first and
last time steps. We note that the viscosity is highest in the middle
of the fluid domain and it decreases towards the boundary, which is
due to the fact that the strain rate increases towards the
boundary. On the other hand, the viscosity does not vary as much in
the poroelastic domain due to the small changes in velocity. These
observations agree with the conclusions in
\cite{ervin2011coupling}. In Figures \ref{fig:3} and \ref{fig:4} we
plot the difference {\em nonlinear} -- {\em linear} solution, where
colors represent the magnitude of the corresponding difference and
arrows represent the direction. We observe that the higher viscosity
in the non-Newtonian model results in lower Stokes velocity, as shown
on Figure \ref{fig:3}(b), which in turn leads to lower displacement,
see Figure \ref{fig:4}(b).

\begin{table}
	\begin{center}
		\begin{tabular}{|c|cc|cc|cc|}
			\hline
			& \multicolumn{2}{c|}{$\frac{\|\u^{ref}_{f,h} -\u_{f,h}\|_{l^2(0,T;H^1(\O_f))}}{\|\u^{ref}_{f,h}\|_{l^2(0,T;H^1(\O_f))}}$} & \multicolumn{2}{c|}{$\frac{\|\u^{ref}_{p,h} -\u_{p,h}\|_{l^2(0,T;L^2(\O_p))}}{\|\u^{ref}_{p,h}\|_{l^2(0,T;L^2(\O_p))}}$} & \multicolumn{2}{c|}{$\frac{\|p^{ref}_{f,h} -p_{f,h}\|_{l^2(0,T;L^2(\O_f))}}{\|p^{ref}_{f,h}\|_{l^2(0,T;L^2(\O_f))}}$} \\ 
\hline
			$h$ & error & order & error & order & error & order \\ 
			\hline
			1/20  & 4.83E-03 & $-$ & 1.55E-01  & $-$  & 2.75E-02 & $-$ \\ 
			1/40  & 2.31E-03 & 1.06 & 8.63E-02 & 0.85 & 1.03E-02 & 1.41 \\ 
			1/80  & 1.04E-03 & 1.16 & 4.08E-02 & 1.08 & 4.62E-03 & 1.16 \\ 
			1/160 & 3.94E-04 & 1.40 & 2.07E-02 & 0.98 & 2.14E-04 & 1.11 \\ 
			\hline
			&\multicolumn{2}{c|}{$\frac{\|p^{ref}_{p,h} -p_{p,h}\|_{l^{2}(0,T;L^2(\O_p))}}{\|p^{ref}_{p,h}\|_{l^{2}(0,T;L^2(\O_p))}}$} & \multicolumn{2}{c|}{$\frac{\|p^{ref}_{p,h} -p_{p,h}\|_{l^{\infty}(0,T;L^2(\O_p))}}{\|p^{ref}_{p,h}\|_{l^{\infty}(0,T;L^2(\O_p))}}$} & \multicolumn{2}{c|}
{$\frac{\|\bbeta^{ref}_{p,h} -\bbeta_{p,h}\|_{l^{\infty}(0,T;H^1(\O_p))}}{\|\bbeta^{ref}_{p,h}\|_{l^{\infty}(0,T;H^1(\O_p))}}$}  \\ 
\hline
			$h$ & error & order & error & order & error &order \\ 
			\hline 
			1/20  & 4.10E-02 & $-$  & 1.15E-01 & $-$  & 4.98E-02 & $-$     \\ 
			1/40  & 1.92E-02 & 1.10 & 5.28E-02 & 1.12 & 2.88E-02 & 0.79     \\ 
			1/80  & 8.24E-03 & 1.22 & 2.25E-02 & 1.23 & 1.61E-02 & 0.84     \\ 
			1/160 & 2.75E-03 & 1.58 & 7.48E-03 & 1.59 & 6.59E-03 & 1.29   \\  
			\hline
		\end{tabular}
	\end{center}
	\caption{Convergence for ($\mathcal{P}_1b\times\mathcal{P}_1b)
\times(\mathcal{RT}_0\times\mathcal{P}_0)\times\mathcal{P}_1\times\mathcal{P}_0$ elements.} \label{table:convergence}
\end{table}

\begin{figure}
	\centering
	\begin{subfigure}{0.45\textwidth}
		\includegraphics[width=\textwidth]{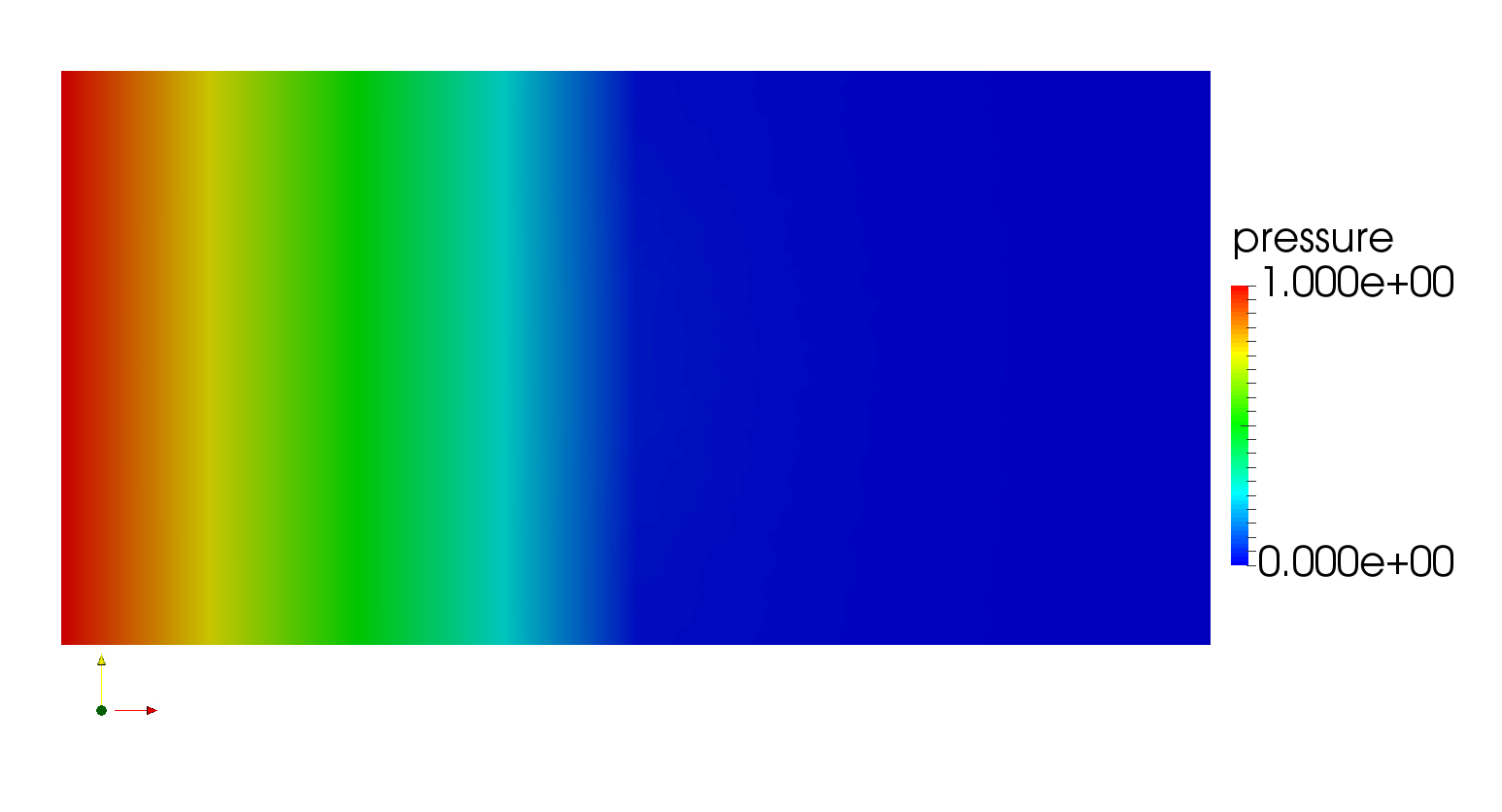}
		\caption{pressure}
	\end{subfigure}
	~ 
	\begin{subfigure}{0.45\textwidth}
		\includegraphics[width=\textwidth]{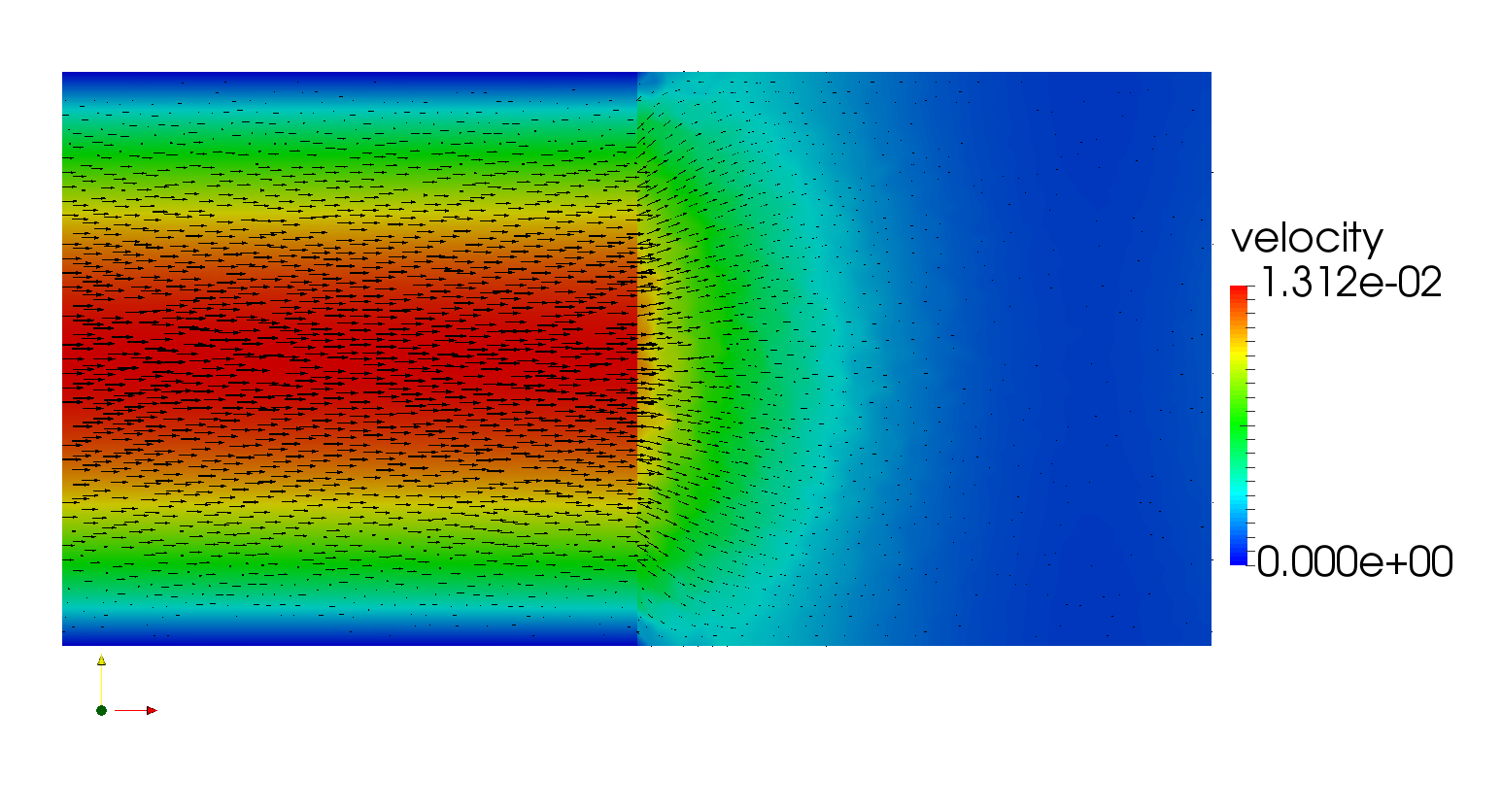}
		\caption{velocity vector (arrows) and magnitude (color)}
	\end{subfigure}
	\caption{Example 1, non-Newtonian pressure and velocity solutions at time $t=1$.}\label{fig:1}
\end{figure}

\begin{figure}
	\centering
	\begin{subfigure}{0.45\textwidth}
		\includegraphics[width=\textwidth]{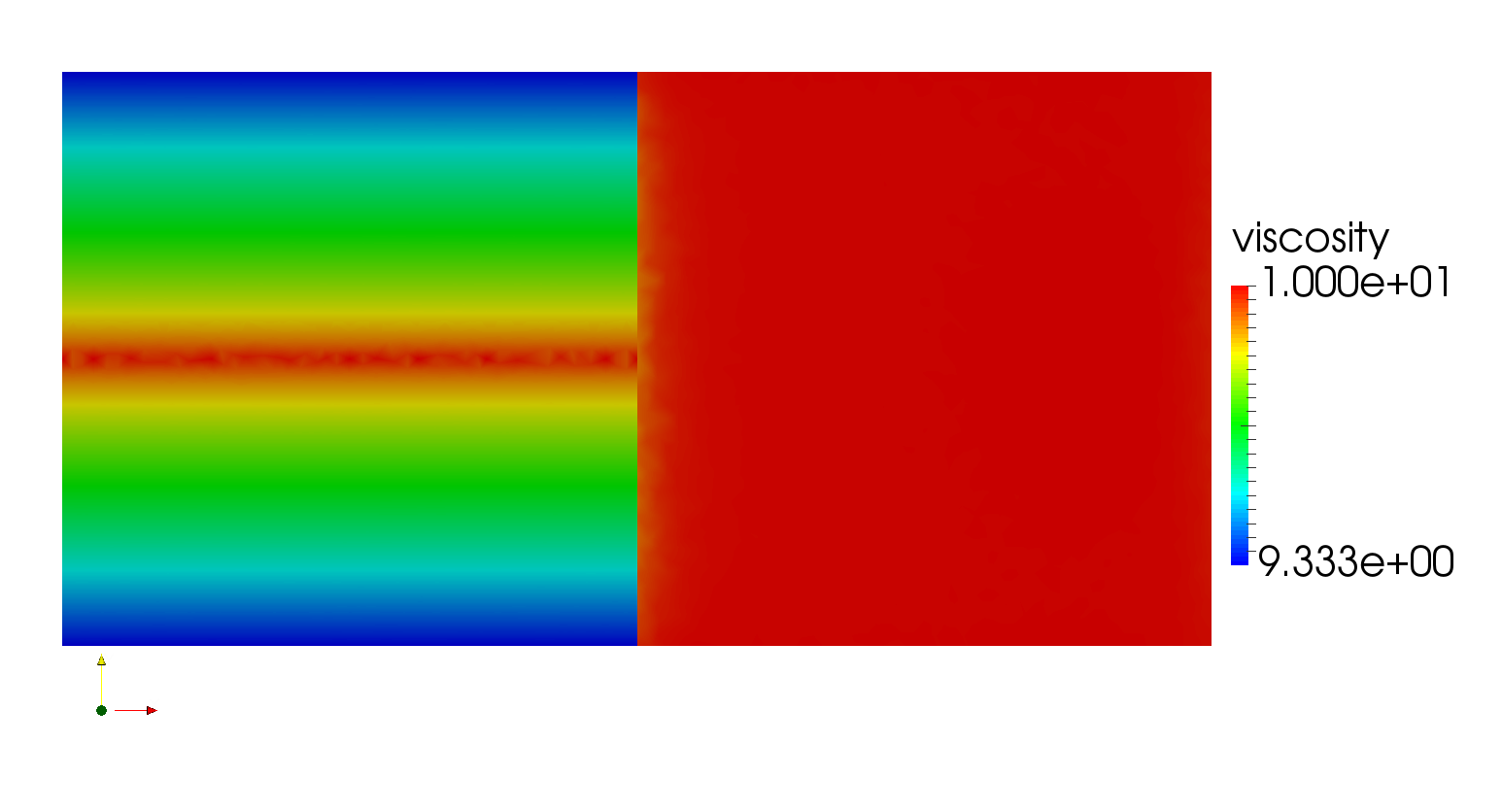}
		\caption{$t=0.01$}
	\end{subfigure}%
	~ 
	\begin{subfigure}{0.45\textwidth}
		\includegraphics[width=\textwidth]{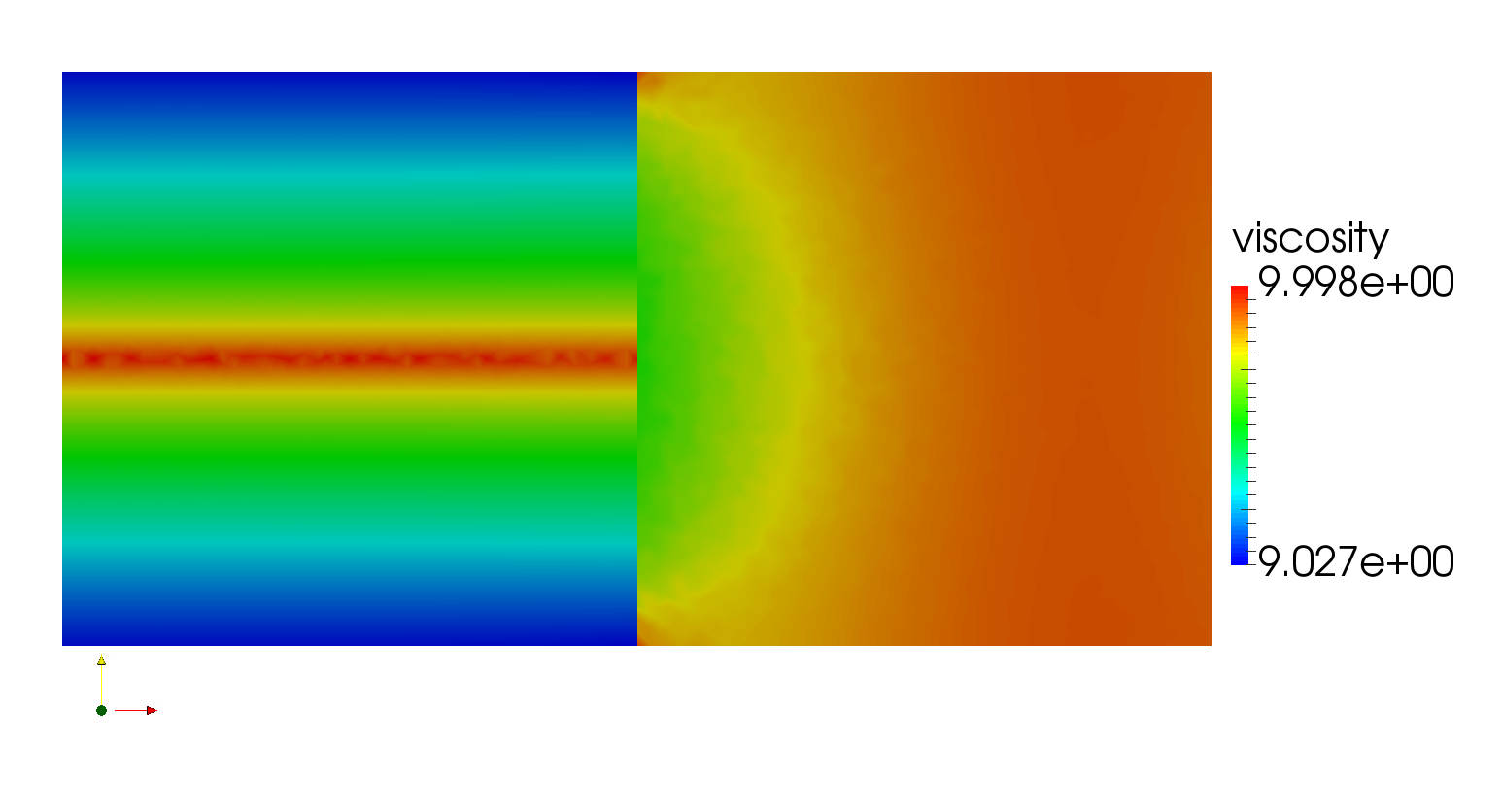}
		\caption{$t=1$}
	\end{subfigure}
	\caption{Example 1, nonlinear viscosity.}\label{fig:2}
\end{figure}

\begin{figure}
	\centering
	\begin{subfigure}{0.45\textwidth}
		\includegraphics[width=\textwidth]{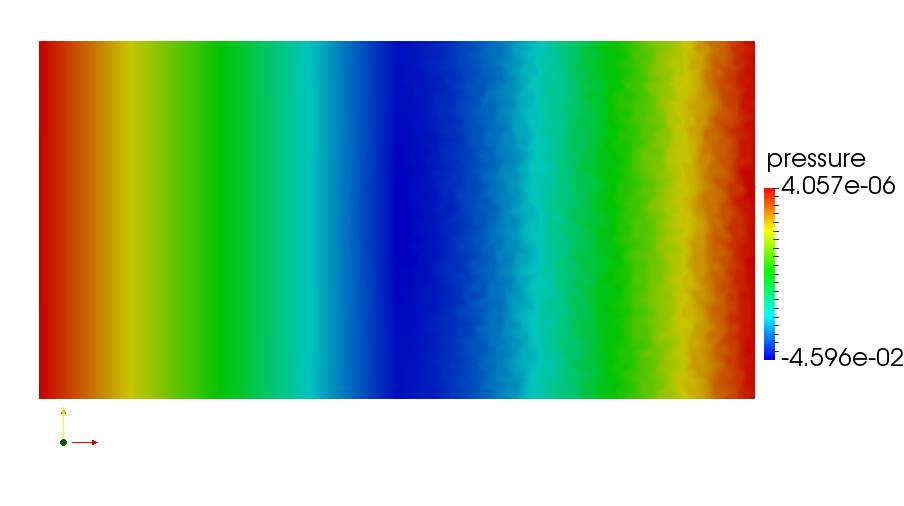}
		\caption{pressure}
	\end{subfigure}
	~ 
	\begin{subfigure}{0.45\textwidth}
		\includegraphics[width=\textwidth]{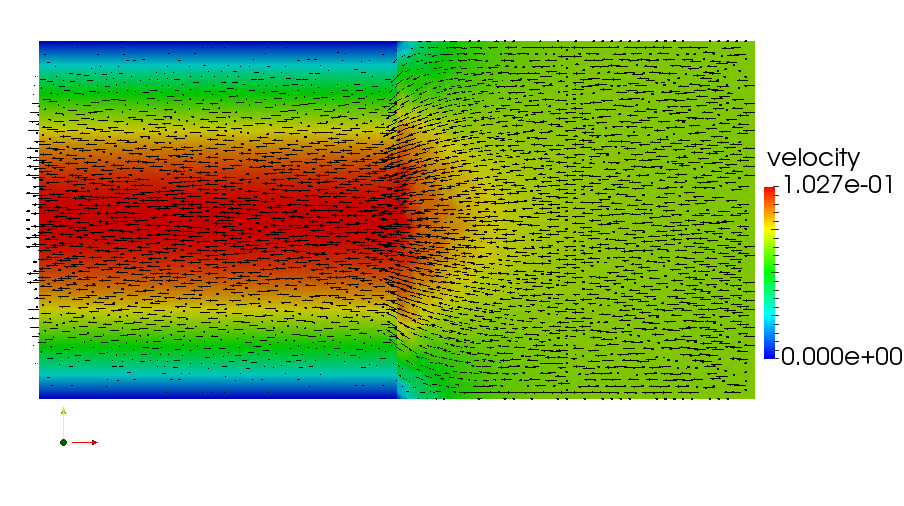}
		\caption{velocity vector (arrows) and magnitude (color)}
	\end{subfigure}
	\caption{Example 1, difference between non-Newtonian and Newtonian solutions at time $t=1$.}\label{fig:3}
\end{figure}

\begin{figure}
	\centering
	\begin{subfigure}{0.45\textwidth}
		\includegraphics[width=\textwidth]{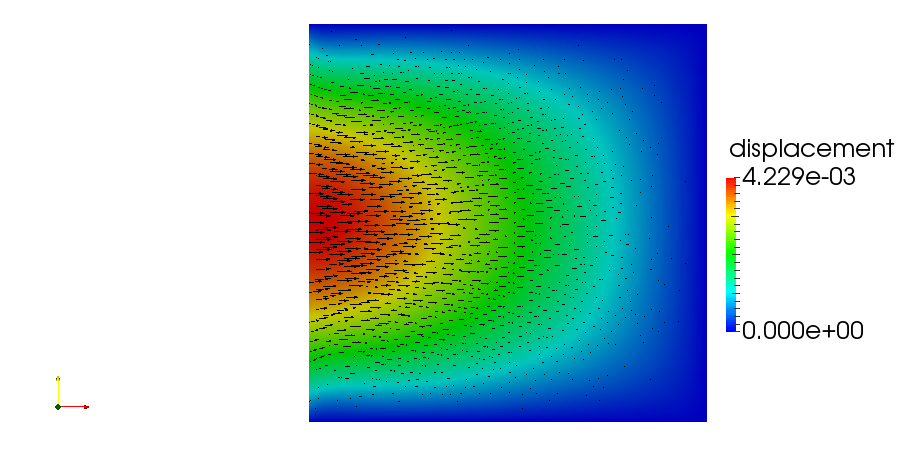}
		\caption{nonlinear displacement}
	\end{subfigure}%
	~ 
	\begin{subfigure}{0.45\textwidth}
		\includegraphics[width=\textwidth]{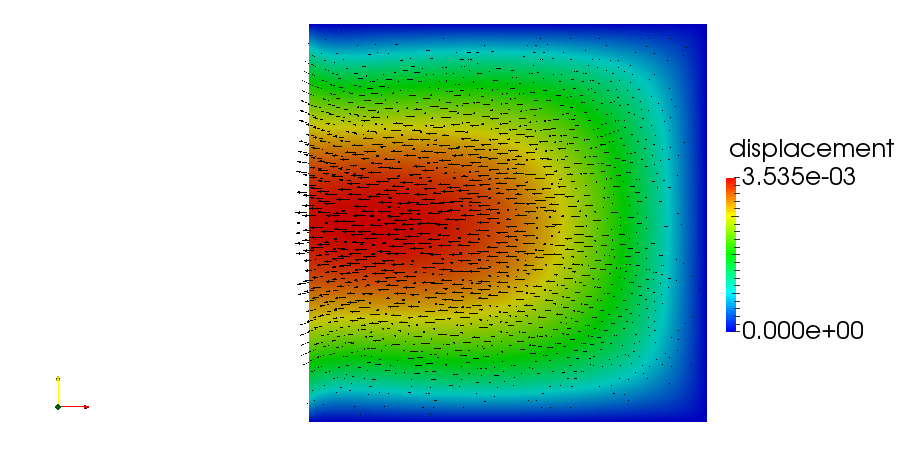}
		\caption{difference}
	\end{subfigure}
	\caption{Example 1, non-Newtonian displacement solution and difference at time $t=1$. }\label{fig:4}
\end{figure}

\subsection{Example 2: application to hydraulic fracturing}

We next present an example motivated by hydraulic fracturing. We study
the interaction between a stationary fracture filled with fluid and
the surrounding reservoir. The units in this example are meters for
length, seconds for time, and kPa for pressure. We consider a reference domain 
$\hat{\Omega} = [0, 1] \times [-1,1]$ and a fracture domain $\hat{\Omega}_f$, 
which is located in the middle with a boundary 
$$
\hat{x}^2 = 200(0.05 - \hat{y})(0.05 + \hat{y}), \quad \hat{y} \in [-0.05,0.05].
$$
The reference poroelastic domain is 
$\hat{\Omega}_p = \hat{\Omega}\setminus \hat{\Omega}_f$. The computational 
domain, shown in Figure~\ref{frac-figure} (left), is obtained from the reference 
domain via the mapping
\begin{equation*}
\begin{bmatrix} x \\ y \end{bmatrix} (\hat{x},\hat{y}) = \begin{bmatrix} x \\ (5\cos(\frac{\hat{x}+\hat{y}}{100})\cos(\frac{\pi \hat{x}+\hat{y}}{100})^2+\hat{y}/2-\hat{x}/10) \end{bmatrix}.
\end{equation*}
We enforce an inflow rate $\u_f\cdot \n_f=10$ m/s, $\u_f\cdot\btau_f=0$ m/s
on the left part of $\partial\O_f$ and no flow $\u_p\cdot\n_p = 0$ m/s and
no stress $\bs_p\n_p={\bf 0}$ kPa on the left part of $\partial\O_p$. On
the top, bottom, and right boundaries we set $p_p=1000$ kPa,
$\bbeta_p\cdot\n_p = 0$ m/s, and $\bs_p\n_p\cdot\btau_p=0$ kPa. The initial
conditions are $p_p=1000$ kPa and $\bbeta = {\bf 0}$ m/s. 
The poroelastic parameters, which are typical for hydraulic fracturing and
are similar to the ones used in \cite{Girault-2015}, are given in
Figure~\ref{frac-figure} (right). The nonlinear viscosity model in Stokes and Darcy is
from \cite{lopez2003predictive}
for a polymer used in hydraulic fracturing, see Figure~\ref{fig:viscosity} (left)
for the viscosity dependence on the shear rate. We match the curve using the Cross model
\eqref{cross} 
with parameters $K_f=K_p=7$, $\nu_{f,\infty}=\nu_{p,\infty}=3.0\times 10^{-6}$ kPa s,
$\nu_{f,0}=\nu_{p,0}=1.0\times 10^{-2}$ kPa s, and $r_f=r_p=1.35$.

We run the simulation for 300 s with time step $\tau = 1$ s and
compare the results of the linear and nonlinear models. For the linear
model we use the viscosity for water, $\nu_f^{lin} = \nu_p^{lin} =
1.0\times 10^{-6}$ kPa s, which is slightly lower than the minimum
value of the nonlinear viscosity. We present the simulation results at
the final time for both models in
Figures~\ref{fig:viscosity}--\ref{fig:darcy_pres_disp}. We note that
the scales in the plots are different for the two models, due to
significant differences in the solution values. The computed Stokes
and Darcy velocities are shown in
Figure~\ref{fig:stokes-darcy-vel}. We observe channel-like flow in the
fracture with both models. However, the higher nonlinear viscosity
results in smaller velocity, especially near the fracture tip. The
nonlinear viscosity in the fracture is shown in
Figure~\ref{fig:viscosity} (middle). We note the significant
shear-thinning effect, especially along the wall of the fracture,
where the viscosity is reduced to values in the order of
$\nu_{f,\infty}$. Comparing the Darcy velocity fields in
Figure~\ref{fig:stokes-darcy-vel}, we observe that the combination of
very small permeability and high fluid viscosity in the nonlinear case
results in very little fluid penetration into the reservoir. This is
an expected behavior in hydraulic fracturing. Correspondingly, the
nonlinear viscosity in the poroelastic region, as shown in
Figure~\ref{fig:viscosity} (right), is significantly reduced in a
close vicinity of the fracture, but is equal to the maximum value
$\nu_{p,0}$ away from the fracture.  In the linear case, the Darcy
velocity is larger and the fluid penetrates further into the
reservoir. The behavior for both models is consistent with the
computed pressure fields shown in
Figure~\ref{fig:darcy_pres_disp}. For both models we observe increase
in pressure near the fracture. In the linear case the pressure
gradient extends away from the fracture. In the nonlinear case, since
the fluid cannot penetrate further into the reservoir, we observe a
significant pressure buildup along the fracture, about 100 times
larger than in the linear case. This in turn results in about 100
times larger displacement in the nonlinear case. This includes larger
opening of the fracture, all the way to the tip. We note that our
models are for stationary fractures, but the large displacement and
corresponding stress near the fracture tip in the nonlinear case may
result in practice in fracture propagation, as would be expected in
hydraulic fracturing. To summarize, this is a numerically very
challenging test case, due to the large stiffness and small
permeability of the rock. The numerical difficulty for the
non-Newtonian fluid is further increased due to the model nonlinearity
and the larger viscosity. We observe that the model is capable of
handling parameters in this challenging range and produce results that
are qualitatively similar to practical hydraulic fracturing
applications.

\begin{figure}
\centering
\begin{minipage}{0.3\textwidth}
\includegraphics[width=0.9\textwidth]{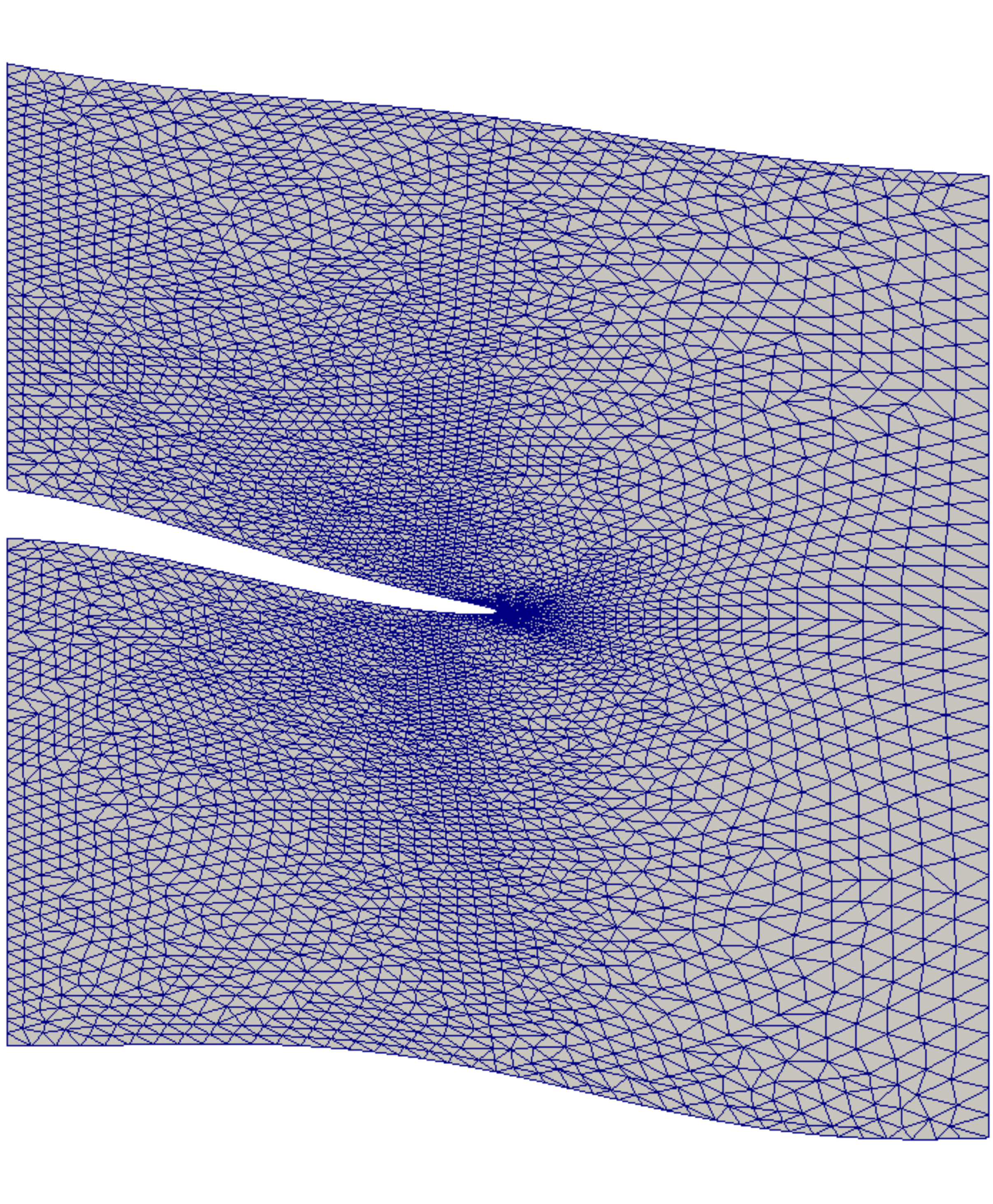}
\end{minipage}
\begin{minipage}{0.6\textwidth}
\begingroup
		\begin{tabular}{|l l l l |}
			\hline
			Parameter                          &      & Units         & Values                        \\ \hline\hline
			Young's modulus                    & $E$         & (kPa)         & $10^7$                        \\
			Poisson's ratio                    & $\sigma$         &          & $0.2$                        \\
			Lame coefficient                   & $\mu_p$     & (kPa)         & $5/12\times 10^{7}$ \\
			Lame coefficient                   & $\lambda_p$ & (kPa)         & $5/18\times 10^{7}$ \\
			Permeability                       & $K$         & (m$^2$)       & $(200,50) \times 10^{-12}$ \hspace*{-2mm}\\
			Mass storativity                   & $s_0$       & (kPa$^{-1}$)  & $6.89\times 10^{-2}$          \\
			Biot-Willis const.               & $\alpha$    &               & 1.0                             \\
			BJS coeff. & $\alpha_{BJS}$     &        & 1.0 \\
			Total time                         & T           & (s)           & 300                           \\ \hline
		\end{tabular}
\endgroup
\end{minipage}
\caption{Computational domain (left) and parameters (right) for Example 2.}
\label{frac-figure}
\end{figure}

\begin{figure}
	\centering
	\begin{subfigure}[H]{0.30\textwidth}
		\includegraphics[width=\textwidth]{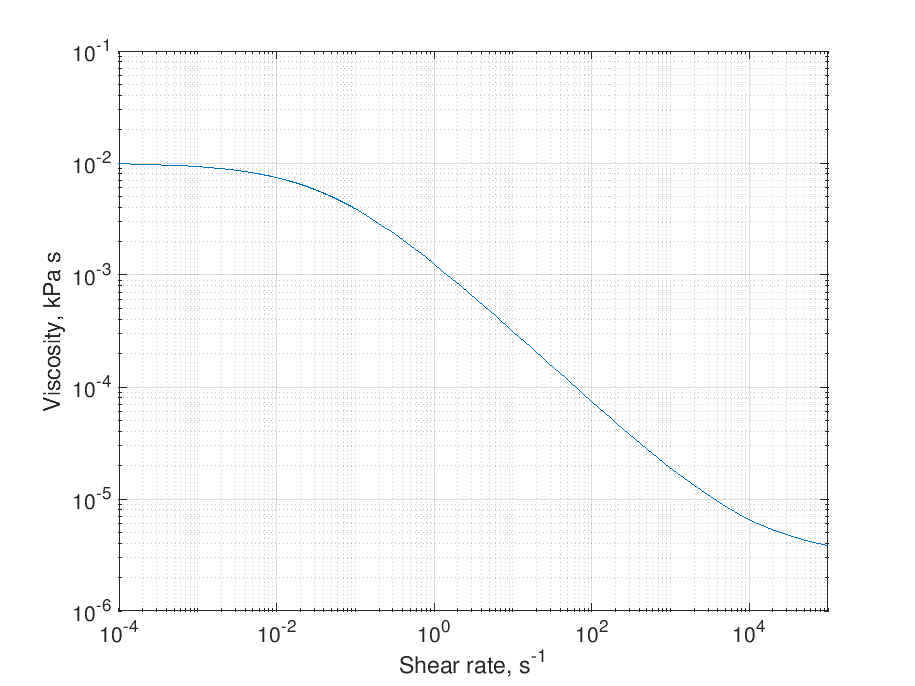}
		\caption{Viscosity model}
	\end{subfigure}
        \begin{subfigure}[H]{0.30\textwidth}
		\includegraphics[width=\textwidth]{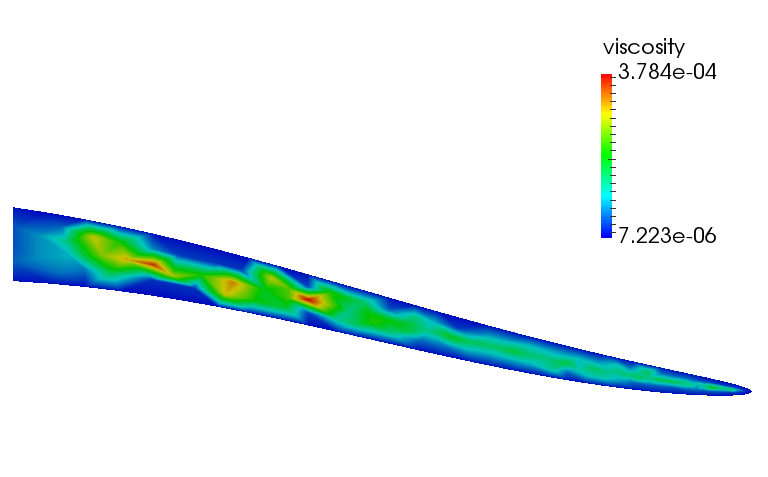}
		\caption{Stokes viscosity}
	\end{subfigure}
        \begin{subfigure}[H]{0.30\textwidth}
		\includegraphics[width=\textwidth]{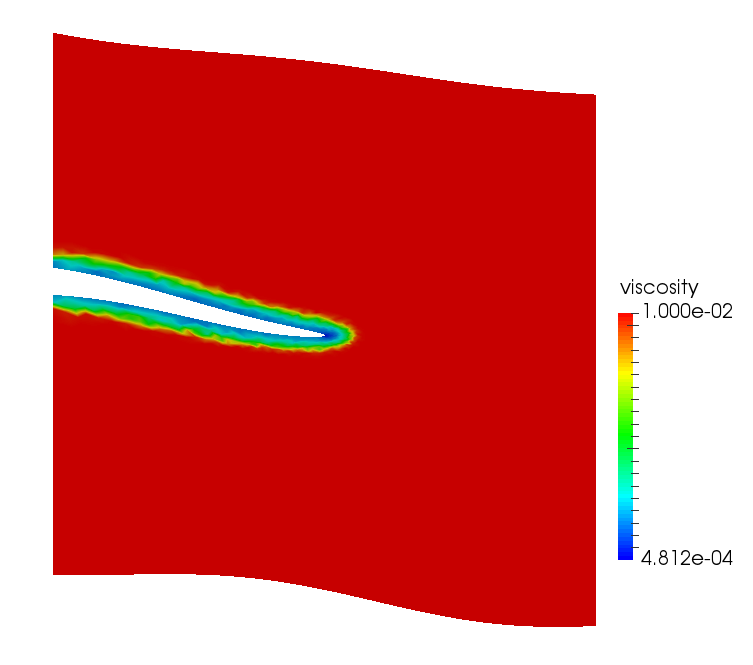}
		\caption{Darcy viscosity}
	\end{subfigure}
        \caption{Example 2, nonlinear viscosity model and computed
          Stokes and Darcy viscosity at $t=300$s.}
        \label{fig:viscosity}
\end{figure}

\begin{figure}
	\centering
	\begin{subfigure}[H]{0.24\textwidth}
		\includegraphics[width=\textwidth]{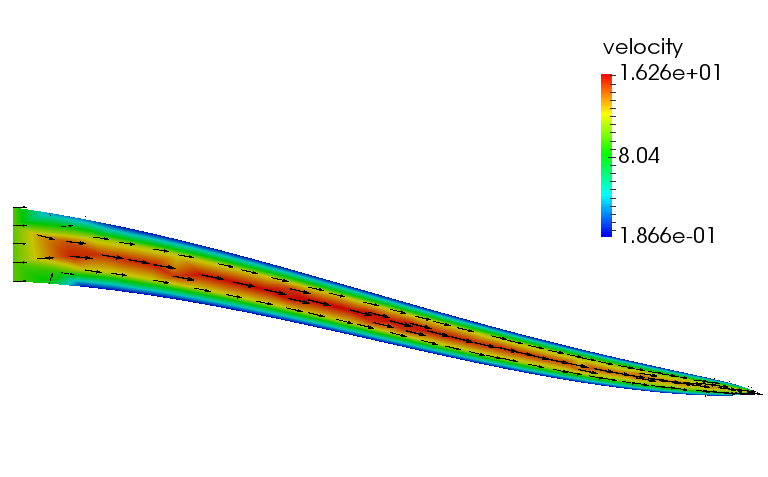}
		\caption{Stokes, linear}
	\end{subfigure}%
	~ 
	\begin{subfigure}[H]{0.24\textwidth}
		\includegraphics[width=\textwidth]{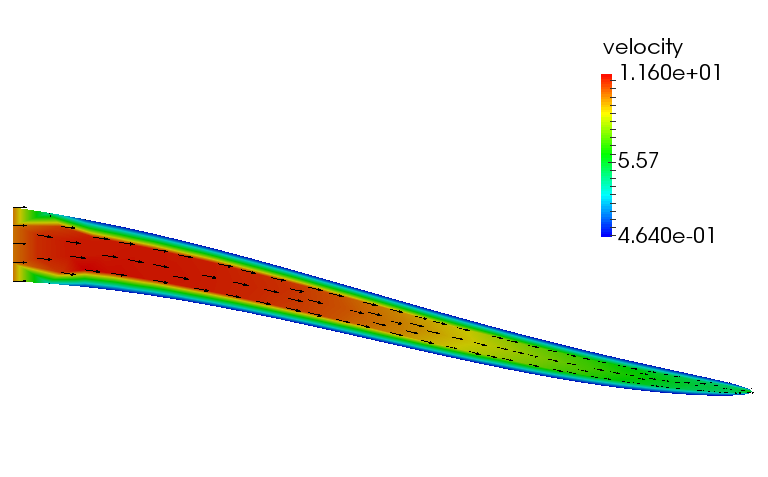}
		\caption{Stokes, nonlinear}
	\end{subfigure}
        	\centering
	\begin{subfigure}[H]{0.24\textwidth}
		\includegraphics[width=\textwidth]{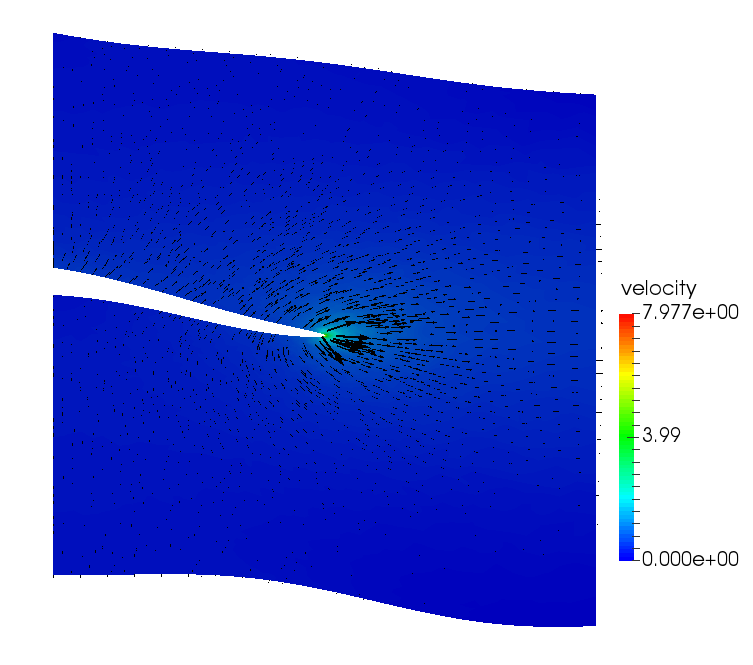}
		\caption{Darcy, linear}
	\end{subfigure}%
	~ 
	\begin{subfigure}[H]{0.24\textwidth}
		\includegraphics[width=\textwidth]{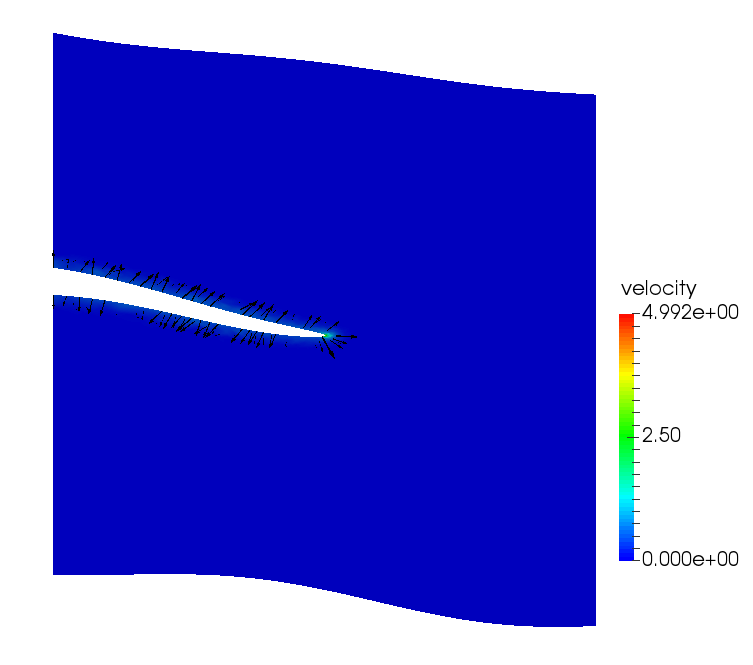}
		\caption{Darcy, nonlinear}
	\end{subfigure}
	\caption{Example 2, Stokes and Darcy velocity at time $t=300$s.}
\label{fig:stokes-darcy-vel}
\end{figure}

\begin{figure}
	\begin{subfigure}[H]{0.24\textwidth}
		\includegraphics[width=\textwidth]{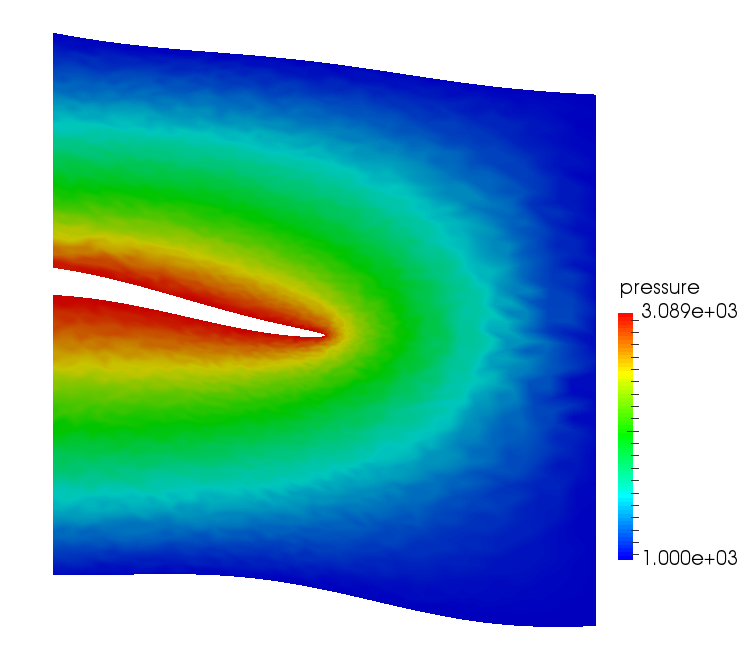}
		\caption{pressure, linear}
		\label{fig:darcy_pres_3}
	\end{subfigure}%
	~ 
	\begin{subfigure}[H]{0.24\textwidth}
		\includegraphics[width=\textwidth]{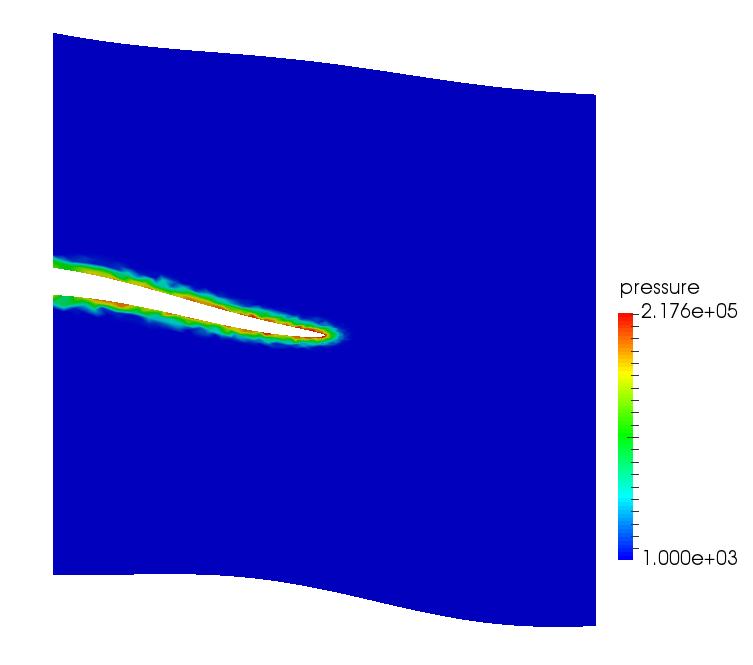}
		\caption{pressure, nonlinear}
		\label{fig:darcy_pres_4}
	\end{subfigure}
        \begin{subfigure}[H]{0.24\textwidth}
		\includegraphics[width=\textwidth]{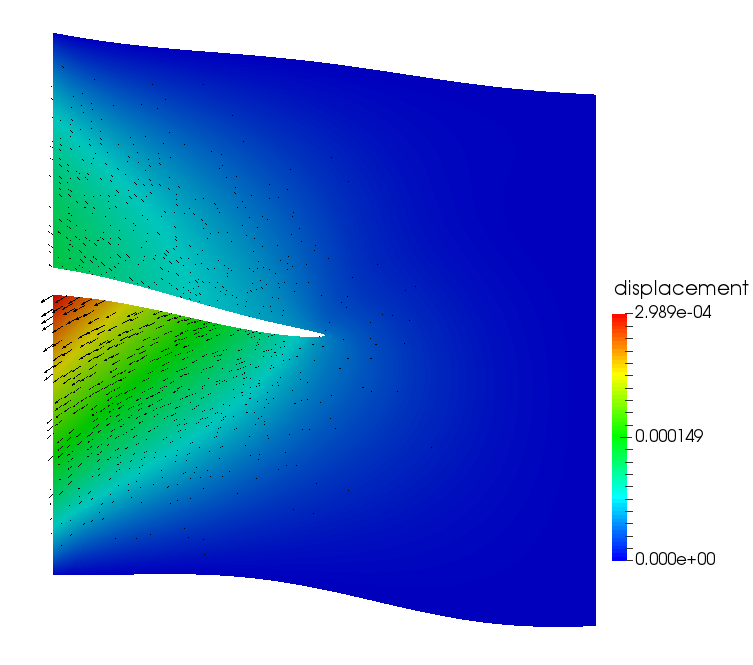}
		\caption{displacement, linear}
		\label{fig:disp_3}
	\end{subfigure}%
	~ 
	\begin{subfigure}[H]{0.24\textwidth}
		\includegraphics[width=\textwidth]{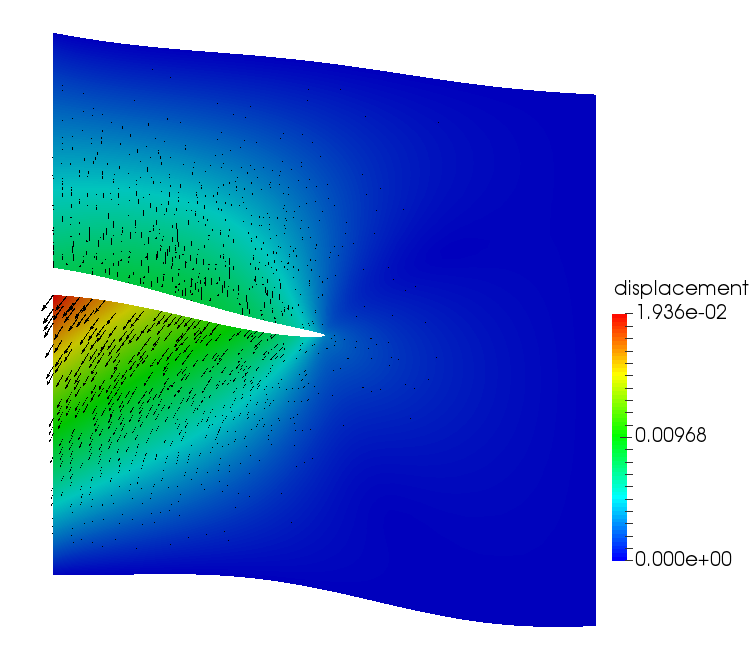}
		\caption{displacement, nonlinear}
		\label{fig:disp_4}
	\end{subfigure}        
	\caption{Example 2, Poroelastic pressure and displacement at time $t=300$s.}
        \label{fig:darcy_pres_disp}
\end{figure}

\bibliographystyle{abbrv}
\bibliography{nonlinear-stokes-biot}

\end{document}